\documentclass[DIV=14,letterpaper]{scrartcl}

\date{}

\usepackage{tikz}
\usetikzlibrary{calc}
\usetikzlibrary{matrix}
\usepackage{amsmath,amssymb,amsfonts,amsthm,subfig}
\usepackage{float}
\usepackage{multirow}
\usepackage{tikz}
\usepackage{graphicx}
\usetikzlibrary{positioning}
\usepackage[pdftex]{hyperref} 
\usepackage{relsize}
\usepackage{graphicx}
\usepackage{chngcntr}
\usepackage[toc,page]{appendix}
\counterwithin{table}{section}

\newcommand{\V}[1]{\mbox{\boldmath $ #1 $}}

\newcommand{\tr}[1]{\text{tr} #1}

\newcommand{\M}[1]{\mathbb{M} #1}
\newcommand{\dFMF}[1]{\det\left( \left(F_K'\right)^T\mathbb{M}_K F_K'\right) #1}
\newcommand{\FMF}[1]{\left( \left(F_K'\right)^T\mathbb{M}_K F_K'\right) #1}
\newcommand{\T}[1]{\mathcal{T} #1}

\theoremstyle{definition}

\def \J{\mathbb{J}}
\theoremstyle{definition}

\newcommand{\bey}{\begin{eqnarray}}
\newcommand{\eey}{\end{eqnarray}}

\newcommand{\beq}{\begin{equation}}
\newcommand{\eeq}{\end{equation}}
\theoremstyle{plain}
\newtheorem{thm}{\hspace{6mm}Theorem}[section]

\newtheorem{lem}{\hspace{6mm}Lemma}[section]

\theoremstyle{definition}

\theoremstyle{remark}
\newtheorem{exam}{\hspace{6mm}Example}[section]

\title{A surface moving mesh method based on equidistribution and alignment}

\author{Avary~Kolasinski%
\thanks{Department of Mathematics, the University of Kansas, Lawrence, KS 66045
({\em avaryk@ku.edu}).}
\and Weizhang~Huang%
\thanks{Department of Mathematics, the University of Kansas, Lawrence, KS 66045
({\em whuang@ku.edu}).}
}

\begin{document}
\vskip 1cm
\maketitle

\begin{abstract}

A surface moving mesh method is presented for general surfaces with or without explicit parameterization. 
The method can be viewed as a nontrivial extension of the moving mesh partial differential equation method that has been developed for bulk meshes and demonstrated to work well for various applications. 
The main challenges in the development of surface mesh movement come from the fact that the Jacobian matrix of the affine mapping between the reference element and any simplicial surface element is not square. 
The development starts with revealing the relation between the area of a surface element in the Euclidean or Riemannian metric and the Jacobian matrix of the corresponding affine mapping, formulating the equidistribution and alignment conditions for surface meshes, and establishing a meshing energy function based on the conditions. 
The moving mesh equation is then defined as the gradient system of the energy function, with the nodal mesh velocities being projected onto the underlying surface. 
The analytical expression for the mesh velocities is obtained in a compact, matrix form, which makes the implementation of the new method on a computer relatively easy and robust.
Moreover, it is analytically shown that any mesh trajectory generated by the method remains nonsingular if it is so initially. It is emphasized that the method is developed directly on surface meshes, making no use of any information on surface parameterization. It utilizes surface normal vectors to ensure that the mesh vertices remain on the surface while moving, and also assumes that the initial surface mesh is given. The new method can apply to general surfaces with or without explicit parameterization since the surface normal vectors can be computed even when the surface only has a numerical representation. A selection of two- and three-dimensional examples are presented.
\end{abstract}



\noindent
\textbf{Abbreviated title.} A surface moving mesh method.

\section{Introduction}

We are interested in methods that can directly move simplicial meshes on general surfaces with or without analytical expressions. Such surface moving mesh methods can be used for adaptation and/or quality improvements of surface meshes and thus are useful for computational geometry and numerical solutions of partial differential equations (PDEs)
defined on surfaces; e.g., see \cite{DD, DE2, TM}.

There has been some work done on mesh movement and adaptation for surfaces.
For example, Crestel et al. \cite{CRR} present a moving mesh method for parametric surfaces
by generalizing Winslow's meshing functional to Riemannian manifolds and
taking into consideration the Riemannian metric associated with the manifolds.
The method is simplified and implemented on a two-dimensional domain 
for surfaces that accept certain parameterizations.
Weller et al. \cite{BBCW} and McRae et al. \cite{MCB2018}
solve a Monge-Amp\' ere type equation on the surface of the sphere to generate optimally transported meshes that become equidistributed with respect to a suitable monitor function. 
MacDonald~et~al.~\cite{MMNI} devise a moving mesh method for the numerical simulation of coupled bulk-surface reaction-diffusion equations on an evolving two-dimensional domain.  They use a one-dimensional moving mesh equation in arclength to concentrate mesh points along the evolving domain boundary. 
Dassi et al. \cite{DPSS} generalize the higher embedding approach proposed in \cite{BB}.
They modify the embedding map between the underlying surface and $\mathbb{R}^6$
to include more information associated with the physical solution and its gradient.
The idea behind this mapping is that it essentially approximates the geodesic length
on the surface via a Euclidean length in $\mathbb{R}^6$. The mesh adapts in the Euclidean space
and then is mapped back to the physical domain.

The objective of this paper is to present a surface moving mesh method for general surfaces with or without explicit parameterization. The method can be viewed as a nontrivial extension of the moving mesh PDE (MMPDE) method
that has been developed for bulk meshes and demonstrated to work well for various applications;
e.g. see \cite{HR99,HR,HRR}. The main challenges in the development of surface mesh movement
come from the fact that the Jacobian matrix of the affine mapping between the reference element and
any simplicial surface element is not square. 
To overcome these challenges, we start by connecting the area of the surface element in the Euclidean metric or
a Riemannian metric with the Jacobian matrix. This connection allows us to formulate the equidistribution
and alignment conditions and ultimately, form a meshing energy function for surface meshes. 
This meshing function is similar to a discrete version of Huang's functional \cite{H,HKR,HK}
for bulk meshes which has been proven to work well in a variety of problems.
Following the MMPDE approach, we define the surface moving mesh equation as the gradient system
of the meshing function, with the nodal mesh velocities being projected onto the underlying surface.
The analytical expression for the mesh velocities is obtained in a compact, matrix form, which makes
the implementation of the new method on a computer relatively easy and robust.
Several theoretical  properties are obtained for the surface moving mesh method.
In particular, it is proven that a surface mesh generated by the method remains nonsingular for all time if it
is so initially. Moreover, the element altitudes and areas of the physical mesh are bounded below by positive constants depending only on the initial mesh, the number of elements, and the metric tensor that is used to provide information
on the size, shape, and orientation of the elements throughout the surface.
Furthermore, limiting meshes exist and the meshing function is decreasing along each mesh trajectory.
These properties are verified in numerical examples. 

It is emphasized that the new method is developed directly on surface meshes, making no use of any information
on surface parameterization. It utilizes surface normal vectors to ensure that the mesh vertices remain on the surface
while moving, and also assumes that the initial surface mesh is given. Since the surface normal vectors can be computed even when the surface only has a numerical representation, the new method can apply to general surfaces
with or without explicit parameterization. A selection of two- and three-dimensional examples are presented.

This paper is organized as follows.  In Section~\ref{SEC:functions}, the area formula for a surface element and the equidistribution and alignment conditions for surface meshes are established. The surface moving mesh equation is described in Section~\ref{SEC:MMPDE} and its theoretical analysis is given in Section~\ref{SEC:theor}. Numerical examples are then provided in Section~\ref{SEC:numerical} followed by conclusions and further remarks in Section~\ref{SEC:conclusion}.
Appendix~\ref{SEC:discrederiv} contains the derivation of derivatives of the meshing function
with respect to the physical coordinates.

\section{Equidistribution and alignment for surface meshes}
\label{SEC:functions}

In this section we formulate the equidistribution and alignment conditions for a surface mesh. These conditions are used to characterize the size, shape, and orientation of the elements and develop a meshing function for surface mesh generation and adaptation. The function is similar to the one \cite{H,HK2} used in bulk mesh generation and adaptation
and also based on  mesh equidistribution and alignment.

\subsection{Area and affine mappings for surface elements}
	
Let $S$ be a bounded surface in $\mathbb{R}^d$ ($d\ge 2$). Assume that we have a mesh $\mathcal{T}_h= \{K\}$ on $S$ and let $N$ and $N_v$ be the number of its elements and vertices, respectively.  The elements are surface simplexes in $\mathbb{R}^d$, i.e., they are $(d-1)$-dimensional simplexes in a $d$-dimensional space.
Notice that their area in $d$ dimensions is equivalent to their volume in $(d-1)$ dimensions.
Assume that the reference element $\hat{K}$ has been chosen to be a  $(d-1)$-dimensional equilateral and unitary simplex in a $(d-1)$-dimensional space.  For $\hat{K}$ and any element $K\in\mathcal{T}_h$ let $F_K:\hat{K}
\subset \mathbb{R}^{d-1} \to K \subset \mathbb{R}^{d}$ be the affine mapping between them
and $F_K'$ be its Jacobian matrix. 
Denote the vertices of $K$ by $\V x_j^K\in\mathbb{R}^d$, $j=1,\dots,d$ and the vertices of $\hat{K}$ by $\V \xi_j\in \mathbb{R}^{d-1}$, $j=1, \dots, d$. Then 
$$\mbox \qquad\qquad \V x_j^K = F_K(\V \xi_j), \quad j = 1, \dots, d.$$
From this, we have
$$\mbox \qquad\qquad\V x_j^K-\V x_1^K = F_K'\left(\V\xi_j - \V \xi_1\right), \quad j=2, \dots, d$$
or
$$\left[\V x_2^K-\V x_1^K, \dots, \V x_d^K-\V x_1^K\right]= F_K'\left[\V\xi_2 - \V \xi_1, \dots, \V\xi_d - \V \xi_1\right],$$
which gives $F_K' = E_K\hat{E}^{-1},$ where $E_K$ and $\hat{E}$ are the edge matrices for $K$ and $\hat{K}$, i.e.,
$$E_K = \left[\V x_2^K-\V x_1^K, \dots, \V x_d^K-\V x_1^K\right], \quad\hat{E} = \left[\V\xi_2 - \V \xi_1, \dots, \V\xi_d - \V \xi_1\right].$$
Notice that $\hat{E}$ is a $(d-1)\times (d-1)$ square matrix and its inverse exists since $\hat{K}$ is not degenerate.
However, unlike the bulk mesh case, matrices $E_K, F_K'\in \mathbb{R}^{d\times (d-1)}$ are not square. This makes the formulation of adaptive mesh methods more difficult for surface than bulk meshes. Nevertheless, the approach is similar for both situations, as will be seen below.

In the following we can see that the area of the physical element $K\in \mathcal{T}_h$ can be determined
using $F_K'$ or $E_K$.
\begin{lem}
\label{volK}
For any surface simplex $K$, there holds
\beq
\frac{|K|}{|\hat{K}|} =\det\left(\left(F_K'\right)^TF_K'\right)^{1/2},
\label{area-1}
\eeq
where $|K|$ and $|\hat{K}|$ denote the area of the simplexes $K$ and $\hat{K}$, respectively,
and $\det (\cdot )$ denotes the determinant of a matrix.
\end{lem}
\vspace{0.1cm}

\begin{proof}
From $F_K' = E_K \hat{E}^{-1}$, we have
\begin{align*}
\det\left(\left(F_K'\right)^TF_K'\right)^{1/2}&=\det\left(\hat{E}^{-T}E_K^TE_K\hat{E}^{-1}\right)^{1/2}\\
&=\det(\hat{E})^{-1}\det\left(E_K^TE_K\right)^{1/2}\\
&=\frac{1}{(d-1)!\, |\hat{K}|}\det\left(E_K^TE_K\right)^{1/2}, 
\end{align*}
where we have used $|\hat{K}|=\frac{1}{(d-1)!}\det(\hat{E})$. Let the QR-decomposition of $E_K\in\mathbb{R}^{d\times (d-1)}$ be given by
$$
E_K=Q_K\left[\begin{matrix} R_K\\ \V 0\end{matrix}\right], 
$$
where $Q_K \in \mathbb{R}^{d\times d}$ is a unitary matrix, $R_K\in\mathbb{R}^{(d-1)\times(d-1)}$ is an upper triangular matrix, and $\V 0\in \mathbb{R}^{1\times (d-1)}$ is a row vector of zeros. This decomposition indicates that $K$ is formed by rotating the convex hull with edges formed by the column vectors of {\small{$\left[\begin{matrix} R_K\\ \V 0\end{matrix}\right]$}}. We have
$$
|K|=\text{area}(E_K)=\text{area}\left(Q_K\left[\begin{matrix} R_K\\ \V 0\end{matrix}\right]\right)=\text{area}\left(\left[\begin{matrix} R_K\\ \V 0\end{matrix}\right]\right),
$$
where we have used the fact that rotation, $Q_K$, does not change the area.
Since the convex hull formed by the column vectors of {\small{$\left[\begin{matrix} R_K\\ \V 0\end{matrix}\right]$}} lies on the $\V x^{(1)}$ -- $\cdots$ -- $\V x^{(d-1)}$ -- plane, its area is equal to the $(d-1)$-dimensional volume of the convex hull formed by the column vectors of $R_K$ in $(d-1)$-dimensions. Then,
$$
|K|=\text{volume}_{(d-1)}(R_K)=\frac{1}{(d-1)!}\det(R_K)=\frac{1}{(d-1)!}\det(R_K^TR_K)^{1/2} .
$$
On the other hand, we have
$$\det\left(E_K^TE_K\right)=\det\left(\left[\begin{matrix} R_K\\ \V 0\end{matrix}\right]^TQ_K^TQ_K\left[\begin{matrix} R_K\\ \V0\end{matrix}\right]\right)=\det\left(\left[\begin{matrix} R_K^T ~~ \V0\end{matrix}\right]\left[\begin{matrix} R_K\\ \V0\end{matrix}\right]\right)=\det\left(R_K^TR_K\right).$$
Therefore,
$$
\det\left(\left(F_K'\right)^TF_K'\right)^{1/2}
=\frac{1}{(d-1)!\, |\hat{K}|}\det\left(E_K^TE_K\right)^{1/2}
= \frac{1}{(d-1)!\, |\hat{K}|}\det\left(R_K^T R_K\right)^{1/2}
= \frac{|K|}{|\hat{K}|} .
$$
\end{proof}

\subsection{Area of surface elements in a Riemannian metric}
We now formulate the area of a surface element in a Riemannian metric using $F_K'$ or $E_K$.
The formula is needed later in the development of algorithms for mesh adaptation. First we consider a symmetric, uniformly positive definite metric tensor $\mathbb{M}(\V x)$ which satisfies
\beq
\label{Mbound}
\underline{m}I\le\mathbb{M}(\V x) \le \overline{m}I, \quad \forall \V x\in S
\eeq
where $\underline{m}$ and $\overline{m}$ are positive constants, $I$ is the identity matrix, and
the less-than-or-equal-to sign is in the sense of negative semi-definiteness.  We define
the average of $\mathbb{M}$ over $K$ as
$$
\mathbb{M}_K=\frac{1}{|K|}\int_{K}\mathbb{M}(\V x)d\V x.
$$
Recall that the distance in the Riemannian metric, $\mathbb{M}_K$, is defined as
\beq
\label{Mnorm}
\|\V x\|_{\mathbb{M}_K}=\sqrt{\V x^T\mathbb{M}_K\V x}=\sqrt{\left(\mathbb{M}_K^{1/2}\V x\right)^T\left(\mathbb{M}_K^{1/2}\V x\right)}=\left\|\mathbb{M}_K^{1/2}\V x\right\|,
\eeq
where $\| \cdot \|$ denotes the standard Euclidean norm.  This implies that the geometric properties of $K$ in the metric $\mathbb{M}_K$ can be obtained from those of $\mathbb{M}_K^{1/2}K$ in the Euclidean metric. 
\begin{lem}
\label{volKM}
For any surface simplex $K$, there holds
\beq
\frac{|K|_{\mathbb{M}_K}}{|\hat{K}|}=\dFMF^{1/2},
\label{thm-|K|-2}
\eeq
where $|K|_{\mathbb{M}_K}$ denotes the area of $K$ in the metric $\mathbb{M}_K$.
\end{lem}
\begin{proof}
The Jacobian matrix of the affine mapping from $\hat{K}$ to $\mathbb{M}_K^{1/2}K$ is given by
\[
F_{\mathbb{M}_K,K}'=\left(\mathbb{M}_K^{1/2}E_K\right) \hat{E}^{-1}=\mathbb{M}_K^{1/2} F_K' .
\]
From the discussion following (\ref{Mnorm}), we know that the area of $K$ in the metric $\mathbb{M}_K$
is equal to the area of $\mathbb{M}_K^{1/2}K$ in the Euclidean metric. Thus, from Lemma~\ref{volK} we have 
\begin{align*}
\frac{|K|_{\mathbb{M}_K}}{|\hat{K}|} & = \frac{|\mathbb{M}_K^{1/2}K|}{|\hat{K}|} =
\det\left(\left(F_{\mathbb{M}_K,K}'\right)^T F_{\mathbb{M}_K,K}' \right)^{1/2} \\
& = \det \left(\left(\mathbb{M}_K^{1/2} F_K'\right)^T \mathbb{M}_K^{1/2} F_K'\right)^{1/2}
= \dFMF^{1/2} .
\end{align*}
\end{proof}

The following lemma gives a lower bound for the area of $K$ with respect to the metric $\mathbb{M}_K$
in terms of the minimum altitude of $K$ with respect to $\mathbb{M}_K$.

\begin{lem}
\label{lem2}
Let $a_{K,\mathbb{M}_K}$ denote the minimum altitude of $K$ with respect to $\mathbb{M}_K$. Then,
\beq\label{kak} |K|_{\mathbb{M}_K} \ge \dfrac{1}{(d-1)^{\frac{d-1}{2}}(d-1)!}~a_{K,\mathbb{M}_K}^{d-1} .\eeq
\end{lem}

\begin{proof}
From Lemma \ref{volKM} and $F_K'=E_K\hat{E}^{-1}$, we have
\begin{align*}
|K|_{\mathbb{M}_K}&= |\hat{K}|\det\FMF^{1/2} \\
&= \dfrac{|\hat{K}|}{\det(\hat{E})}\det\left(E_K^T\mathbb{M}_KE_K\right)^{1/2} \\
&= \dfrac{1}{(d-1)!} \det\left(\left(\mathbb{M}_K^{1/2}E_K\right)^T\left(\mathbb{M}_K^{1/2}E_K\right)\right)^{1/2}.\\
\end{align*}
 Let the $QR$-decomposition of $\mathbb{M}_K^{1/2}E_K$ be denoted as
\[
\mathbb{M}_K^{1/2}E_K = Q_K\left[\begin{matrix} R_K\\ \V 0\end{matrix}\right],
\] 
where $Q_K\in \mathbb{R}^{d\times d}$ is a unitary matrix, $R_K\in\mathbb{R}^{d-1 \times d-1}$ is
an upper triangular matrix, and $\V0$ is a $(d-1)$-dimensional row vector of zeros. This gives
\begin{align*}
|K|_{\mathbb{M}_K} & = 
\dfrac{1}{(d-1)!} \det\left(\left(\mathbb{M}_K^{1/2}E_K\right)^T\left(\mathbb{M}_K^{1/2}E_K\right)\right)^{1/2}\\
 &= \dfrac{1}{(d-1)!} \det\left([R_K^T ~\V0^T]Q_K^T Q_K \left[\begin{matrix} R_K\\ \V0\end{matrix}\right]\right)^{1/2}\\
 & = \dfrac{1}{(d-1)!} \det(R_K^TR_K)^{1/2}\\
 & = \dfrac{1}{(d-1)!} \prod_{i=1}^{d-1}s_i,
 \end{align*}
 where $s_i$, $i = 1, \dots, d-1$ denote the singular values of $R_K$. Additionally, by  \cite[Lemma 5.12]{BCY} we have that
$$ s_i \ge \dfrac{a_{R_K}}{\sqrt{d-1}},$$
where $a_{R_K}$ denotes the minimum altitude of the simplex formed by the columns of $R_K$. Combining these, we get
\[
|K|_{\mathbb{M}_K} \ge \dfrac{1}{(d-1)^{\frac{d-1}{2}}(d-1)!} a_{R_{K}}^{d-1}.
\]
Since $Q_K$ is a rotational matrix, the minimum altitude of $K$ with respect to the metric $\mathbb{M}_K$ is the same as the minimum altitude of the convex hull formed by the columns of $R_K$ i.e., $a_{K,\mathbb{M}_K} = a_{R_K}$. Thus, we have obtained (\ref{kak}).
\end{proof}

The relationship given in the above lemma between the area and minimum height will be used in the proof
of the nonsingularity for surface meshes in \S\ref{SEC:theor}. It is instructional to note that in two dimensions ($d = 2$),
$K$ is a line segement and both $|K|_{\mathbb{M}_K}$ and $a_{K,\mathbb{M}_K}$ represent the length
of $K$ in the metric $\mathbb{M}_K$ and are equal. In this case, the inequality (\ref{kak}) reduces to 
\[
|K|_{\mathbb{M}_K} \ge a_{K,\mathbb{M}_K},
\]
which is very sharp. For $d = 3$, (\ref{kak}) becomes
\[
|K|_{\mathbb{M}_K} \ge \dfrac{1}{4}~a_{K,\mathbb{M}_K}^{2},
\]
which is not as sharp as in two dimensions. Indeed, when $K$ is equilateral with respect to $\mathbb{M}_K$,
we have \cite{EN}
\[
|K|_{\mathbb{M}_K} = \dfrac{1}{\sqrt{3}}~a_{K,\mathbb{M}_K}^{2} .
\]

\subsection{Equidistribution and alignment conditions}

We can now define the equidistribution and alignment conditions characterizing a general nonuniform, simplicial surface mesh.  Notice that any nonuniform mesh can be viewed as a uniform one in some metric tensor.  Specifically, a mesh is uniform in some metric if all of the elements in the mesh have the same size and are similar to a reference element with respect to that metric.  In this point of view, the equidistribution condition requires that all of the elements in the mesh have the same size.
Mathematically, this can be expressed as
\beq
\label{equi1}
|K|_{\M_K}=\dfrac{\sigma_h }{N}, \quad \forall K\in\mathcal{T}_h
\eeq
where, as before, $|\cdot |_{\mathbb{M}_K}$ denotes the area of the surface with respect to the metric $\mathbb{M}_K$ and \newline $\sigma_h = \sum_{K\in\mathcal{T}_h}|K|_{\mathbb{M}_K}$. Using Lemma~\ref{volKM} and recalling
$|\hat{K}| = 1$, we have
\[
|K|_{\M_K} = \det\left(\left(F_K'\right)^T\M_KF_K'\right)^{1/2}, \qquad
\sigma_h=\displaystyle\sum_{K\in\mathcal{T}_h}\det\left(\left(F_K'\right)^T\M_KF_K'\right)^{1/2}.
\]
 Thus, the equidistribution condition (\ref{equi1}) becomes
\beq
\det\left(\left(F_K'\right)^T\M_KF_K'\right)^{1/2}=\dfrac{\sigma_h}{N}, \quad \forall K \in \mathcal{T}_h.
\label{equ}
\eeq

The alignment condition, on the other hand, requires that all of the elements $K\in\mathcal{T}_h$ be similar to the reference element $\hat{K}$. Notice that any element $K$ is similar to $\hat{K}$ if and only if $F_K:\hat{K}\to K$ is composed by dilation, rotation, and translation, or equivalently, $F_K'$ is composed by dilation and rotation. Mathematically, $F_K'$ can therefore be expressed as
\beq
F_K' =\alpha U \left[\begin{matrix}I\\ \V 0\end{matrix}\right] V^T,
\eeq
where $\alpha$ is a constant representing dilation and $U \in \mathbb{R}^{d\times d}$ and $V\in\mathbb{R}^{(d-1)\times (d-1)}$ are orthogonal matrices representing rotation.
Thus,
\[
\M_K^{1/2}~F_K' =\alpha ~\M_K^{1/2}~U \left[\begin{matrix}I\\ \V 0\end{matrix}\right] V^T.
\]
It can be verified that the above equation is equivalent to
\beq
\label{ali}
\frac{1}{d-1}\text{tr}\left(\left(F_K'\right)^T\mathbb{M}_KF_K'\right)=\dFMF^{\frac{1}{d-1}},
\quad
\forall K \in \mathcal{T}_h
\eeq
which is referred to as the alignment condition. Here, $\tr(\cdot )$ denotes the trace of a matrix.

With these two conditions, we can now formulate the meshing function.  To do so, we first consider the alignment condition (\ref{ali}) and note that an equivalent condition is
$$\frac{1}{d-1}\tr\left[\left(\left(F_K'\right)^T\mathbb{M}_K F_K'\right)^{-1}\right]=\det \left[ \left(\left(F_K'\right)^T\mathbb{M}_K F_K'\right)^{-1}\right]^{\frac{1}{d-1}}.$$
Notice that the left- and right-hand sides are the arithmetic mean and the geometric mean of the eigenvalues of the matrix $\left((F_K')^T\mathbb{M}_K F_K'\right)^{-1}$, respectively. The inequality of arithmetic and geometric means gives 
\beq \label{ine}\frac{1}{d-1}\tr\left[\FMF^{-1}\right]\ge\det \left[ \left(\left(F_K'\right)^T\mathbb{M}_K F_K'\right)^{-1}\right]^{\frac{1}{d-1}},\eeq
with equality if and only if all of the eigenvalues are equal.  From (\ref{ine}), for any general mesh which does not necessarily satisfy (\ref{ali}), we have
$$\tr\left[\FMF^{-1}\right]^{d-1} \ge (d-1)^{d-1}\dFMF^{-1},$$
and therefore
$$\tr\left[\FMF^{-1}\right]^{\frac{p(d-1)}{2}} - (d-1)^{\frac{p(d-1)}{2}}\dFMF^{-\frac{p}{2}}\ge 0,$$
where $p>0$ is a dimensionless parameter which has been added to agree with the equidistribution energy function below.  Then, we define the alignment energy function as
\begin{align}
\label{Iali}
\notag I_{ali}=\displaystyle\sum_{K\in\mathcal{T}_h}&|\hat{K}|\dFMF^{\frac{1}{2}}
\tr\left[\FMF^{-1}\right]^{\frac{p(d-1)}{2}} \\
&- (d-1)^{\frac{p(d-1)}{2}}\displaystyle\sum_{K\in\mathcal{T}_h}|\hat{K}|\dFMF^{\frac{1-p}{2}} ,
\end{align}
whose minimization will result in a mesh that closely satisfies the alignment condition (\ref{ali}). One may notice that $|\hat{K}|\dFMF^{\frac{1}{2}} = |K|_{\mathbb{M}_K}$ has been added as a weight.  

Similarly, we consider the equidistribution condition (\ref{equ}). From H\"older's inequality, for any $p>1$ we have
\beq
\label{Holders}
\displaystyle\sum_{K\in\T_h}\dfrac{|K|_{\mathbb{M}_K}}{\sigma_h}\cdot \dfrac{1}{\dFMF^{1/2}}\le\left( \displaystyle\sum_{K\in\mathcal{T}_h} \dfrac{|K|_{\mathbb{M}_K}}{\sigma_h} \cdot \dfrac{1}{\dFMF^{p/2 }}\right)^{\frac{1}{p}},
\eeq
with equality if and only if 
\[
\dFMF^{-1/2}= \text{constant},\quad\quad \forall~K\in\mathcal{T}_h.
\]
That is, minimizing the difference between the left-hand side and the right-hand side of (\ref{Holders}) tends to make $\dFMF^{-1/2}$ constant for all $K\in\mathcal{T}_h$. Noticing that the left-hand side of (\ref{Holders}) is simply $N/\sigma_h$, we can rewrite this inequality as
\[
\left(\frac{N}{\sigma_h}\right)^p\cdot\sigma_h \le\displaystyle\sum_{K\in\mathcal{T}_h}|\hat{K}| \dFMF^{\frac{1-p}{2}}.
\]
We can consider $\sigma_h$ constant since $\sigma_h\approx \displaystyle\int_{S} \det(\M(\V x))^{1/2}d\V x$ and hence it only weakly depends on the mesh.  Therefore, we define the equidistribution energy function as
\beq
\label{Ieq}
I_{eq} = \left(d-1\right)^{\frac{p(d-1)}{2}}\displaystyle\sum_{K\in\T_h} |\hat{K}|\dFMF^{\frac{1-p}{2}},
\eeq
whose minimization will result in a mesh that closely satisfies the equidistribution condition.

\subsection{Energy function for combined equidistribution and alignment}

We now have two functions, one for each of equidistribution and alignment. Our goal is to formulate a single meshing function for which minimizing will result in a mesh that closely satisfies both conditions.  One way to ensure this is to average (\ref{Iali}) and (\ref{Ieq}), that is, define $I_h=\theta I_{ali}+(1-\theta) I_{eq}$ for $\theta\in[0,1]$. This leads to
\begin{align}
I_h 
& =\theta \displaystyle\sum_{K\in\mathcal{T}_h} |\hat{K}| \dFMF^{\frac 1 2}
\text{tr}\left[\FMF^{-1}\right]^{\frac{p(d-1)}{2}}
\notag \\
& \quad  + (1-2\theta) (d-1)^{\frac{p(d-1)}{2}} \displaystyle\sum_{K\in\T_h} |\hat{K}| \dFMF^{\frac{1-p}{2}},
\label{Ih-1}
\end{align}
where $p > 1$ and $\theta \in [0, 1]$ are dimensionless parameters, with the latter balancing the equidistribution and alignment conditions for which full alignment is achieved when $\theta = 1$ and full equidistribution is achieved when $\theta=0$. We can write (\ref{Ih-1}) as
\begin{equation}
I_h=\displaystyle \sum_{K\in \mathcal{T}_h} |\hat{K}| \dFMF^{\frac{1}{2}} \tilde{G}_K,
\label{Ih-2-0}
\end{equation}
where
\begin{equation}
\tilde{G}_K = \theta~\tr\left[\FMF^{-1}\right]^{\frac{p(d-1)}{2}} + (1-2\theta) (d-1)^{\frac{p(d-1)}{2}} \dFMF^{-\frac{p}{2}}.
\label{Ih-2-1}
\end{equation}

We remark that for $0<\theta\le \frac{1}{2}$ and $p> 1$, $\tilde{G}_K$ is coercive; see the definition of coercivity
in Section \ref{SEC:theor}. As can be seen therein, coercivity is an important property when proving mesh nonsingularity. 
It is also instructional to point out that the function (\ref{Ih-1}) is very similar to a Riemann sum
of the meshing function developed in \cite{H} for bulk meshes based on equidistribution and alignment.
One of the main differences is that $\FMF$ cannot be simplified in (\ref{Ih-1}) since it is not a square matrix
as it is in the bulk mesh case. Additionally, the constant terms and exponents that contain $d$ are $(d-1)$
in (\ref{Ih-1}) instead of $d$ in the bulk mesh case. The functional of \cite{H} has been proven to work well
for a variety of problems \cite{HR}.

\section{Moving mesh equations for surface meshes}
\label{SEC:MMPDE}

In this section we describe the surface MMPDE method used to find the minimizer of (\ref{Ih-1}).  In principle, we can directly minimize it. However, this minimization problem can be extremely difficult to solve since (\ref{Ih-1})
is highly nonlinear in general.  We employ here the MMPDE approach (a time transient approach) to find
the minimizer and define the moving mesh equation as the gradient system of the meshing function.

\subsection{Gradient of meshing energy}

Motivated by the function (\ref{Ih-1}), we consider meshing functions in a general form
(\ref{Ih-2-0}), i.e.,
\begin{equation}
I_h=\displaystyle \sum_{K\in \mathcal{T}_h} |\hat{K}| \dFMF^{\frac{1}{2}} \tilde{G}_K
\equiv \displaystyle \sum_{K\in \mathcal{T}_h} G(\J_K, r_K),
\label{Ih-2}
\end{equation}
where $\tilde{G}_K$ is a given smooth function of 
\[
\J_K = \left(\left(F_K'\right)^T\mathbb{M}_KF_K'\right)^{-1}, \quad r_K = \dFMF^{-1},
\]
that is, $\tilde{G}_K = \tilde{G}(\J_K, r_K)$, and  
\begin{equation}
G(\J_K, r_K) = |\hat{K}|r_K^{-\frac{1}{2}} \tilde{G}_K.
\label{G-1}
\end{equation}
Indeed, a special example is (\ref{Ih-2-1}) but $\tilde{G}_K$ can be chosen differently.
Moreover, both $\mathbb{J}_K$ and $r_K$ depend on the coordinates of the vertices of the physical element $K$
and hence $G$ is a function of them, i.e., $G\left(\J_K, r_K\right)$ can be expressed as
\beq
G\left(\J_K, r_K\right) = G_K\left(\V x_1^K, \dots, \V x_d^K\right),
\eeq
where $\V x_i^K \in \mathbb{R}^d$ for $i =1, \dots, d$ are the coordinations of the vertices of $K$. 
As a consequence, the sum in ($\ref{Ih-2}$) is a function of the coordinates of all vertices of
the physical mesh $\mathcal{T}_h$, i.e.,
\beq
I_h\left(\V x_1, \dots, \V x_{N_v}\right)  =\displaystyle\sum\limits_{K\in\mathcal{T}_h} G_K\left(\V x_1^K, \dots, \V x_d^K\right),
\eeq
where $\V x_i \in \mathbb{R}^d$ for $i =1, \dots, N_v$ are the coordinates of the vertices  of the mesh with global indices. One of the underlying keys to our approach is to find the derivatives of $I_h$ with respect to the physical coordinates $\V x_1, \dots, \V x_{N_v}$ which requires elementwise derivatives of $G_K$ with respect to $\V x_1^K, \dots,\V x_d^K$.
That is,
\beq
\dfrac{\partial I_h}{\partial \V x_i} = \displaystyle \sum_{K\in \mathcal{T}_h} \dfrac{\partial G_K}{\partial \V x_i}
=  \sum_{K \in \omega_i} \dfrac{\partial G_K}{\partial \V x_{i_K}^K},\quad i = 1, \dots, N_v
\label{Ih_xi}
\eeq
where $i_K$ denotes the local index of vertex $\V{x}_i$ in $K$ and $\omega_i$ is the element patch associated
with $\V{x}_i$.

 In order to calculate the necessary derivatives, we recall some definitions and properties of scalar-by-matrix differentiation (cf. \cite{HK2} for details).
Let $f = f(A)$ be a scalar function of a matrix $A\in \mathbb{R}^{m\times n}$. Then the scalar-by-matrix derivative of $f$ with respect to $A$ is defined as
\beq
\dfrac{\partial f}{\partial A} = 
\begin{bmatrix}
\frac{\partial f}{\partial A_{11}} & \cdots & \frac{\partial f}{\partial A_{m1}}\\
\vdots & & \vdots\\
\frac{\partial f}{\partial A_{1n}} & \cdots & \frac{\partial f}{\partial A_{mn}}\\
\end{bmatrix}_{n\times m}
\quad \text{or} \quad \left(\dfrac{\partial f}{\partial A}\right)_{i,j} = \dfrac{\partial f}{\partial A_{j,i}}~.
\eeq
The chain rule of differentiation with respect to $t$ is
\beq
\dfrac{\partial f}{\partial t} = \sum_{ij} \dfrac{\partial f}{\partial A_{j,i}} \dfrac{\partial A_{j,i}}{\partial t} = \sum_{ij} \left(\dfrac{\partial f}{\partial A}\right)_{i,j}\dfrac{\partial A_{j,i}}{\partial t} = \tr\left(\dfrac{\partial f}{\partial A}\dfrac{\partial A}{\partial t}\right) .
\eeq
With this and when $A$ is a square matrix, the following properties have been proven in \cite{HK2},
\beq\label{deriv}
\dfrac{\partial \tr\left(A\right)}{\partial A} = I, \qquad\frac{\partial A^{-1}}{\partial t} = - A^{-1} \frac{\partial A}{\partial t} A^{-1},\quad
\frac{\partial \det(A)}{\partial t} = \det(A) ~\text{tr}\left ( A^{-1} \frac{\partial A}{\partial t} \right ).
\eeq
Using the above, we can find the expressions for $\frac{\partial G}{\partial \mathbb{J}}$ and $\frac{\partial G}{\partial r}$ which are needed to compute (\ref{Ih_xi}).
For the function (\ref{Ih-1}), the first derivatives of $G$ are given by
\beq
\label{der}
\begin{cases}
 \dfrac{\partial G}{\partial \J} & = \dfrac{\theta p(d-1) }{2} |\hat{K}|r^{-\frac{1}{2}}  \text{tr}(\J) ^{\frac{p(d-1)-2}{2}} I, \\
\vspace{.005cm}\\
\dfrac{\partial G}{\partial r} & = -\dfrac{\theta}{2} |\hat{K}|r^{-\frac{3}{2}} \text{tr}(\J)^{\frac{p(d-1)}{2}} 
+ \dfrac{p-1}{2} (1-2\theta) (d-1)^{\frac{p(d-1)}{2}}|\hat{K}| r^{\frac{p-3}{2}}.
\end{cases}
\eeq

\subsection{Derivatives of the meshing function with respect to the physical coordinates }
\label{discretization}

From (\ref{Ih_xi}), we can see that we will need $\partial G_K/\partial \V x_{i_K}^K$ to compute
$\partial I_h/\partial \V x_i$. The former can be obtained once we know the derivatives of $G_K$ with
respect to the coordinates of all vertices of $K$, i.e.,
\[
\frac{\partial G_K}{\partial [\V x_1^K,\V  x_2^K, \dots ,\V  x_d^K]}
= \begin{bmatrix} \dfrac{\partial G_K}{\partial \V x_1^K} \\ \vdots \\ \dfrac{\partial G_K}{\partial \V x_d^K}\end{bmatrix} .
\]
The derivation of these derivatives is given in Appendix \ref{SEC:discrederiv}. They are given as
\begin{align}
\begin{bmatrix} \dfrac{\partial G_K}{\partial \V x_2^K}\\ \vdots \\ \dfrac{\partial G_K}{\partial \V x_d^K}\end{bmatrix}
= & -2 \left(E_K^T\mathbb{M}_K E_K\right)^{-1} \hat{E}^T
\frac{\partial G_K}{\partial \J} \hat{E} (E_K^T\mathbb{M}_K E_K)^{-1} E_K^T\mathbb{M}_K
\notag \\
& \quad  -2 \frac{\det(\hat{E})^{2}}{\det\left(E_K^T \mathbb{M}_K E_K\right)}\frac{\partial G_K}{\partial r} \left(E_K^T \mathbb{M}_K E_K\right)^{-1} E_K^T\mathbb{M}_K\notag\\
&\quad+\frac{1}{d} \displaystyle\sum_{j=1}^d \text{tr} \left (\frac{\partial G_K}{\partial \mathbb{M}_K}  \mathbb{M}_{j,K} \right )
\begin{bmatrix} \frac{\partial \phi_{j,K}}{\partial \V x} \\ \vdots \\ \frac{\partial \phi_{j,K}}{\partial \V x} \end{bmatrix},
\label{vel2}
\\
\dfrac{\partial G_K}{\partial \V x_1^K}\
= &  -\sum_{j=2}^d \dfrac{\partial G_K}{\partial \V x_j^K}
+ \displaystyle\sum_{j=1}^d \text{tr} \left (\frac{\partial G_K}{\partial \mathbb{M}_K}  \mathbb{M}_{j,K} \right )
\frac{\partial \phi_{j,K}}{\partial \V x} ,
\label{vel1}
\end{align}
where ${\partial G_K}/{\partial \J}$ and ${\partial G_K}/{\partial r}$ are given in (\ref{der}), 
$\phi_{j,K}$ is the linear basis function associated with $\V{x}_j^K$, $\mathbb{M}_{j,K} = \mathbb{M}(\V{x}_j^K)$, and
\begin{align}
\frac{\partial G_K}{\partial \mathbb{M}_K} 
& = - E_K (E_K^T\mathbb{M}_K E_K)^{-1} \hat{E}^T
\frac{\partial G_K}{\partial \J} \hat{E} (E_K^T\mathbb{M}_K E_K)^{-1} E_K^T
\notag \\
& \quad  -  \frac{\det(\hat{E})^{2}}{\det\left(E_K^T \mathbb{M}_K E_K\right)}\frac{\partial G_K}{\partial r} E_K\left(E_K^T \mathbb{M}_K E_K\right)^{-1} E_K^T ,
\label{der-4}
\\
\begin{bmatrix}\frac{\partial \phi_{2,K}}{\partial \V x}\\ \vdots \\ \frac{\partial \phi_{d,K}}{\partial \V x}\end{bmatrix}
& =(E_K^TE_K)^{-1}E_K^T, \qquad 
\dfrac{\partial \phi_{1,K}}{\partial \V x}=-\displaystyle\sum_{j=2}^d\dfrac{\partial \phi_{j,K}}{\partial \V x} .
\label{der-5}
\end{align}
Having computed $\partial G_K/\partial \V x_{j}^K$ ($j=1, ..., d$) for all elements using (\ref{vel2}) and (\ref{vel1}), 
we can obtain $\partial I_h/\partial \V x_i$ from (\ref{Ih_xi}).


\subsection{Surface moving mesh equations}
\label{MMeqn}
   
As mentioned above, we employ a surface MMPDE method to minimize the meshing function ($\ref{Ih-1}$)
or a more general form (\ref{Ih-2}).  An MMPDE is a mesh equation that involves mesh speed.  There are various formulations of MMPDEs; we focus here on the approach where the surface MMPDE is defined as the modified gradient system of the meshing function. A distinct feature for surface meshes, other than bulk meshes, is that the nodes need to stay on the surface. By Section \ref{discretization} we may assume that we have the matrix
\[
\dfrac{\partial I_h}{\partial  \V{x}_i},
\quad {i=1,\cdots, N_v}.
\]
Let $\Phi( \V{x})=0$ denote the surface, where $\Phi$ can be defined through an analytical expression or a numerical representation such as by spline functions.  Then for the vertices to stay on the surface we should have
$\Phi( \V{x}_i)=0, \; i = 1, ..., N_v$ or at least 
\begin{equation}
\frac{d \V{x}_i}{d t} \cdot \nabla \Phi( \V{x}_i) = 0, \quad i = 1, ..., N_v
\label{surf-1}
\end{equation}
where $\frac{d \V{x}_i}{d t}$ is the nodal mesh velocity.
Following the MMPDE approach, we would define the mesh equation as the gradient system of $I_h$, i.e.,
\begin{equation}
\frac{d \V{x}_i}{d t} = - \frac{P_i}{\tau} \left ( \frac{\partial I_h}{\partial \V{x}_i} \right )^T, \quad i = 1, ..., N_v
\label{MMPDE-0}
\end{equation}
where $P_i$ is a positive scalar function used to make the equation have desired invariance properties and $\tau>0$ is a constant parameter used for adjusting the time scale of mesh movement.
Obviously, this does not satisfy (\ref{surf-1}). Here we propose to project the velocities in (\ref{MMPDE-0})
onto the surface and define the surface moving mesh equation as
\begin{equation}
\frac{d \V{x}_i}{d t} = - \frac{P_i}{\tau} \left [ \left ( \frac{\partial I_h}{\partial \V{x}_i} \right )^T - \left (\left (\frac{\partial I_h}{\partial \V{x}_i}\right )^T \cdot \V{n}_i \right ) \V{n}_i \right ], \quad i = 1, ..., N_v
\label{MMPDE-1}
\end{equation}
where 
$\V{n}_i = {\nabla \Phi (\V{x}_i)}/{\| \nabla \Phi (\V{x}_i) \|}$ is the unit normal to the surface at $\V{x}_i$
and the difference inside the square bracket is the projection of the vector ${\partial I_h}/{\partial \V{x}_i}$
onto the tangential plane of the surface at $\V{x}_i$.
Notice that this surface MMPDE inherently ensures that (\ref{surf-1}) be satisfied or, in words,
{\em the nodes stay on the surface during the mesh movement}. Moreover, it is important to note that (\ref{MMPDE-1}) only utilizes the unit normal vectors of the surface whose computation does not require explicit parameterization or analytical expression of the surface. As mentioned, for surfaces without an analytical expression, spline functions may be used to approximate the gradient for (\ref{MMPDE-1}). Although the numerical examples presented in this work have explicit parameterizations, it is a goal for our future work to study the use of spline functions to approximate
the unit normal vectors for surfaces represented by simplicial background meshes.

Using (\ref{Ih_xi}) we can rewrite the above equation in a compact form as
\begin{equation}
\dfrac{d \V x_i}{dt}=\dfrac{P_i}{\tau}\sum_{K\in \omega_i}\V v_{i_K}^K,\quad i = 1, ..., N_v
\label{MMPDEx}
\end{equation} 
where $\V v_{i_K}^K \in \mathbb{R}^d$ is the local mesh velocities contributed by $K$ to $\V x_{i_K}^K$
and has the expressions
\begin{align}
\V v_{i_K}^K = - \left ( \frac{\partial G_K}{\partial \V{x}_{i_K}^K} \right )^T + \left (\left (\frac{\partial G_K}{\partial \V{x}_{i_K}^K}\right )^T \cdot \V{n}_{i_K} \right ) \V{n}_{i_K} ,
\label{v-0}
\end{align}
and ${\partial G_K}/{\partial \V x_{i_K}^K}$ is given in (\ref{vel2}) and (\ref{vel1}).

The surface MMPDE (\ref{MMPDEx}) must be modified properly for boundary vertices when $S$ has a boundary. For fixed boundary vertices, the corresponding equation is replaced by
\[
\frac{d \V x_i}{dt} = 0.
\]
The velocities for other boundary vertices should be modified such that they slide on the boundary.

With proper modification of the boundary vertices, the system (\ref{MMPDEx}) can be integrated in time. To do so, one first starts by calculating the edge matrices $E_K$ for all elements and $\hat{E}$ for the reference element.
One can then readily calculate (\ref{der}) which is needed for (\ref{vel2}) and (\ref{vel1}). Then one can integrate (\ref{MMPDEx}) in time. For this work we use Matlab's ODE solvers ode45 and ode15s. The explicit scheme, ode45, implements a 4(5)-order Runge-Kutta method with a variable time step. The implicit scheme, ode15s, is a variable time step and variable-order solver based on the numerical differentiation formulas of orders 1 to 5. All of the numerical examples in this paper use ode45 although both ode45 and ode15s have been tested and proven to work very well in computation.

\section{Nonsingularity of surface moving meshes}
\label{SEC:theor}
In this section we study the nonsingularity of the mesh trajectory and the existence of limiting meshes as $t \to \infty$ for the MMPDE (\ref{MMPDEx}).

\subsection{Equivalent measure of minimum height}
We begin the theoretical analysis by establishing the relation between $\| \FMF^{-1}\|$ and
the minimum altitude of $K$ with respect to $\mathbb{M}_K$.
\begin{lem}
\label{lem1}
There holds
\beq
\dfrac{\hat{a}^2}{a_{K,\mathbb{M}_K}^2} \le \left \| \FMF^{-1}\right\| \le \dfrac{(d-1)^2\hat{a}^2}{a_{K,\mathbb{M}_K}^2},
\label{lem1-1}
\eeq
where $\hat{a}$ is the altitude of $\hat{K}$ and $a_{K,\mathbb{M}_K}$ is the minimum altitude of $K$ with respect to the metric $\mathbb{M}_K$.
\end{lem}

\begin{proof}
First of all, we have
$$
\left\| \left[ \left(F_K'\right)^T F_K' \right]^{-1}\right\| = \left\| \left[\hat{E}^{-T}E_K^T E_K\hat{E}^{-1}\right]^{-1}\right\|  = \left \| \hat{E}\left[E_K^T E_K\right]^{-1}\hat{E}^{T}\right\|.$$
Now, consider the QR decomposition of $E_K$
\[
E_K = Q_K\left[\begin{matrix} R_K\\ \V0\end{matrix}\right],
\]
where $Q_K\in \mathbb{R}^{d\times d}$ is a unitary matrix, $R_K\in\mathbb{R}^{(d-1) \times (d-1)}$ is an upper triangular
matrix, and $\V0$ is a $(d-1)$-dimensional row vector of zeros.  With this we have
\begin{align*}
 \left \| \hat{E}\left[E_K^T E_K\right]^{-1}\hat{E}^{T}\right\|& =  \left \| \hat{E}\left([R_K^T ~\V0^T] Q_K^T Q_K \left[\begin{matrix} R_K\\ \V0\end{matrix}\right] \right)^{-1}\hat{E}^{T}\right\|\\
 & = \left \| \hat{E}R_K^{-1}R_K^{-T}\hat{E}^{T}\right\|\\
 & = \left \| \left(R_K\hat{E}^{-1}\right)^{-1}\left(R_K\hat{E}^{-1}\right)^{-T}\right\| .
 \end{align*}
 By  \cite[Lemma 4.1]{HK2} we have
 $$\dfrac{\hat{a}^2}{a_{R_K}^2} \le \left\| \left(R_K\hat{E}^{-1}\right)^{-1}\left(R_K\hat{E}^{-1}\right)^{-T}\right\|\le \dfrac{(d-1)^2\hat{a}^2}{a_{R_K}^2},$$
where $a_{R_K}$ is the minimum altitude of the simplex formed by the columns of $R_K$.
Since $Q_K$ is a rotation matrix, $a_{R_K}$ is the same as $a_K$, the minimum altitude of $K$ with respect to the Euclidean metric. Combining the above results, we get
\[
\dfrac{\hat{a}^2}{a_{K}^2} \le \left \|  \left[ \left(F_K'\right)^T F_K' \right]^{-1} \right\| \le \dfrac{(d-1)^2\hat{a}^2}{a_{K}^2}.
\]
The inequality (\ref{lem1-1}) follows from this and the observation that the geometric properties of $K$ with respect
to the metric $\mathbb{M}_K$ are the same as those of $\mathbb{M}_K^{1/2}K$ with respect to the Euclidean metric.
\end{proof}

Lemma~\ref{lem1} indicates that if $\hat{K}$ is chosen to satisfy $|\hat{K}| = \mathcal{O}(1)$ then
\beq
\left \| \FMF^{-1}\right\|\sim a_{K,\mathbb{M}_K}^{-2}.
\eeq

\subsection{Mesh nonsingularity}

We now consider the MMPDE (\ref{MMPDEx}). Recall that the velocities for the boundary vertices need to
be modified in order for them to stay on the boundary. However, the analysis is similar with or without modifications. Hence, for simplicity we do not consider modifications in the analysis.
We also note that for theoretical purposes, we assume that $\hat{K}$ is taken to satisfy
$|\hat{K}| = \frac{1}{N}$ instead of being unitary as we have been considering thus far.
This change does not affect the actual computation. However, since typically we expect $|K| = \mathcal{O}(1/N)$,
the assumption $|\hat{K}| = \frac{1}{N}$ will likely lead to $F_K' = \mathcal{O}(1)$ and thus $I_h(\mathcal{T}_h(0))$
(the value of $I_h$ on the initial mesh $\mathcal{T}_h(0)$) stays $\mathcal{O}(1)$. 
On the other hand, if $|\hat{K}|  = 1$ (unitary), we have $F_K' = \mathcal{O}(1/N)$ and $I_h(\mathcal{T}_h(0))$
will depend strongly on $N$.

In the following analysis, the mesh at time $t$ is denoted by $\mathcal{T}_h(t) = \left(\V x_1(t), \dots , \V x_{N_v}(t)\right)$.

\begin{thm}
\label{nonsin}
Assume that the meshing function in the form (\ref{Ih-2}) satisfies the coercivity condition
\beq
\label{coer}
\tilde{G} \left(\mathbb{J}, \det\left(\mathbb{J}\right),\V x \right) \ge \alpha~ \left (\tr\left[\FMF^{-1}\right]\right )^q
- \beta, \quad \quad \forall \V x \in S
\eeq
where $q> (d-1)/2$, $\alpha>0$, and $\beta\ge 0$ are constants.  We also assume that $\hat{K}$ is equilateral and $|\hat{K}| = \frac{1}{N}$. Then if the elements of the mesh trajectory of the MMPDE (\ref{MMPDEx}) have positive areas initially, they will have positive areas for all time. Moreover, their minimum altitudes in the metric $\mathbb{M}_K$ and their areas in the Euclidean metric are bounded below by
\begin{align}
\label{amin}
a_{K,\mathbb{M}_K} &\ge C_1 ~\left[I_h(\mathcal{T}_h(0)) + \beta \bar{m}^{d/2} |S|\right]^{-\frac{1}{2q-d+1}} ~N^{-\frac{2q}{(d-1)(2q-d+1)}}, \\[10pt]
\label{kmin}
|K| & \ge C_2~\left[I_h(\mathcal{T}_h(0)) + \beta \bar{m}^{d/2} |S|\right]^{-\frac{d-1}{2q-d+1}} ~N^{-\frac{2q}{2q-d+1}}~\overline{m}~^{-\frac{d}{2}},
\end{align}
where 
\beq
C_1 = \left(\frac{\alpha~ d^{\frac{q(d - 2)}{d-1}}~(d-1)!^{\frac{2q-d+1}{d-1}}}{(d-1)^{\frac{d-1+2q}{2}}}\right)^{\frac{1}{2q-d+1}}, \qquad C_2 = \dfrac{C_1^{d-1}}{(d-1)^{\frac{d-1}{2}}(d-1)!}.
\eeq
\end{thm}

\begin{proof}
From (\ref{MMPDE-1}) we have
\begin{align*}
\frac{dI_h}{dt} & = \sum_{i} \frac{\partial I_h}{\partial  \V x_i} \frac{d \V x_i}{dt}
 = -\sum_{i} \dfrac{P_i}{\tau}\frac{\partial I_h}{\partial  \V x_i} \left [ \left ( \frac{\partial I_h}{\partial \V{x}_i} \right )^T - \left (\left (\frac{\partial I_h}{\partial \V{x}_i}\right )^T \cdot \V{n}_i \right ) \V{n}_i \right ]\\
& = -\sum_{i} \dfrac{P_i}{\tau} \left [ \left \| \frac{\partial I_h}{\partial \V{x}_i} \right \|^2 - \left (\left (\frac{\partial I_h}{\partial \V{x}_i}\right )^T \cdot \V{n}_i \right )^2 \right ]\\
& \le 0.\\
\end{align*}
This implies $I_h\left(\mathcal{T}_h(t)\right) \le I_h\left(\mathcal{T}_h(0)\right)$ for all $t$.
From coercivity (\ref{coer}) and Lemma \ref{lem1}, we get 
\begin{align*}
I_h\left(\mathcal{T}_h(t)\right) & \ge \alpha\displaystyle \sum_{K\in \mathcal{T}_h} |\hat{K}| \dFMF^{1/2}\left (  \tr\left[\FMF^{-1}\right]\right )^q - \beta \bar{m}^{d/2} |S|\\
& \ge \alpha\displaystyle \sum_{K\in \mathcal{T}_h} |\hat{K}| \dFMF^{1/2}  \left\|\FMF^{-1}\right\|^q - \beta \bar{m}^{d/2} |S|\\
& \ge \alpha\displaystyle \sum_{K\in \mathcal{T}_h}  |\hat{K}|\dFMF^{1/2} \dfrac{\hat{a}^{2q}}{a_{K,\mathbb{M}_K}^{2q}} - \beta \bar{m}^{d/2} |S| .
\end{align*}
By Lemma \ref{lem2}, $|\hat{K}|\dFMF^{1/2} = |K|_{\mathbb{M}_K} \ge \frac{1}{(d-1)^{\frac{d-1}{2}}(d-1)!}~a_{K,\mathbb{M}_K}^{d-1}$, thus
\beq I_h\left(\mathcal{T}_h(t)\right) +\beta \bar{m}^{d/2} |S| \ge \dfrac{\alpha \hat{a}^{2q}}{(d-1)^{\frac{d-1}{2}}(d-1)!} \displaystyle \sum_{K\in \mathcal{T}_h} \dfrac{1}{a_{K,\mathbb{M}_K}^{2q-d+1}}~,\eeq
and therefore 
\beq\label{akm2} a_{K,\mathbb{M}_K}^{2q-d+1} \ge \dfrac{\alpha \hat{a}^{2 q}}{(d-1)^{\frac{d-1}{2}}(d-1)!}\left(I_h\left(\mathcal{T}_h(0)\right) +\beta \bar{m}^{d/2} |S|\right)^{-1}.\eeq
Moreover, from the assumption that $\hat{K}$ is equilateral and $|\hat{K}| = \frac{1}{N} $ it follows that 
\beq
\label{ahat}
\hat{a} = \dfrac{\sqrt{d}~(d-1)!^{\frac{1}{d-1}}}{\sqrt{d-1}~d^{\frac{1}{2(d-1)}}} ~N^{-\frac{1}{d-1}}.
\eeq
Combining (\ref{akm2}) and (\ref{ahat}) we get
\beq
\label{akm}
a_{K,\mathbb{M}_K}\ge \left(\frac{\alpha~ d^{\frac{q(d - 2)}{d-1}}~(d-1)!^{\frac{2q-d+1}{d-1}}}{(d-1)^{\frac{d-1+2q}{2}}}\right)^{\frac{1}{2q-d+1}}\left[I_h(\mathcal{T}_h(0)) + \beta \bar{m}^{d/2} |S|\right]^{-\frac{1}{2q-d+1}} N^{-\frac{2q}{(d-1)(2q-d+1)}},
\eeq
which gives (\ref{amin}).

Furthermore, we have
\begin{align*}
 \frac{a_{K,\mathbb{M}_K}^{d-1}}{(d-1)^{\frac{d-1}{2}}(d-1)!} &\le |K|_{\mathbb{M}_K} = |\hat{K}|\dFMF^{1/2}\\
 & \le \overline{m}~^{d/2}|\hat{K}|\det\left(\left(F_K'\right)^TF_K'\right)^{1/2} = \overline{m}~^{d/2}|K| .
 \end{align*}
 Then (\ref{kmin}) follows from  the above inequality and (\ref{amin}).

Finally, from (\ref{vel2}) and (\ref{vel1}) it is not difficult to see that
the magnitude of the mesh velocities is bounded from above when $|K|$ is bounded from below.
As a consequence, the mesh vertices will move continuously with time and $|K|$ cannot jump over the bound to become negative. Hence, $|K|$ will stay positive if so initially.
\end{proof}

From the proof we have seen that the key points are the energy decreasing property and
the coercivity of the meshing function. The former is satisfied by the MMPDE (\ref{MMPDEx}) by design
while the latter is an assumption for the meshing function. We emphasize that the result holds for
any function satisfying the coercivity condition (\ref{coer}).

On the other hand, the condition (\ref{coer}) is satisfied by the meshing function (\ref{Ih-1})
for $0<\theta\le \frac{1}{2}$ and $p>1$ (with $q = {(d-1)p}/{2}$ and $\beta = 0$ in Theorem \ref{nonsin}).
It is interesting to point out that the role of the parameter $p$ can be explained from (\ref{amin}).
Indeed, for this case the inequality (\ref{amin}) becomes
\beq 
a_{K,\mathbb{M}_K} \ge C_1 ~\left[I_h(\mathcal{T}_h(0)) + \beta \bar{m}^{d/2} |S|\right]^{-\frac{1}{(d-1)(p-1)}} ~N^{-\frac{p}{(d-1)(p-1)}} \to C_1~N^{-\frac{1}{d-1}}, \quad p\to \infty.
\eeq 
Since $N^{-\frac{1}{d-1}}$ represents the average diameter of the elements, the above inequality implies that
the mesh becomes more uniform as $p$ gets larger.
We take $p = 1.5$, which has been found to work well for all examples we have tested.

\begin{thm}
\label{limiting}
Under the assumptions of Theorem \ref{nonsin}, for any nonsingular initial mesh, the mesh trajectory $\{\mathcal{T}_h(t), t>0\}$ of MMPDE (\ref{MMPDEx}) has the following properties.
\begin{enumerate}
	\item $I_h(\mathcal{T}_h(t))$ has a limit as $t\to \infty$, i.e.,
	\beq \lim_{t\to \infty} I_h(\mathcal{T}_h(t)) = L.\eeq
	\item The mesh trajectory has limiting meshes, all of which are nonsingular and satisfy (\ref{amin}) and (\ref{kmin}).
	\item The limiting meshes are critical points of $I_h$, i.e., they satisfy
	\beq\label{critical} \dfrac{\partial I_h}{\partial \V x_i} = 0, \quad i = 1, \dots, N_v.\eeq
\end{enumerate}
\end{thm}

\begin{proof}
The proof is very much the same as that for \cite[Theorem 4.3]{HK1} for the bulk mesh case.
The key ideas to the proof are the monotonicity and boundedness of $I_h(\mathcal{T}_h(t))$ and the compactness of $\overline{S}$.  With these holding for the surface mesh case, one can readily prove the three properties.
\end{proof}

It is remarked that the above two theorems have been obtained for the MMPDE (\ref{MMPDEx}) which is semi-discrete
in the sense that it is discrete in space and continuous in time. A fully discrete scheme can be obtained by applying
a time-marching scheme to (\ref{MMPDEx}). Then similar results can be obtained for the fully discrete scheme
under similar assumptions and for a sufficiently small but not diminishing time step. Since the analysis is similar
to that for the bulk mesh case, the interested reader is referred to \cite{HK1} for the detailed discussion.


\section{Numerical examples}
\label{SEC:numerical}
     
Here we present numerical results for a selection of two- and three-dimensional examples to demonstrate the performance of the surface moving mesh method described in the previous sections. The main focus will be on showing how our method can be used for mesh smoothing and concentration. To assess the quality of the generated meshes, we compare the equidistribution ($Q_{eq}$) and alignment ($Q_{ali}$) mesh quality measures which are defined as
\begin{equation}
Q_{eq}=\max_{K\in\mathcal{T}_h} \dfrac{\det\left((F_K')^T\mathbb{M}_K F_K'\right)^{\frac{1}{2}}}{\sigma_h/N} \qquad \text{and}
\qquad Q_{ali}=\max_{K\in\mathcal{T}_h}\dfrac{\tr\left[\left((F_K')^T\mathbb{M}_K F_K'\right)^{-1}\right]}
{(d-1)\det\left((F_K')^T\mathbb{M}_K F_K'\right)^{-\frac{1}{d-1}}}~.
\end{equation}
These measures are indications of how closely the mesh 
satisfies the equidistribution condition (\ref{equ}) and the alignment condition (\ref{ali}), respectively.
The closer these quality measures are to 1, the closer they are to a uniform mesh with respect to
the metric $\mathbb{M}_K$.  It should be noted that the alignment condition does not apply to the two-dimensional case
where a ``surface'' is actually a curve.  Mathematically, when $d=2$, $\left(F_K'\right)^T\mathbb{M}_KF_K'$ is a
number and hence ($\ref{ali}$) is always satisfied.

For all computations we use $p = 3/2$ and $\theta = 1/3$ in the meshing function (\ref{Ih-1}).
This choice has been known to work well in bulk mesh applications. Interestingly, we have found that
it also works well for all surface mesh examples we have tested.
We take $\tau = 0.01$, $dt = 0.01$, and
\[
P_i = \det\left(\mathbb{M}(\V{x}_i)\right)^{\frac{p(d-1)-d}{2}}.
\]
The latter is to ensure that the MMPDE (\ref{MMPDEx}) be invariant under scaling transformations of $\mathbb{M}$. For all of the results, we run to a final time of 1.0. 

We choose two forms of $\mathbb{M}_K$. The first is $\mathbb{M}_K=I$, which will ensure the mesh move to become as uniform as possible with respect to the Euclidean norm. The second is a curvature-based metric tensor defined as a scalar matrix $\mathbb{M}_K = \left(k_K+\epsilon\right)I$, where $k_K$ is the mean curvature and $\epsilon$ is machine epsilon. The mean curvature is defined (e.g., see \cite{G}) for a curve $\Phi(x,y) = 0$ in $\mathbb{R}^2$ as
\[
k = \left| \dfrac{\Phi_{xx}\Phi_y^2 - 2\Phi_{xy}\Phi_x\Phi_y+\Phi_x^2\Phi_{yy}}{\left(\Phi_x^2+\Phi_y^2\right)^{\frac{3}{2}}}\right|
\]
and for a surface $\Phi(x,y,z) = 0$ in $\mathbb{R}^3$ as
\[
k = \left|\dfrac{ D_1 + D_2+ D_3- D_4}{2\left(\Phi_x^2+\Phi_y^2+\Phi_z^2\right)^{3/2}}\right| ,
\]
where
\begin{align*}
D_1 &= \Phi_x\left(\Phi_x\Phi_{xx}+\Phi_y\Phi_{xy}+\Phi_z\Phi_{xz}\right),\\
D_2 &= \Phi_y\left(\Phi_x\Phi_{xy}+\Phi_y\Phi_{yy}+\Phi_z\Phi_{yz}\right),\\
D_3 &= \Phi_z\left(\Phi_x\Phi_{xz}+\Phi_y\Phi_{yz}+\Phi_z\Phi_{zz}\right),\\
D_4 &= \left(\Phi_x^2+\Phi_y^2+\Phi_z^2\right)\left(\Phi_{xx} + \Phi_{yy} + \Phi_{zz}\right) .
\end{align*}
We would like to explore more metric tensors in future work but will focus on these two for this paper.

\vspace{10pt}

\begin{exam}
\label{circle}
For the first example, we generate adaptive meshes for the unit circle in two dimensions,
$$\Phi(x,y) = x^2 + y^2 - 1.$$
We take $N = 80$ and fix the node $\V x_{1} = (1,0)$.

Fig. \ref{figcircle} shows the meshes for this example.  Studying the figures we see that the initial mesh
Fig.~\ref{figcircle}(a) is very nonuniform but the final meshes Fig. \ref{figcircle}(b) and (c) have adapted to be equidistant along the curve.  Moreover, the final meshes for both $\mathbb{M}_K = I$ (Fig. \ref{figcircle} (b)) and $\mathbb{M}_K = (k_K+\epsilon)I$ (Fig. \ref{figcircle}(c)) adapt the mesh in the same manner. This is consistent with the fact that the curvature of a circle is constant thus the nodes do not concentrate in one particular region of the curve. The final meshes in both cases provide good size adaptation and are more uniformly distributed along the curve when compared with the initial mesh.  This can be further supported assessing the mesh quality measure for which $Q_{eq}$ improves from 7.509604 to 1.000004 for both cases of $\mathbb{M}_K$. The fact that $Q_{eq}\approx 1$ indicates that the mesh is close to satisfying the equidistribution condition (\ref{equ}) and hence the mesh is almost uniform with respect to the metric tensor $\mathbb{M}_K$.  It can also be seen that the nodes remain on the curve $\Phi$, which, as mentioned earlier,
is an inherent feature of the new surface moving mesh method and indeed an important one when adapting
a mesh on a curve.

\begin{figure}[h]
\begin{center}
\hbox{
\begin{minipage}[t]{2in}
\begin{center}
\includegraphics[width=2in]{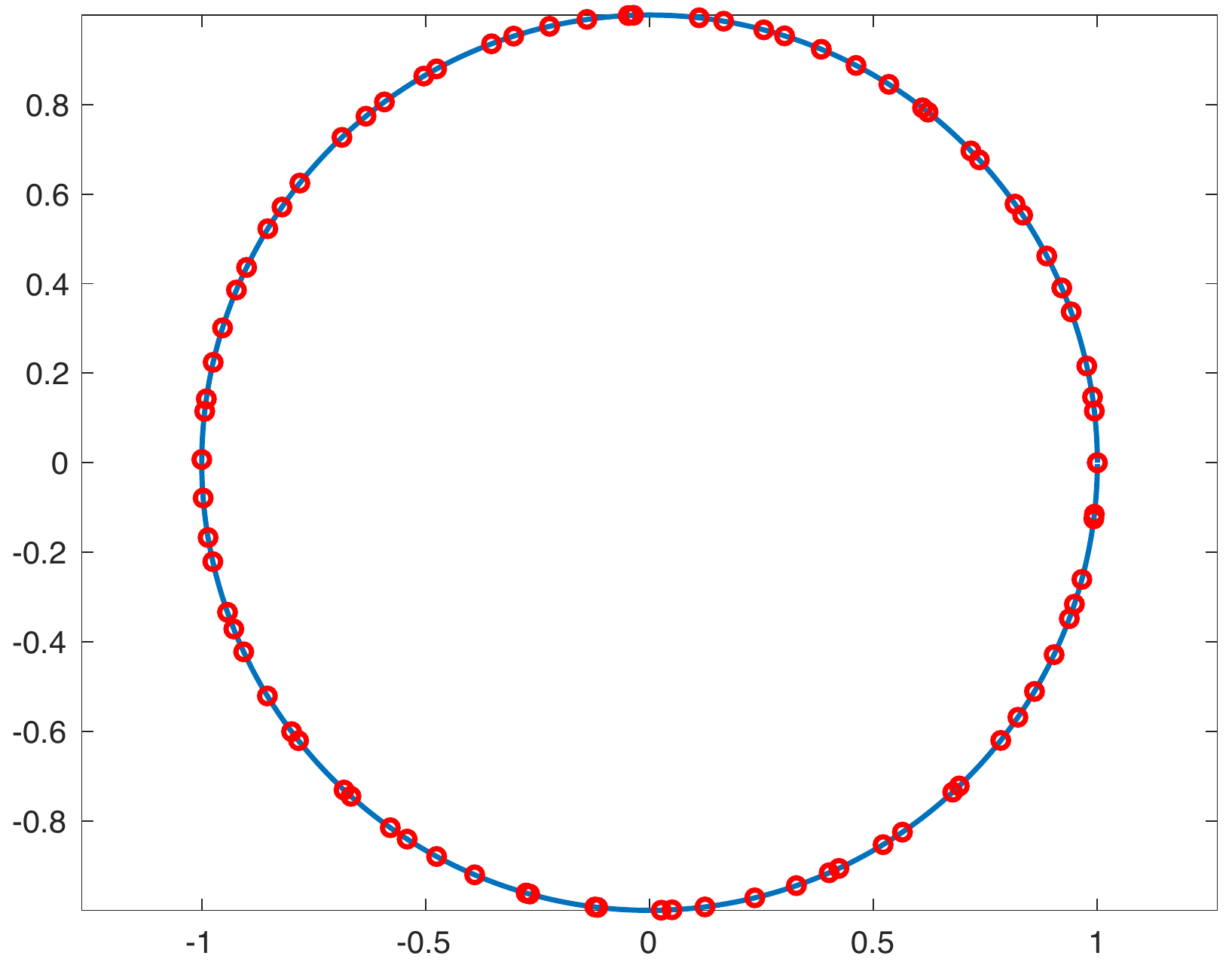}
\end{center}
\vspace{0.42cm}
\centerline{(a) Initial Mesh}\
\end{minipage}
\hspace{2mm}
\begin{minipage}[t]{2in}
\begin{center}
\includegraphics[width=2in]{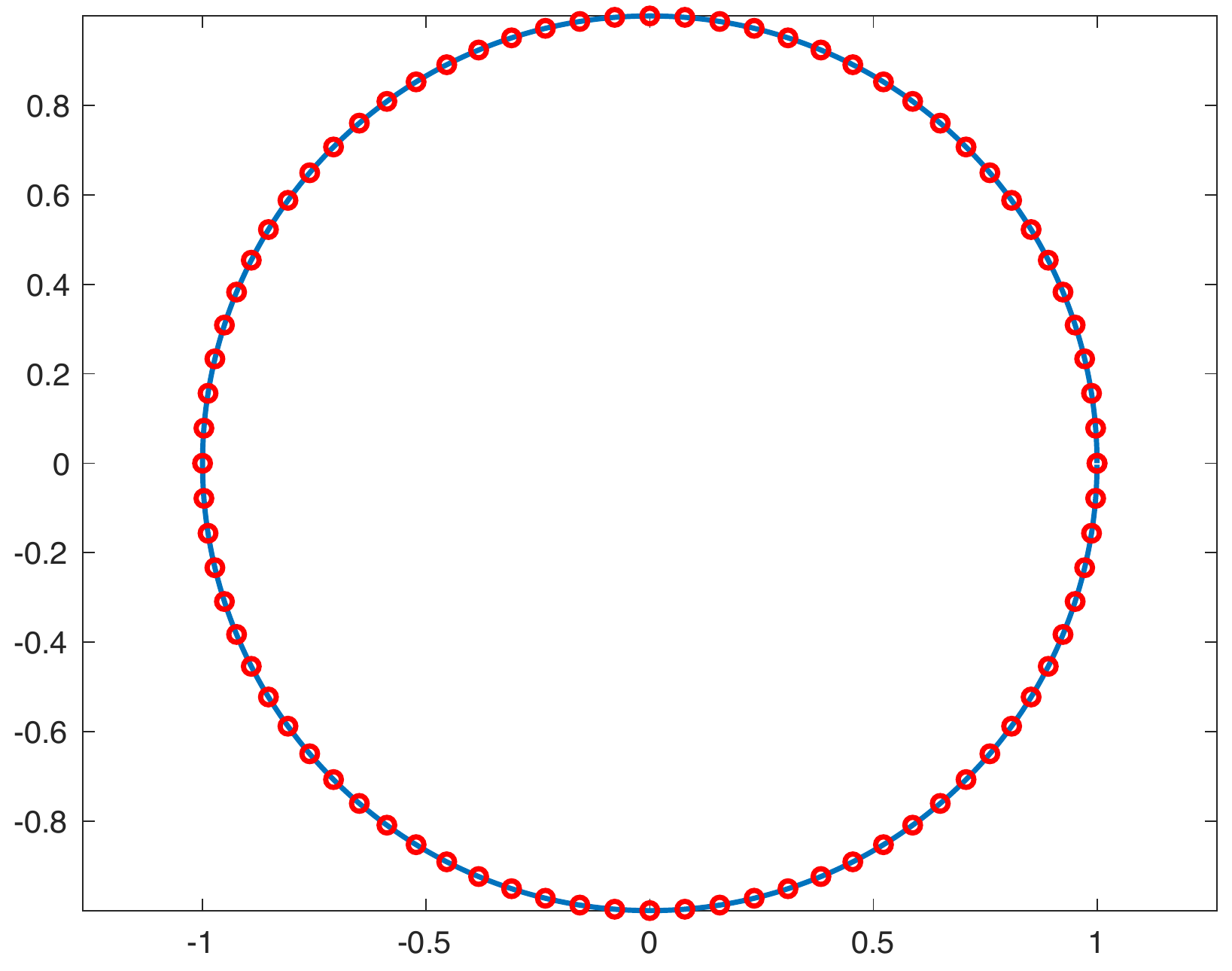}
\end{center}
\centerline{(b) Final Mesh, $\mathbb{M}_K = I$}\
\end{minipage}
\hspace{2mm}
\begin{minipage}[t]{2in}
\begin{center}
\includegraphics[width=2in]{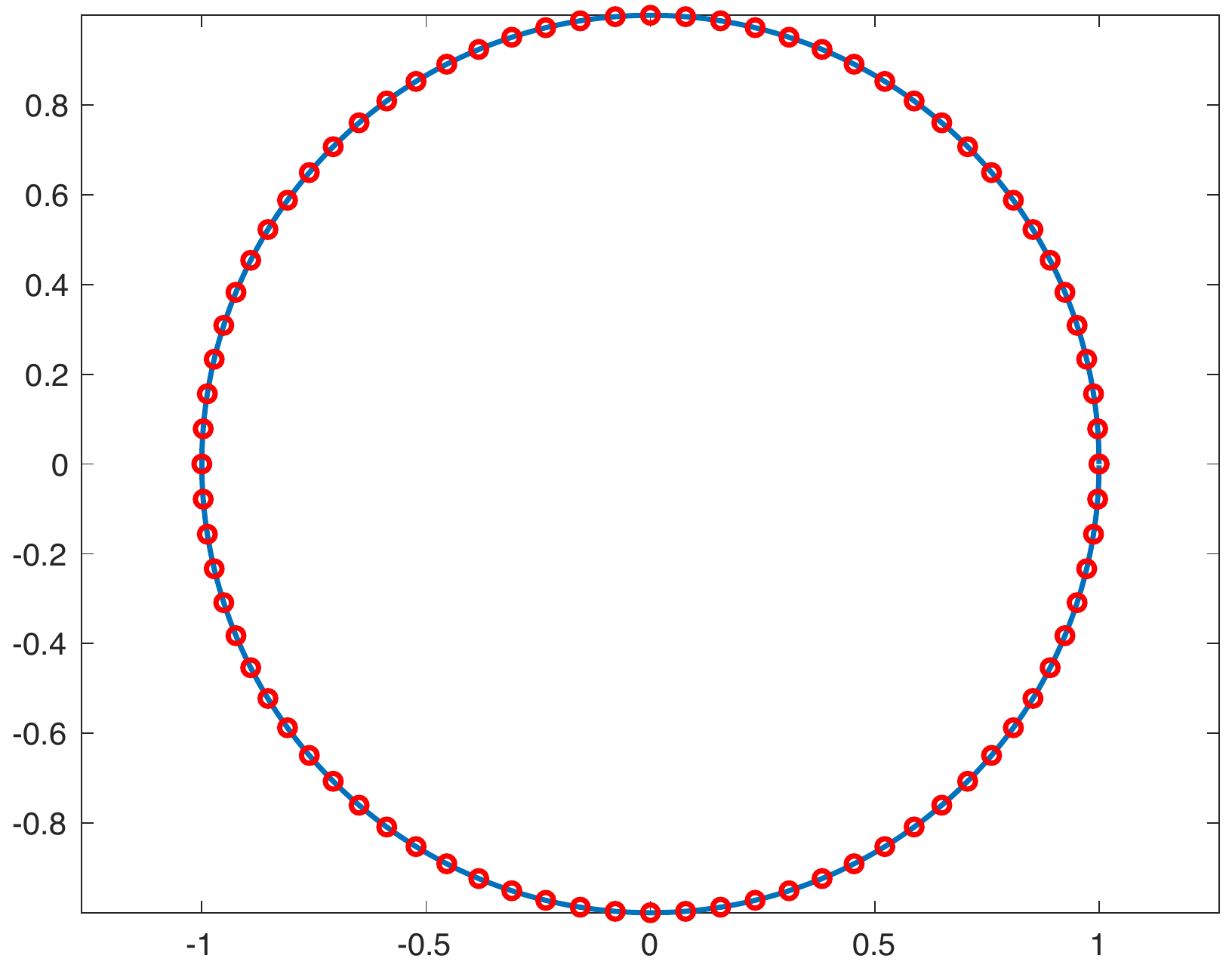}
\end{center}
\centerline{(c) Final Mesh, $\mathbb{M}_K = (k_K + \epsilon)I$}\
\end{minipage}
}
\caption{Example~\ref{circle}. Meshes of $N=80$ are obtained for $\Phi(x,y) = x^2+y^2-1$.
}
\label{figcircle}
\end{center}
\end{figure}

\end{exam}

\vspace{10pt}

\begin{exam}
\label{ellipse}

The second two-dimensional example is the ellipse defined by
$$\Phi(x,y) = \dfrac{x^2}{64}+y^2-1.$$
In this example we take $N= 60$ and fix the node $\V x_1 = (8,0)$.

The initial nodes (Fig. \ref{figellipse}(a)) are randomly distributed through the curve. However, for $\mathbb{M}_K = I$, the final mesh (Fig. \ref{figellipse}(b)) is equidistant along the ellipse providing a much more uniform mesh. This can also be seen in $Q_{eq}$ which improves from 5.497002 initially to 1.026912 in the final mesh. 

Now considering the curvature-based metric tensor (Fig. \ref{figellipse}(c)), we can see a high concentration of elements near the regions of the ellipse with large curvature. This is consistent with the equidistribution principle which
requires higher concentration in the regions with larger determinant of the metric tensor (larger mean curvature in the
current situation). The mean curvature is large in the regions of the ellipse close to $x = -8 , 8$ and almost 0 for $x \in (-2,2)$. From Fig. \ref{figellipse}(c) we can see that the adaptation with $\mathbb{M}_K = (k_K + \epsilon) I$ provides a mesh that represents the shape of the curve much better than other two meshes. The improvement of $Q_{eq}$ from 5.126216 to 1.015848 indicates that the final mesh is almost uniform with respect to the curvature-based metric tensor.

\begin{figure}[h]
\begin{center}
\hbox{
\begin{minipage}[t]{2in}
\begin{center}
\includegraphics[width=2in]{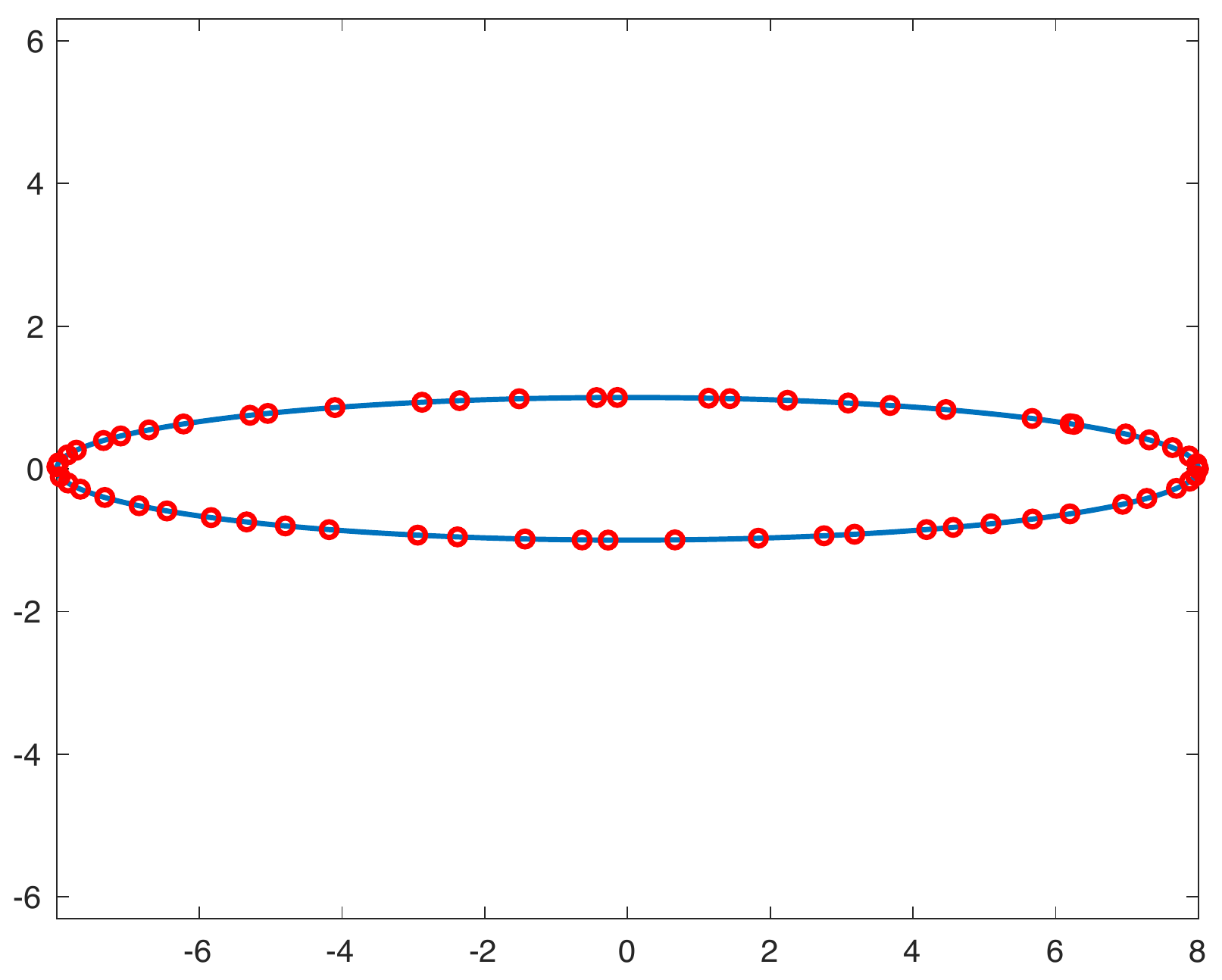}
\end{center}
\vspace{0.42cm}
\centerline{(a) Initial Mesh}\
\end{minipage}
\hspace{2mm}
\begin{minipage}[t]{2in}
\begin{center}
\includegraphics[width=2in]{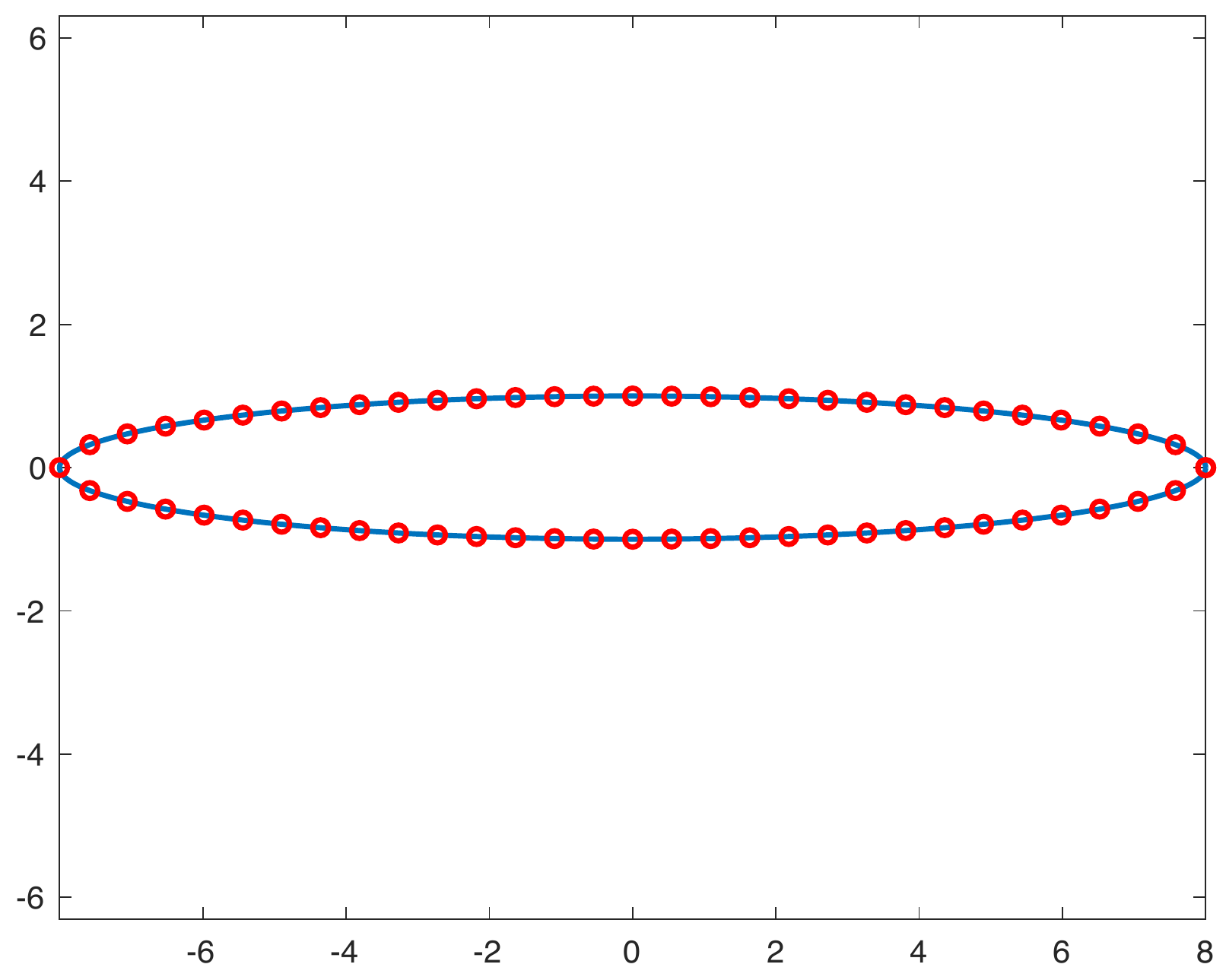}
\end{center}
\centerline{(b) Final Mesh, $\mathbb{M}_K = I$}\
\end{minipage}
\hspace{2mm}
\begin{minipage}[t]{2in}
\begin{center}
\includegraphics[width=2in]{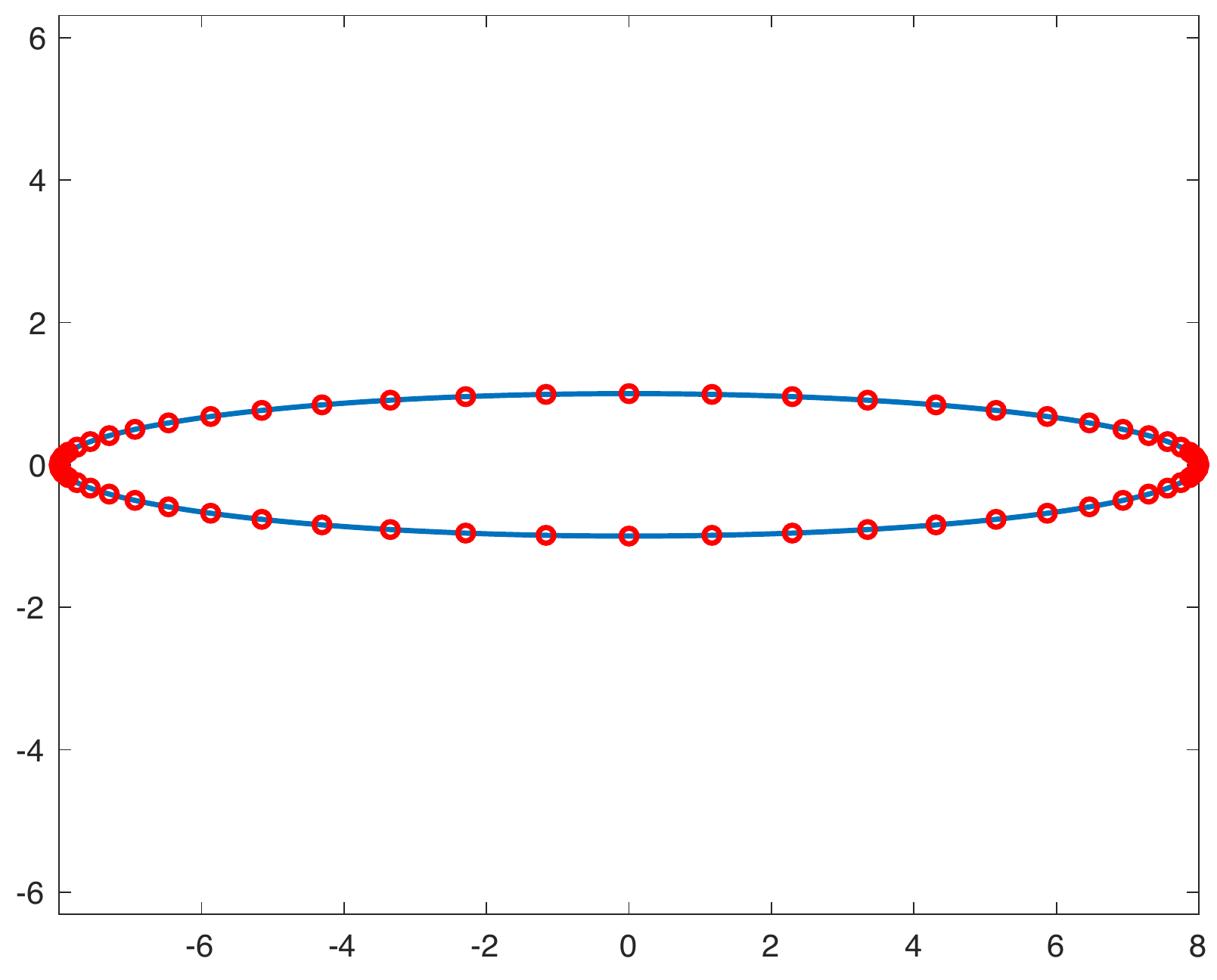}
\end{center}
\centerline{(c) Final Mesh, $\mathbb{M}_K = (k_K + \epsilon)I$}\
\end{minipage}
}
\caption{Example~\ref{ellipse}.  Meshes of $N = 60$ are obtained for $\Phi(x,y) = \dfrac{x^2}{64}+y^2-1$.
}
\label{figellipse}
\end{center}
\end{figure}
\end{exam}

\vspace{10pt}

\begin{exam}
\label{sin2d}

For the next two-dimensional example, we generate adaptive meshes for the sine curve defined by
$$\Phi(x,y) = 4\sin(x) - y.$$
In this example we take $N= 60$ and fix the end nodes $\V x_1 = (0,0)$ and $\V x_{61} = (2\pi, 0)$.

Fig. \ref{figsin2d} shows the meshes for this example.  From Fig. \ref{figsin2d}(a) and (b) we see that for $\mathbb{M}_K = I$, the mesh becomes much more uniform.  This is consistent with the fact that for $\mathbb{M} = I$, the minimization of the meshing function will make the mesh more uniform with respect to the Euclidean norm.  The observation can be further supported by assessing the mesh quality measures for which $Q_{eq}$ measure improves from 4.183312 to 1.002906 indicating that the final mesh satisfies the equidistribution condition (\ref{equ}) closely.

Now studying Fig. \ref{figsin2d}(c) where $\mathbb{M}_K = (k_K + \epsilon) I$ is used, we see that there is a high concentration of mesh elements in regions with large curvature, i.e., the hill at $y = 4$ and cup at $y = -4$, which is consistent with the use of the curvature-based metric tensor. Moreover, the equidistribution measure $Q_{eq}$ improves from 6.254755 to 1.007493. This indicates that although the mesh may seem nonuniform in the Euclidean metric, it is almost uniform in the metric $\mathbb{M}_K$.

\begin{figure}[h]
\begin{center}
\hbox{
\begin{minipage}[t]{2in}
\begin{center}
\includegraphics[width=2in]{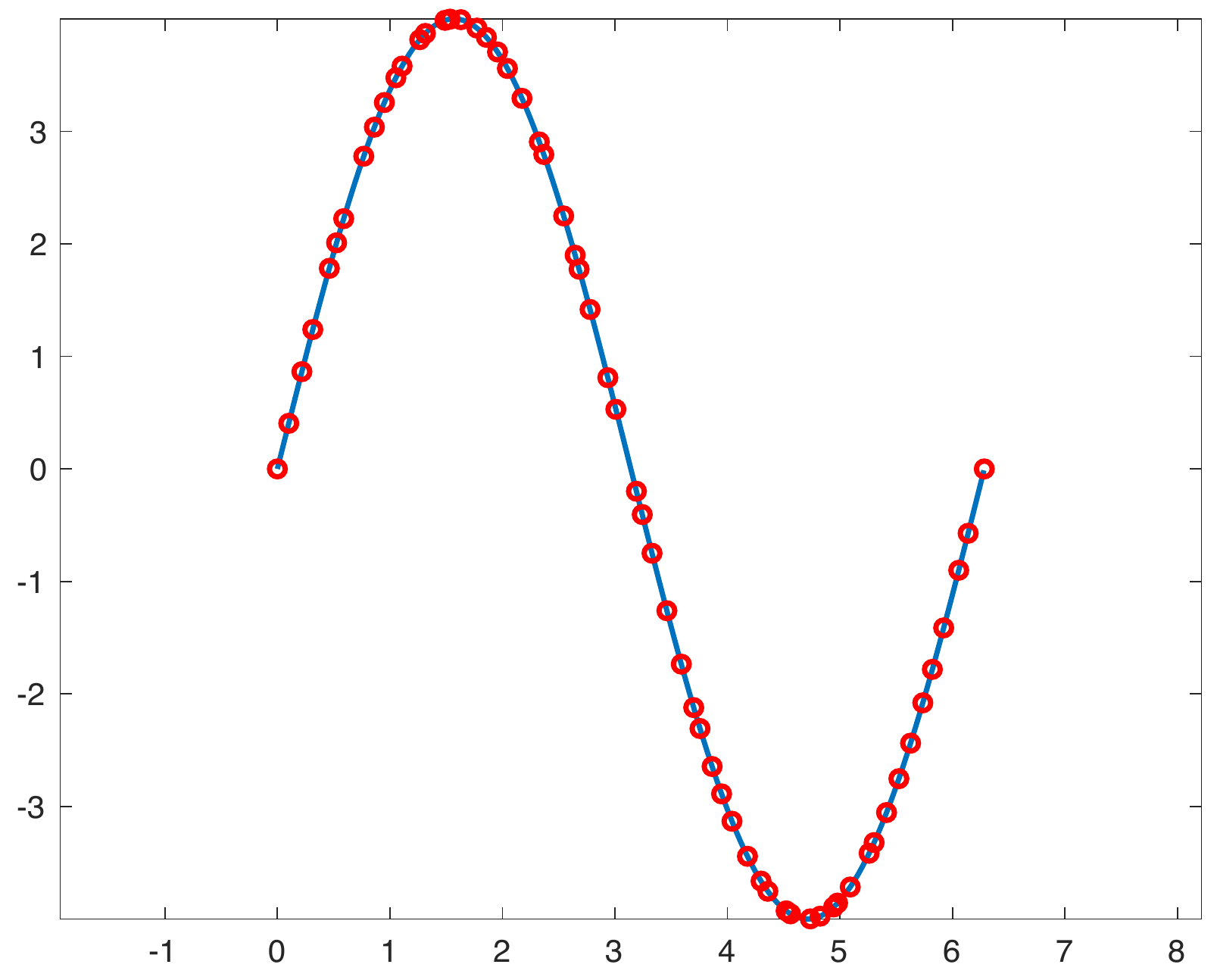}
\end{center}
\vspace{0.42cm}
\centerline{(a) Initial Mesh}\
\end{minipage}
\hspace{2mm}
\begin{minipage}[t]{2in}
\begin{center}
\includegraphics[width=2in]{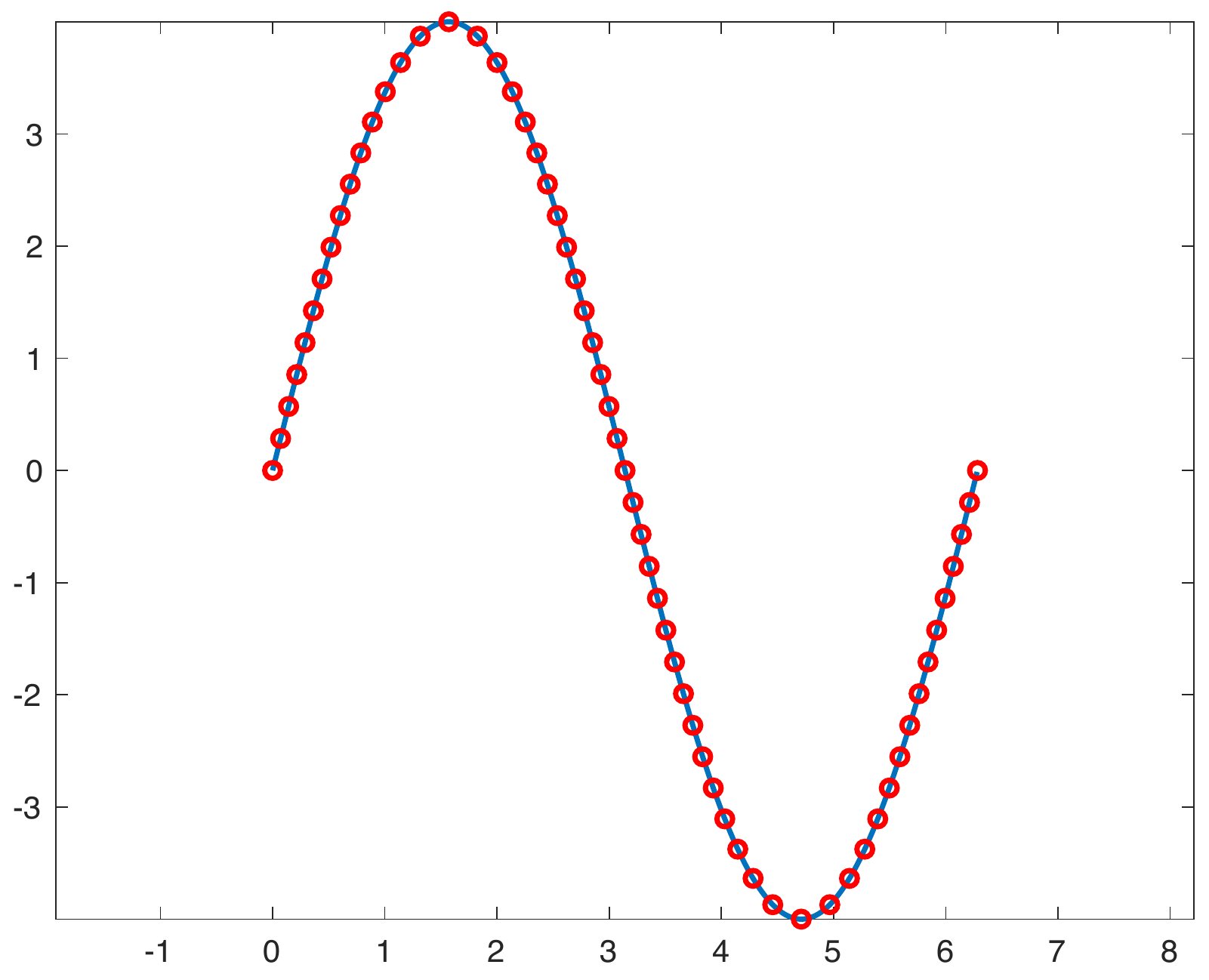}
\end{center}
\centerline{(b) Final Mesh, $\mathbb{M}_K = I$}\
\end{minipage}
\hspace{2mm}
\begin{minipage}[t]{2in}
\begin{center}
\includegraphics[width=2in]{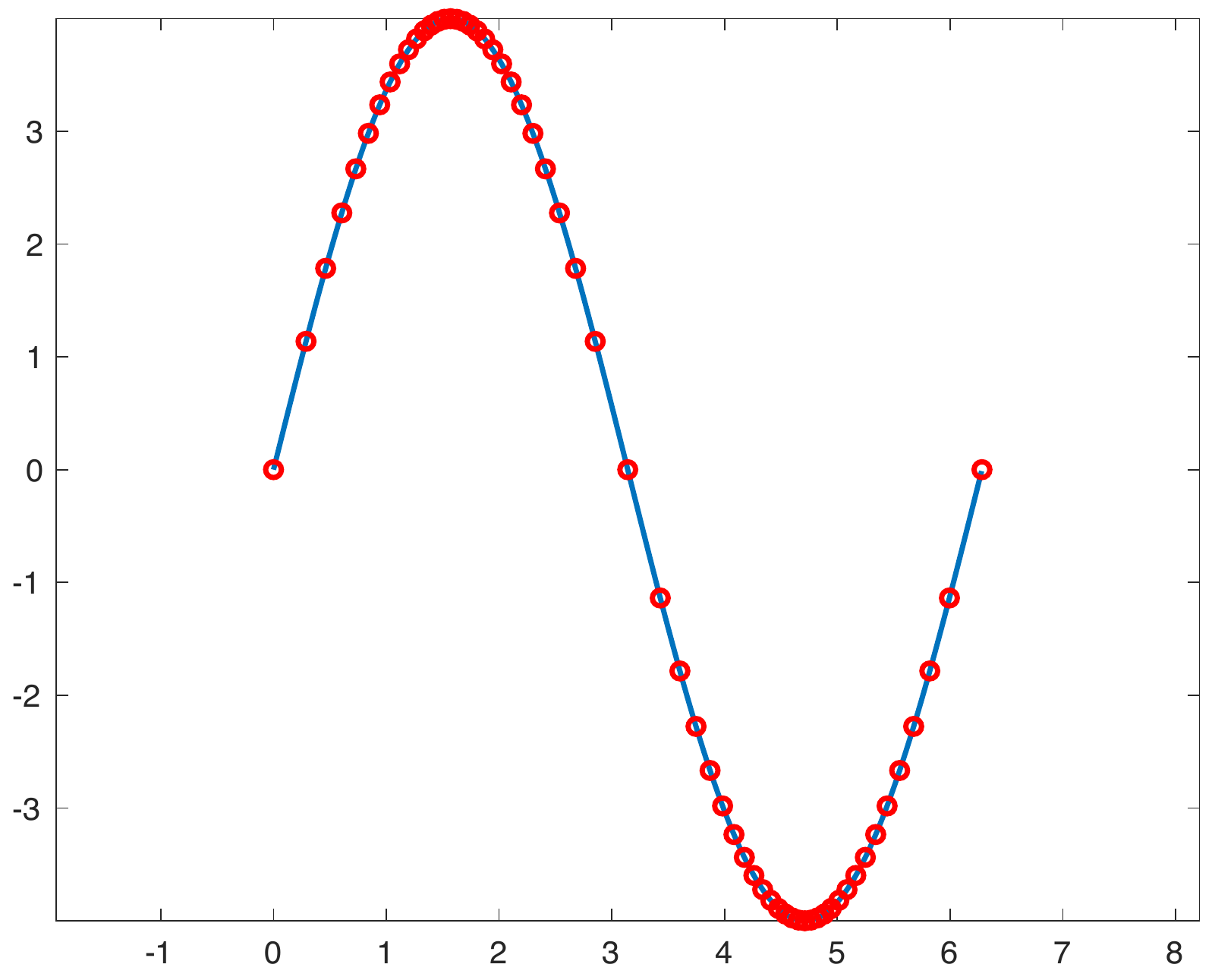}
\end{center}
\centerline{(c) Final Mesh, $\mathbb{M}_K = (k_K+\epsilon)I$}\
\end{minipage}
}
\caption{Example~\ref{sin2d}. Meshes of $N = 60$ are obtained for $\Phi(x,y) = 4\sin(x) - y$.
}
\label{figsin2d}
\end{center}
\end{figure}

As discussed in Section~\ref{SEC:theor}, theoretically we know that the value of $I_h$ is decreasing and $|K|$ is bounded below.  To see these numerically, we plot $I_h$ and $|K|_{\min}$ as functions of $t$ in Fig. \ref{figsin2d2}, where $|K|_{\min}$ denotes the minimum area of $K$ over all elements in $\mathcal{T}_h$. The numerical results are shown to be consistent with the theoretical predictions.  Specifically, for $\mathbb{M}_K = I$, Fig. \ref{figsin2d2}(a) shows
that $I_h$ is decreasing and bounded below by 9.535. Additionally, Fig. \ref{figsin2d2}(b) suggests that $|K|_{\min}$ is bounded below by 0.235 which is the value of $|K|_{\min}$ of the initial mesh. As we see, $|K|_{\min}$ first increases and then converges to about 0$.285 \approx \frac{|S|}{N}$. The reason is because in the final mesh, the elements are close to being uniform with respect to the Euclidean metric and thus $|K| \approx \frac{|S|}{N}$ for all $K$. Since the initial mesh is nonuniform, we expect an increase in $|K|_{\min}$ as the mesh is becoming more uniform.  Moreover, as the mesh reaches the limiting mesh trajectory around $t = 0.05$, we see that $|K|_{\min}$ converges as shown in
Fig. \ref{figsin2d2}(b).

For the case with $\mathbb{M}_K = (k_K + \epsilon) I$, the numerical results are again consistent with the theoretical predictions.  Fig. \ref{figsin2d2}(c) shows that $I_h$ is decreasing for all time and bounded below by 15.5. This figure also shows that at around $t = 0.15$, $I_h$ begins to converge. In Fig. \ref{figsin2d2} (d), $|K|_{\min}$ has similar properties to Fig. \ref{figsin2d2}(b). That is, we see an initial increase in $|K|_{\min}$ after which, the value converges to 0.11 starting at around $t = 0.15$. Furthermore, Fig. \ref{figsin2d2}(d) suggests that $|K|_{\min}$ is bounded below by the initial value of 0.045.

\begin{figure}[h]
\begin{center}
\hbox{
\begin{minipage}[t]{3in}
\begin{center}
\includegraphics[width=2.75in]{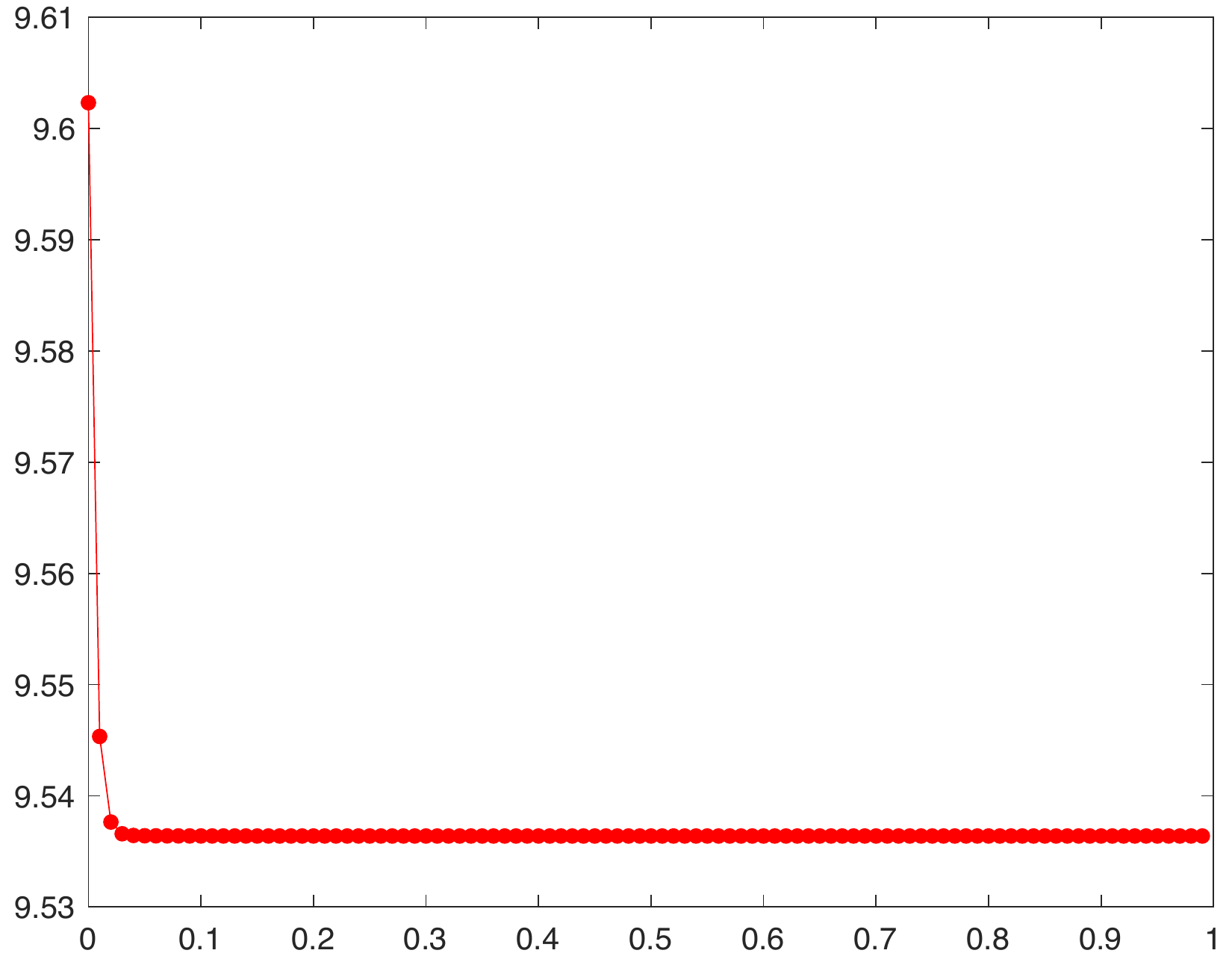}
\end{center}
\vspace{0.42cm}
\centerline{(a) $I_h$, $\mathbb{M}_K = I$}\
\end{minipage}
\hspace{2mm}
\begin{minipage}[t]{3in}
\begin{center}
\includegraphics[width=2.75in]{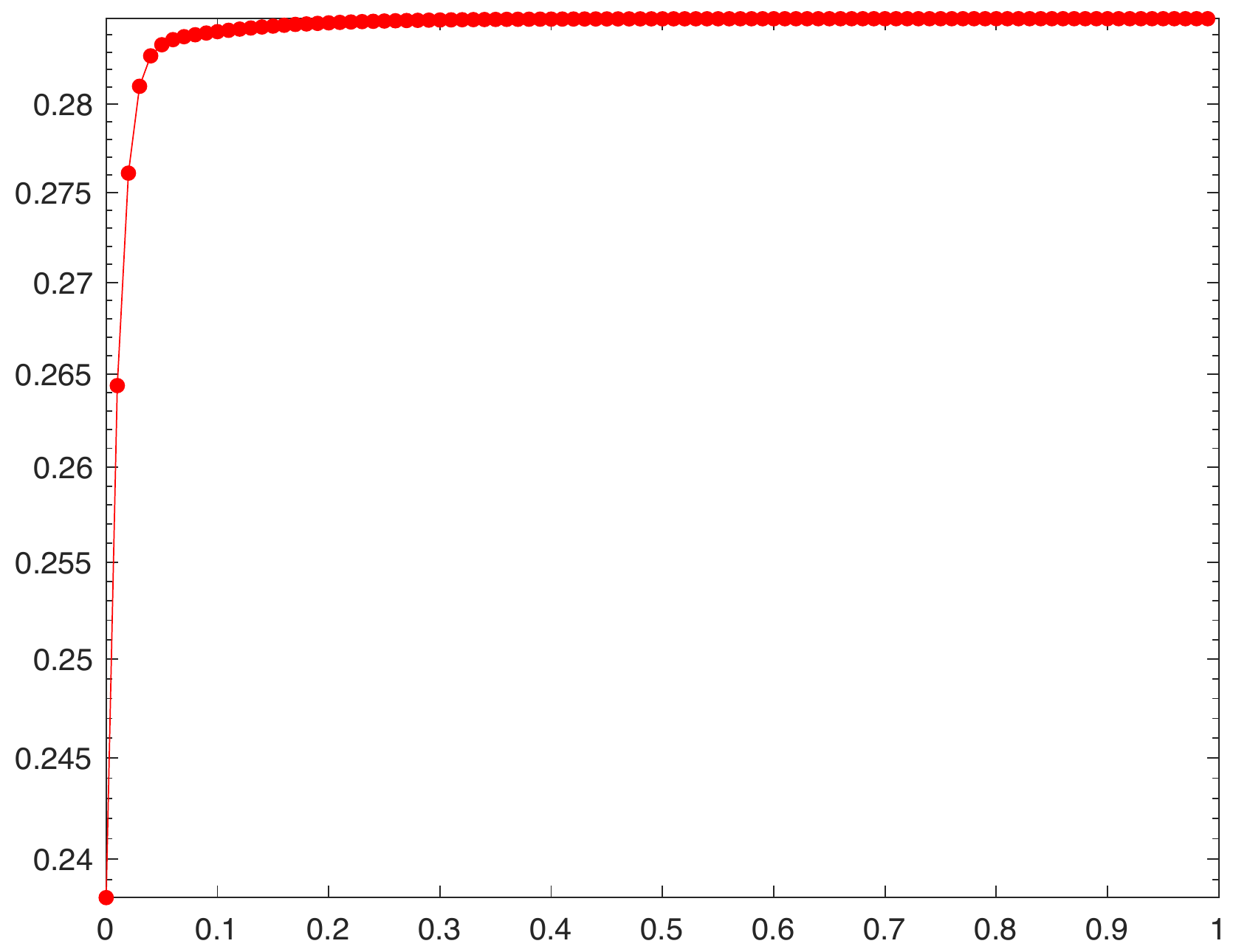}
\end{center}
\centerline{(b) $|K|_{\min}$, $\mathbb{M}_K = I$}\
\end{minipage}
}
\hbox{
\begin{minipage}[t]{3in}
\begin{center}
\includegraphics[width=2.75in]{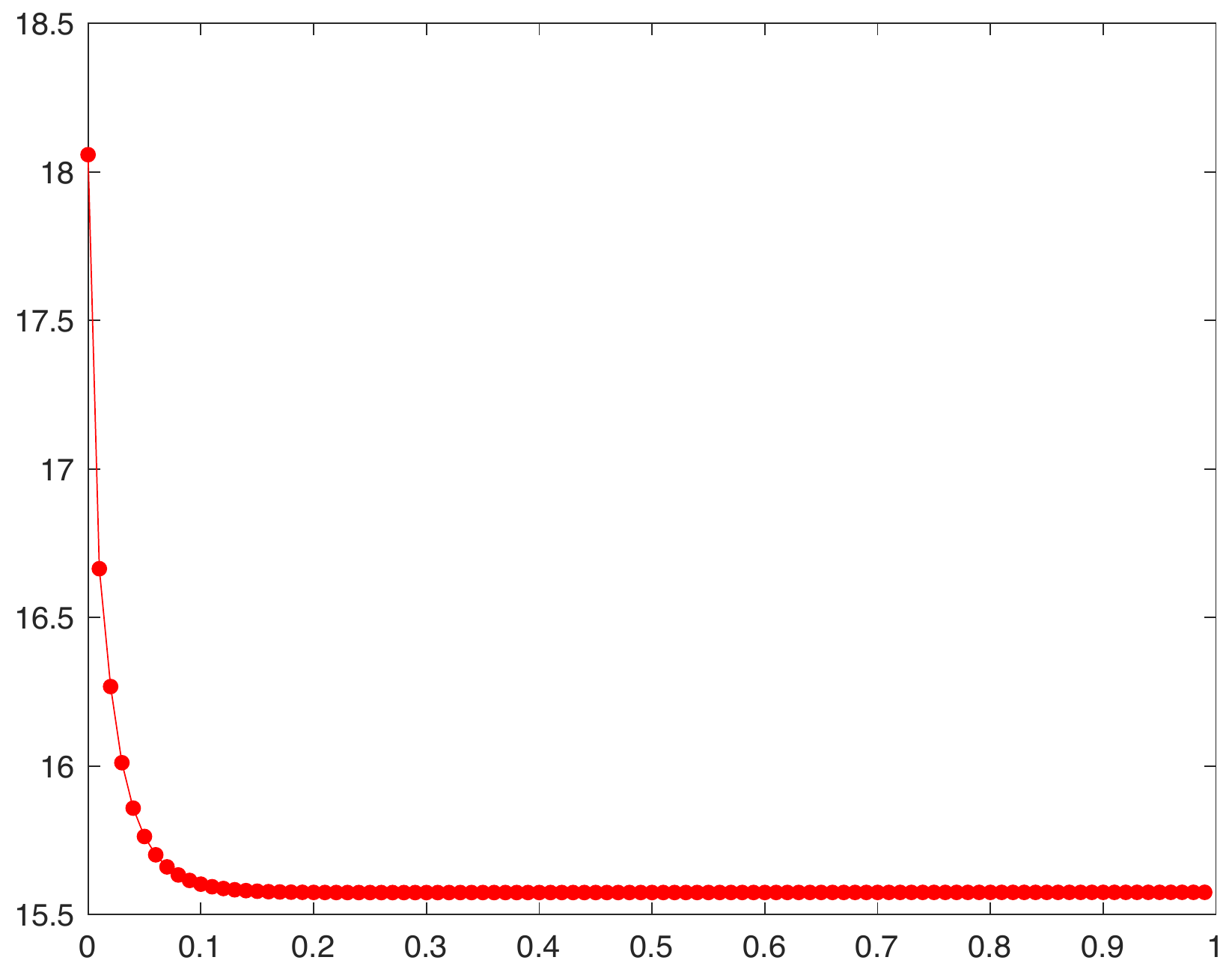}
\end{center}
\vspace{0.42cm}
\centerline{(c) $I_h$, $\mathbb{M}_K = (k_K + \epsilon) I$}\
\end{minipage}
\hspace{2mm}
\begin{minipage}[t]{3in}
\begin{center}
\includegraphics[width=2.75in]{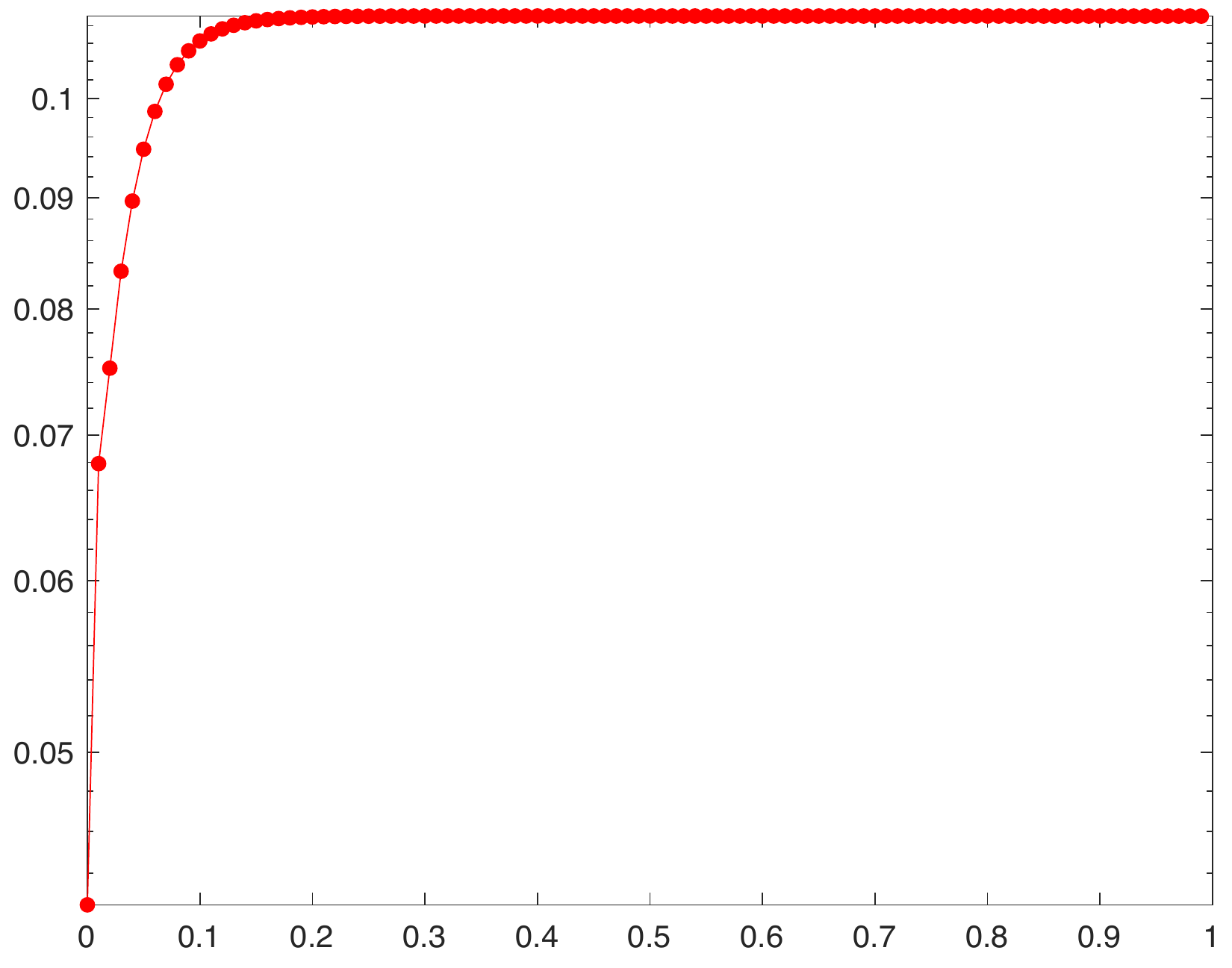}
\end{center}
\centerline{(d) $|K|_{\min}$, $\mathbb{M}_K = (k_K + \epsilon) I$}\
\end{minipage}
}
\caption{Example~\ref{sin2d}. $I_h$ and $K_{\min}$ plotted as functions of $t$ for $\Phi(x,y) = 4\sin(x) - y$.
}
\label{figsin2d2}
\end{center}
\end{figure}

\end{exam}

\vspace{10pt}

\begin{exam}
\label{lemni}

As the final two-dimensional example, we generate adaptive meshes for the lemniscate defined by
\[
\Phi(x,y) = (x^2 + y^2)^2 - 4(x^2 - y^2).
\]
In this example adapt the mesh on the curve for both $N= 60$ and $N= 120$. In both situations, we fix the node $\V x_1 = (2,0)$.

From Fig. \ref{lemnifig} we see that for $N = 60$ the mesh adapts from a very nonuniform initial mesh (Fig. \ref{lemnifig}(a)) to a uniform final mesh (Fig. \ref{lemnifig}(b)) when considering the metric tensor corresponding to the Euclidean metric. The nodes are equidistant apart while remaining on the curve. This improvement in uniformity can be further supported by the equidistribution quality measure which improves from 2.083287 for the initial mesh to 1.002549 for the final mesh.    

We see a similar result when the curvature-based metric tensor is used (Fig. \ref{lemnifig}(c)). A higher concentration of nodes occurs in the circular regions with larger curvature compared to the cross section which has smaller curvature (i.e., the linear regions).  It is not a significant difference in concentration but this is consistent with the fact that the curvature of the lemniscate is close to but not exactly constant.  The equidistribution quality measure improves from 3.364232 to 1.001011 indicating that the final mesh is much more uniform with respect to the curvature-based $\mathbb{M}_K$
than the initial mesh.

\begin{figure}[h]
\begin{center}
\hbox{
\begin{minipage}[t]{2in}
\begin{center}
\includegraphics[width=2in]{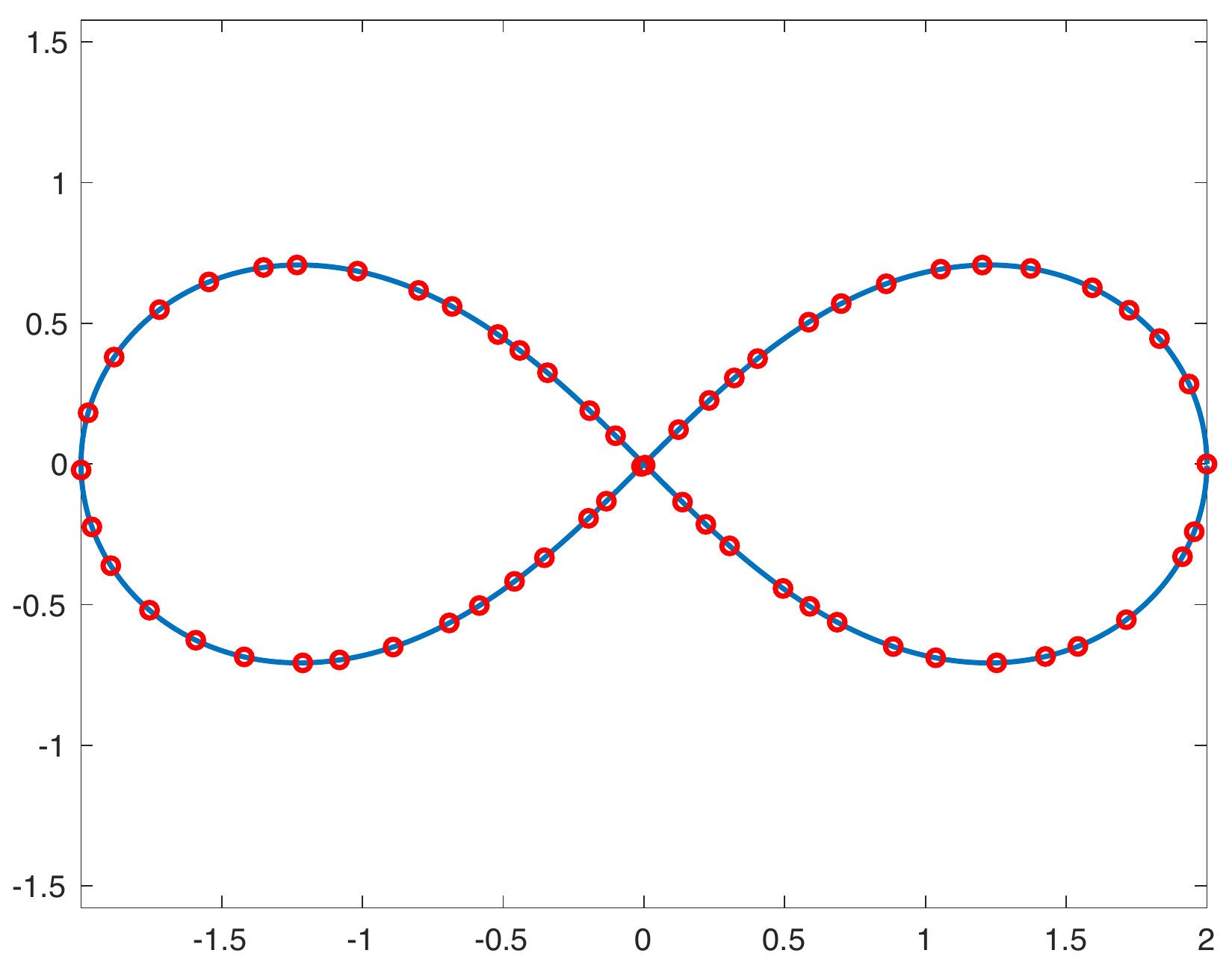}
\end{center}
\vspace{0.42cm}
\centerline{(a) Initial Mesh}\
\end{minipage}
\hspace{2mm}
\begin{minipage}[t]{2in}
\begin{center}
\includegraphics[width=2in]{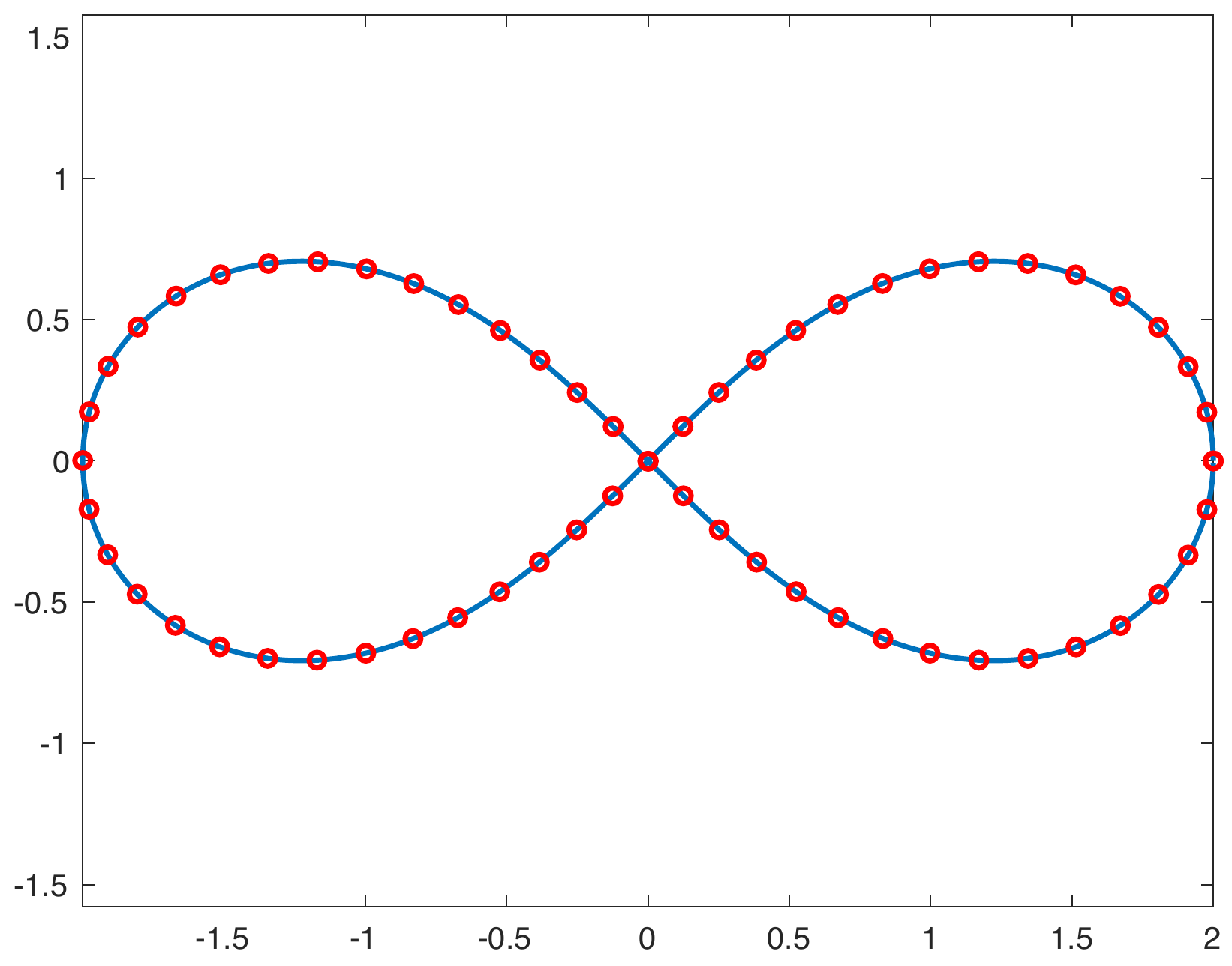}
\end{center}
\centerline{(b) Final Mesh, $\mathbb{M}_K = I$}\
\end{minipage}
\hspace{2mm}
\begin{minipage}[t]{2in}
\begin{center}
\includegraphics[width=2in]{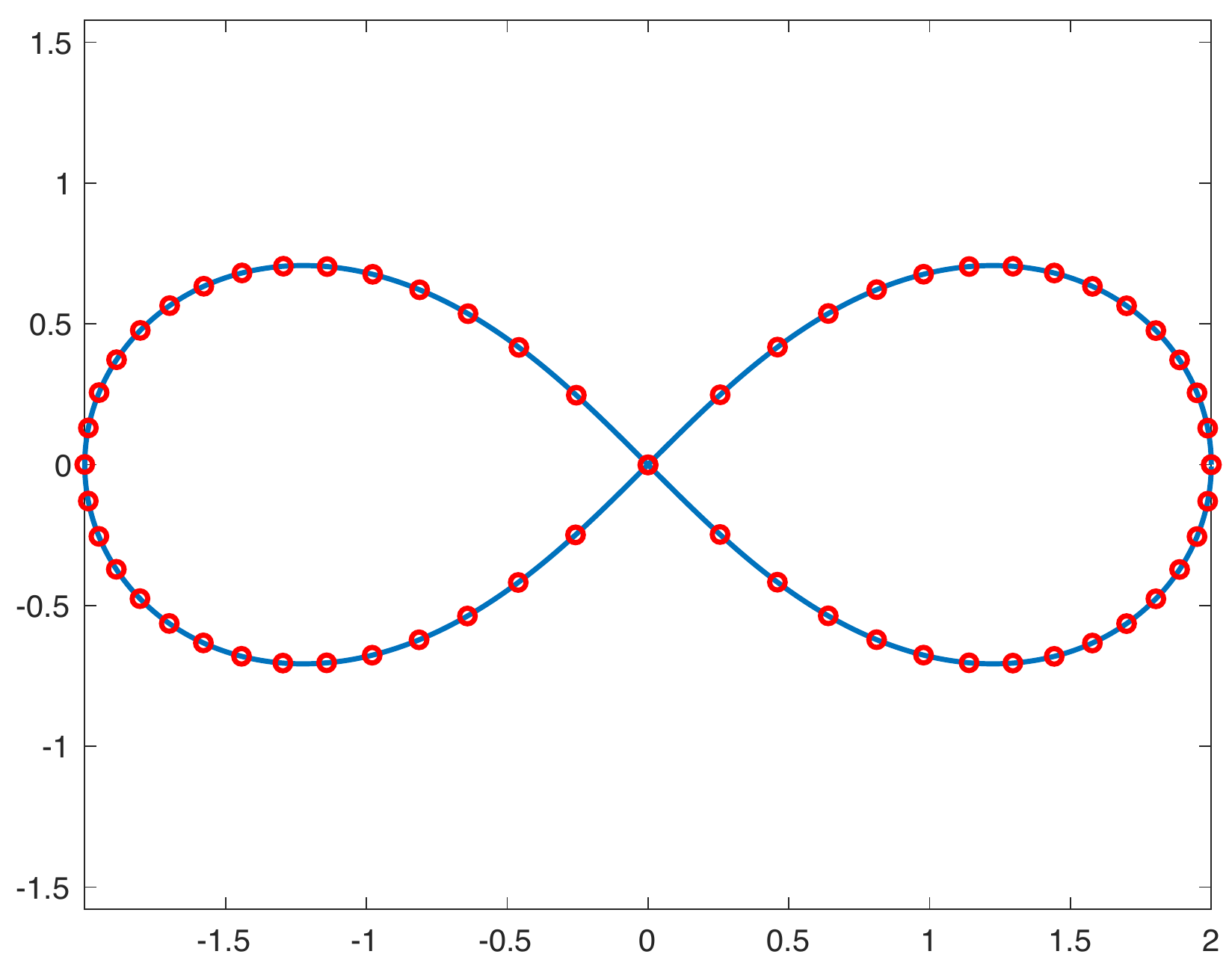}
\end{center}
\centerline{(c) Final Mesh, $\mathbb{M}_K = (k_K+\epsilon)I$}\
\end{minipage}
}
\caption{Example~\ref{lemni}. Meshes of $N = 60$ are obtained for the lemniscate $\Phi(x,y) = (x^2 + y^2)^2 - 4(x^2 - y^2)$.}
\label{lemnifig}
\end{center}
\end{figure}

\begin{figure}[h]
\begin{center}
\hbox{
\begin{minipage}[t]{2in}
\begin{center}
\includegraphics[width=2in]{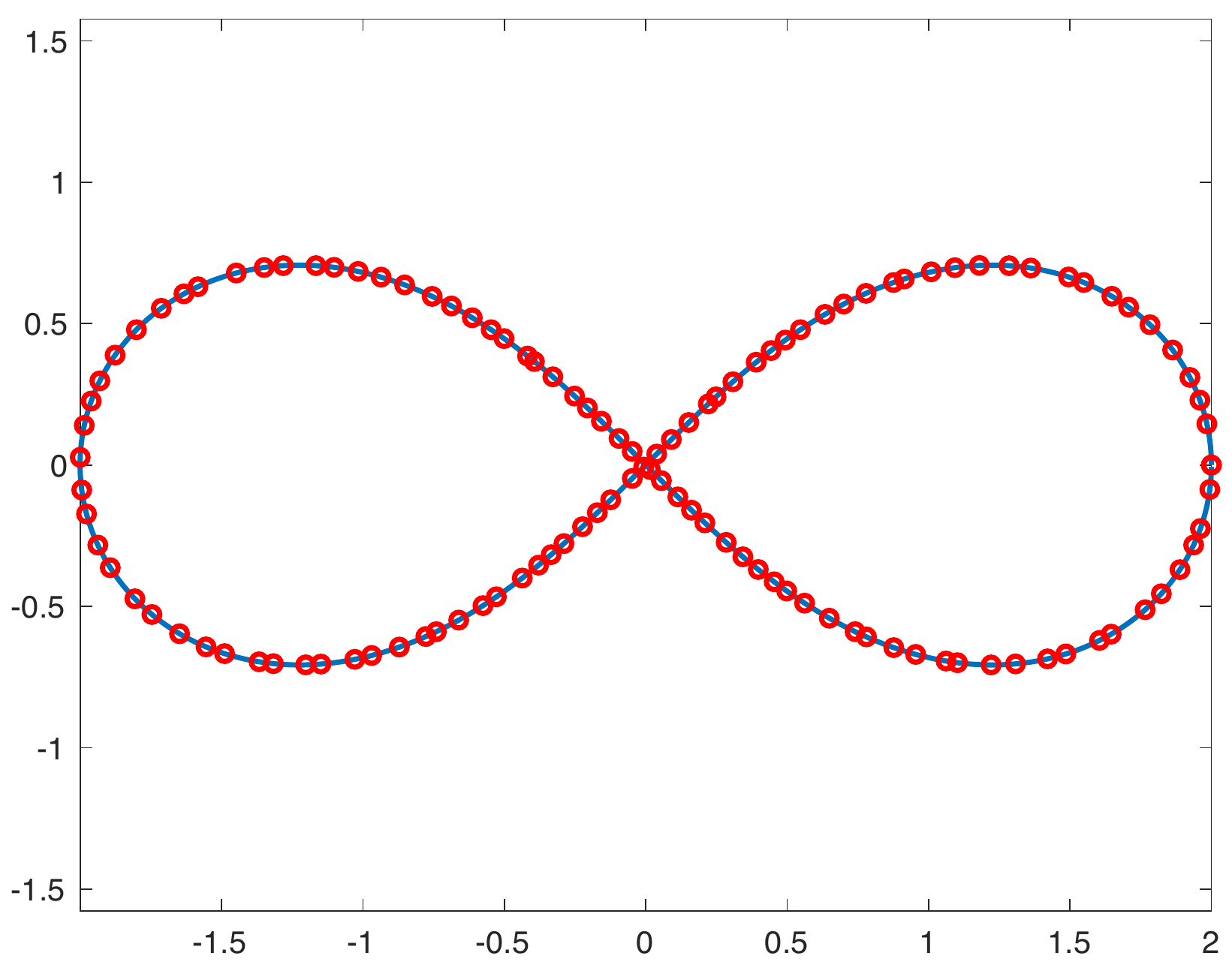}
\end{center}
\vspace{0.42cm}
\centerline{(a) Initial Mesh}\
\end{minipage}
\hspace{2mm}
\begin{minipage}[t]{2in}
\begin{center}
\includegraphics[width=2in]{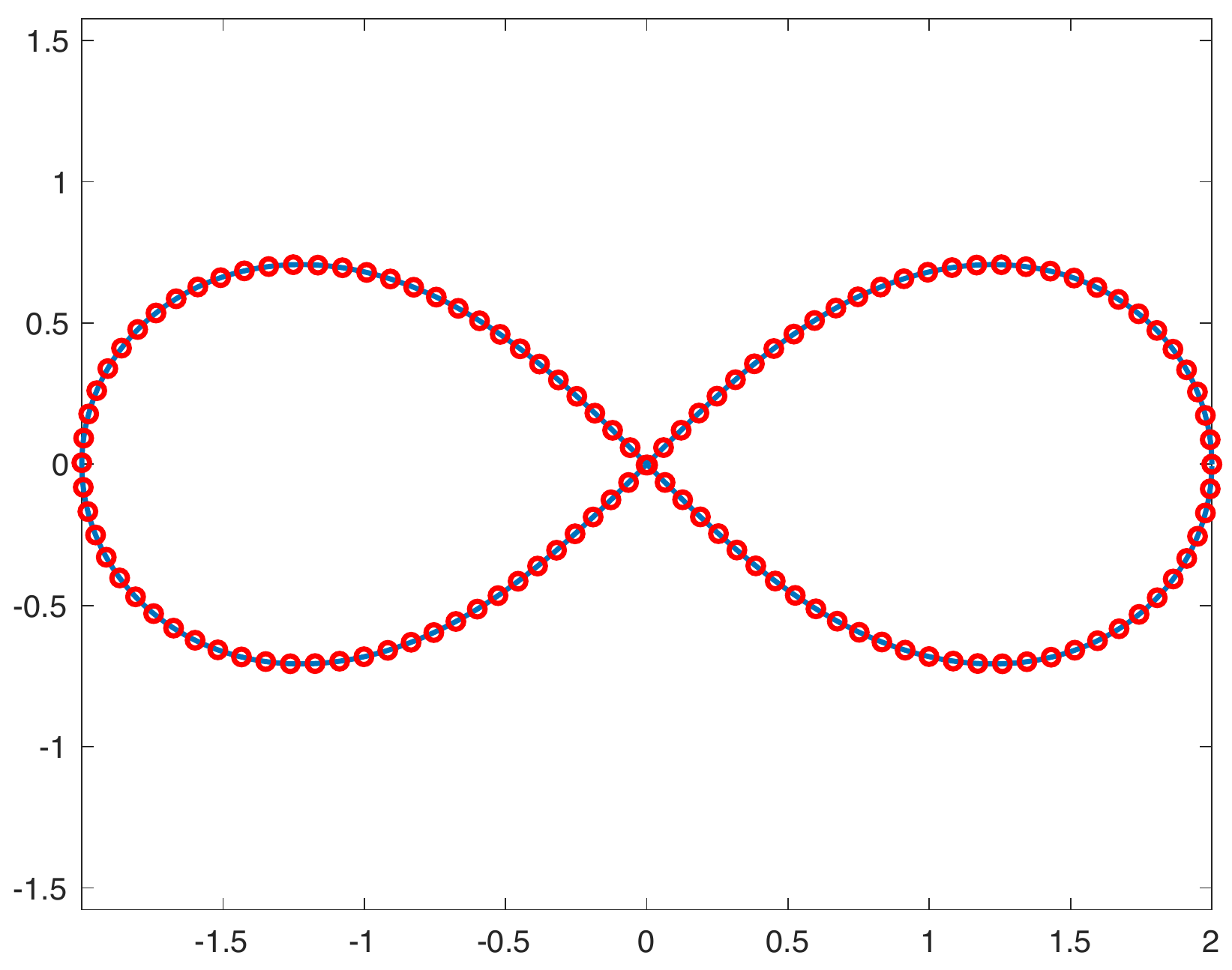}
\end{center}
\centerline{(b) Final Mesh, $\mathbb{M}_K = I$}\
\end{minipage}
\hspace{2mm}
\begin{minipage}[t]{2in}
\begin{center}
\includegraphics[width=2in]{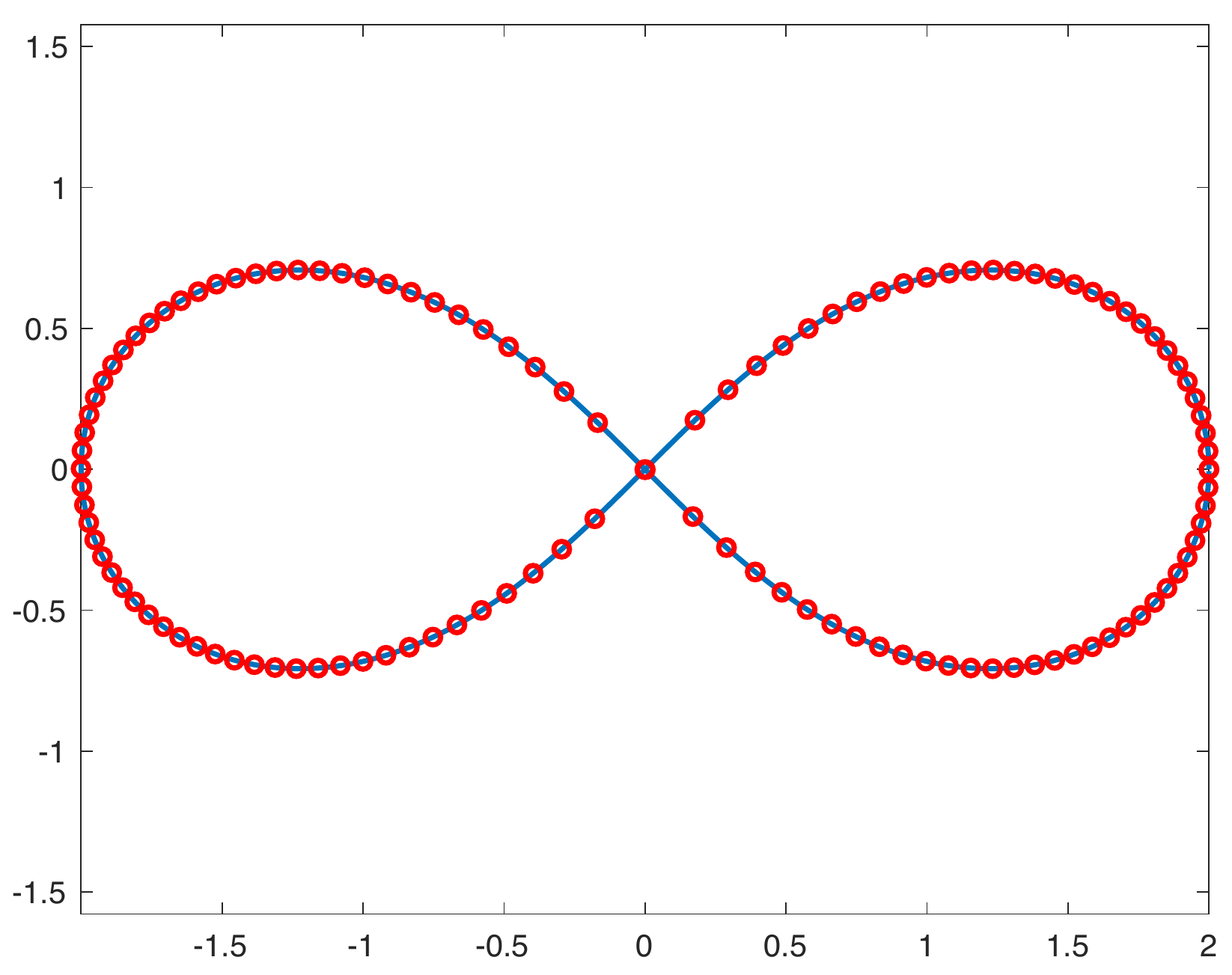}
\end{center}
\centerline{(c) Final Mesh, $\mathbb{M}_K = (k_K+\epsilon)I$}\
\end{minipage}
}
\caption{Example~\ref{lemni}. Meshes of $N = 120$ are obtained for the lemniscate $\Phi(x,y) = (x^2 + y^2)^2 - 4(x^2 - y^2)$.}
\label{lem2}
\end{center}
\end{figure}

For $N = 120$, Fig. \ref{lem2} shows similar findings.  When considering the Euclidean metric, we see the mesh, Fig. \ref{lem2}(a), is very nonuniform initially and adapts to a equidistant spacing of the nodes along the curve. This is further supported in the quality measures for which the equidistribution measure improves from 2.552134 to 1.002167 indicating that the final mesh is close to satisfying the equidistribution condition.

The curvature-based metric tensor results in a similar adaptation as before.  Fig. \ref{lem2}(c) shows the final mesh has adapted in such a way where there is a higher concentration of nodes in those regions of the curve with larger curvature, i.e., circular regions.  Comparatively, there are fewer nodes in the cross section which has smaller curvature.  The difference in concentration can be clearly seen in Fig. \ref{lem2}(c) with $N = 120$ nodes.  The adaptation is consistent with the curvature of the lemniscate, which is close to but not exactly constant. This improvement in uniformity can be further supported by the equidistribution quality measure which improves from 8.023253 for the initial mesh to 1.001855 for the final mesh.

\end{exam}

\vspace{10pt}

\begin{exam}
\label{torus}
Let us now consider surfaces in $\mathbb{R}^3$. In this first example, we consider adaptive meshes for the torus defined by
$$\Phi(x,y,z) = \left(2-\sqrt{x^2+y^2}\right)^2+z^2-1,$$
where $x,y \in [-3,3]$, and $z\in [-1,1]$. We take $N= 3200$.

Fig. \ref{figtorus} shows the meshes for this example in two different views.  Studying Fig. \ref{figtorus}(a), the initial mesh, and Fig. \ref{figtorus}(b), the final mesh with $\mathbb{M}_K = I$, we can see that the final mesh provides a more uniform distribution of the nodes. That is consistent with the use of the metric tensor $\mathbb{M} = I$ whose goal is to make the mesh as uniform as possible in the Euclidean norm. This can also be confirmed from the equidistribution and alignment quality measures. The equidistribution measure for the initial mesh is 15.50150 and for the final mesh 1.332488. Similarly, the initial alignment quality measure is 30.63276 compared to that of the final mesh which is 1.920701.

For the curvature-based metric tensor, we see similar results to that of the Euclidean metric. That is, the final mesh for the curvature-based metric tensor, Fig. \ref{figtorus}(c), looks identical to the final mesh for the Euclidean metric, Fig. \ref{figtorus}(b). This is because the absolute value of the mean curvature of a torus is close to constant and hence, the elements do not concentrate in any particular region of the surface. 

\begin{figure}[h]
\begin{center}
\hbox{
\begin{minipage}[t]{2in}
\begin{center}
\includegraphics[width=2in]{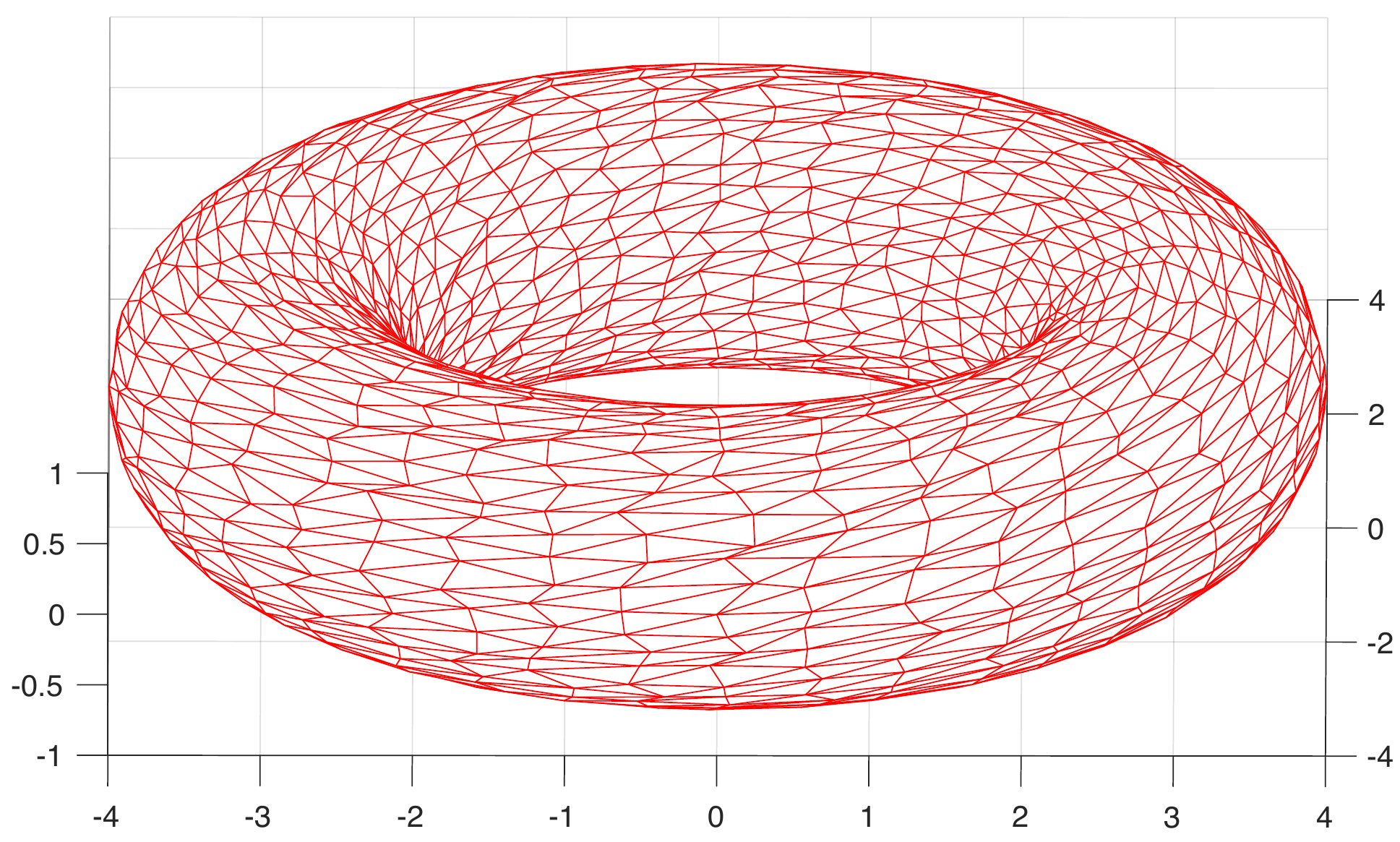}
\end{center}
\vspace{0.42cm}
\centerline{(a) Initial Mesh}\
\end{minipage}
\hspace{2mm}
\begin{minipage}[t]{2in}
\begin{center}
\includegraphics[width=2in]{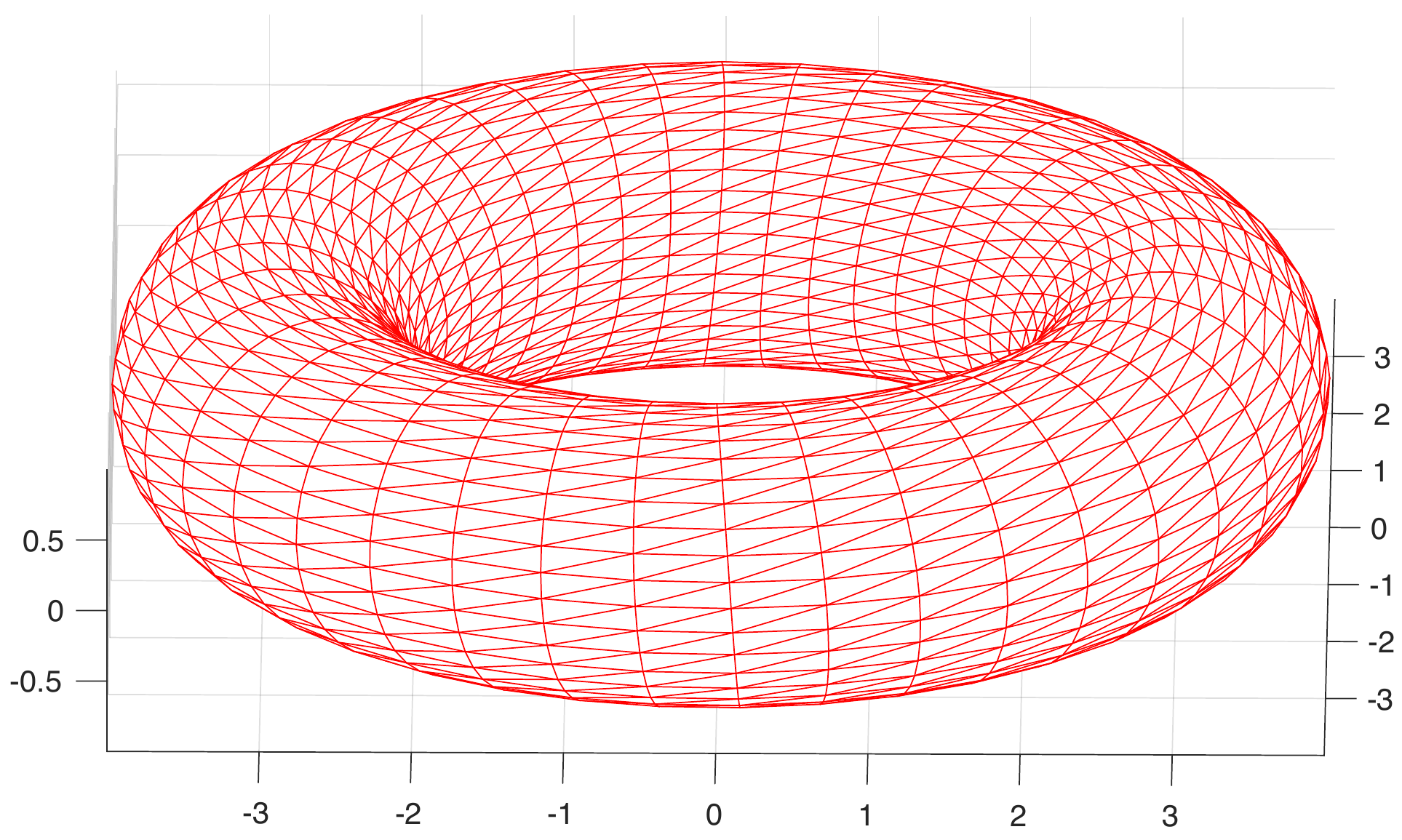}
\end{center}
\centerline{(b) Final Mesh, $\mathbb{M}_K = I$}\
\end{minipage}
\hspace{2mm}
\begin{minipage}[t]{2in}
\begin{center}
\includegraphics[width=2in]{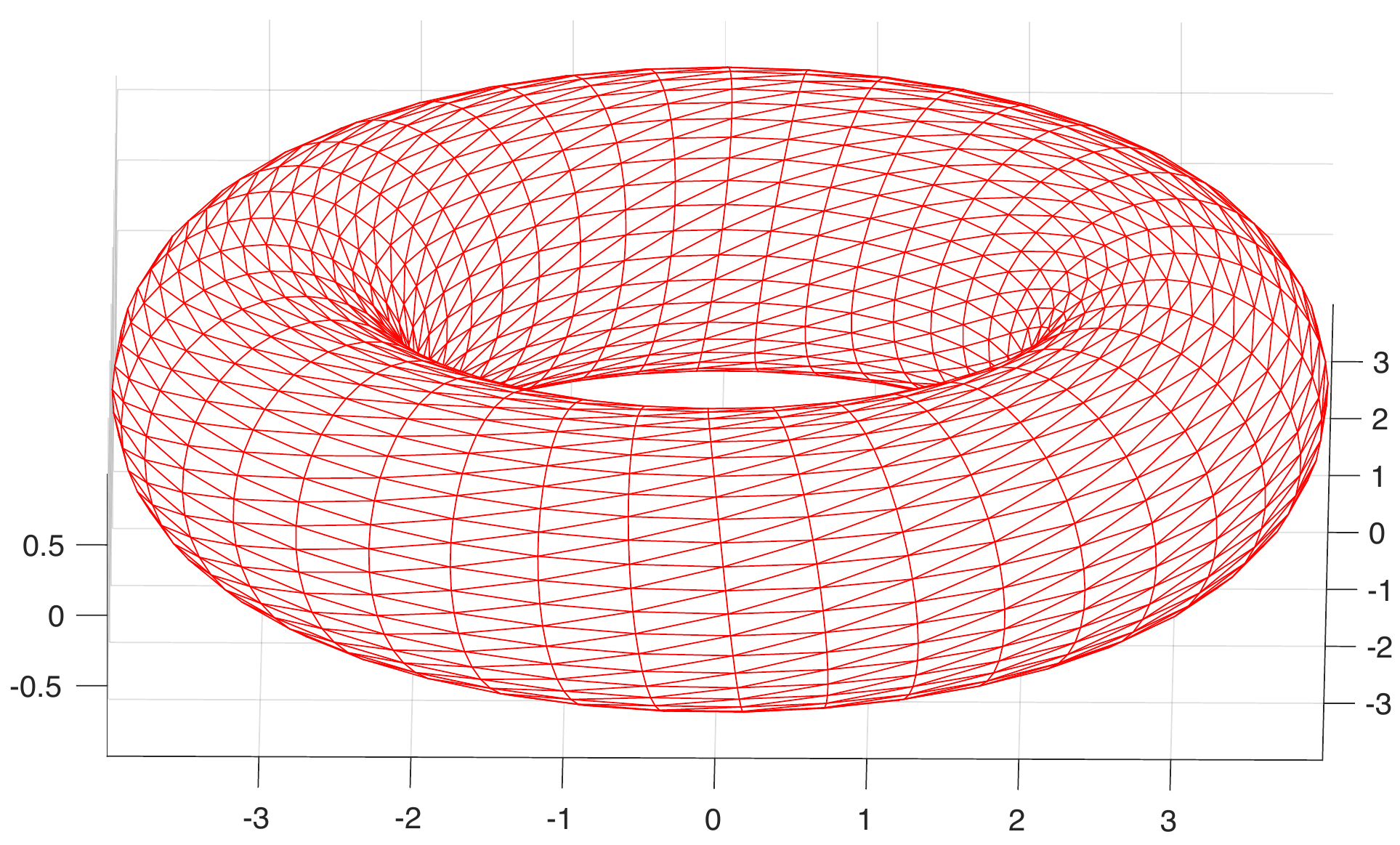}
\end{center}
\centerline{(c) Final Mesh, $\mathbb{M}_K = (k_K+\epsilon)I$}\
\end{minipage}
}
\hbox{
\begin{minipage}[t]{2in}
\begin{center}
\includegraphics[width=2in]{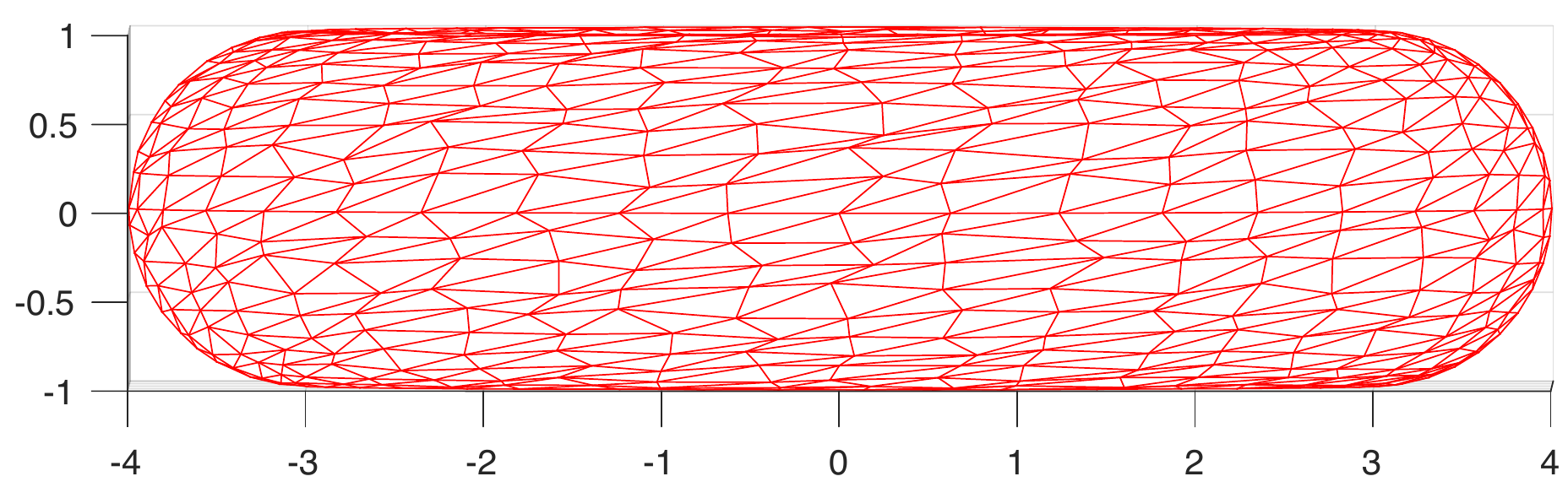}
\end{center}
\centerline{(d) side view of (a)}
\end{minipage}
\hspace{2mm}
\begin{minipage}[t]{2in}
\begin{center}
\includegraphics[width=2in]{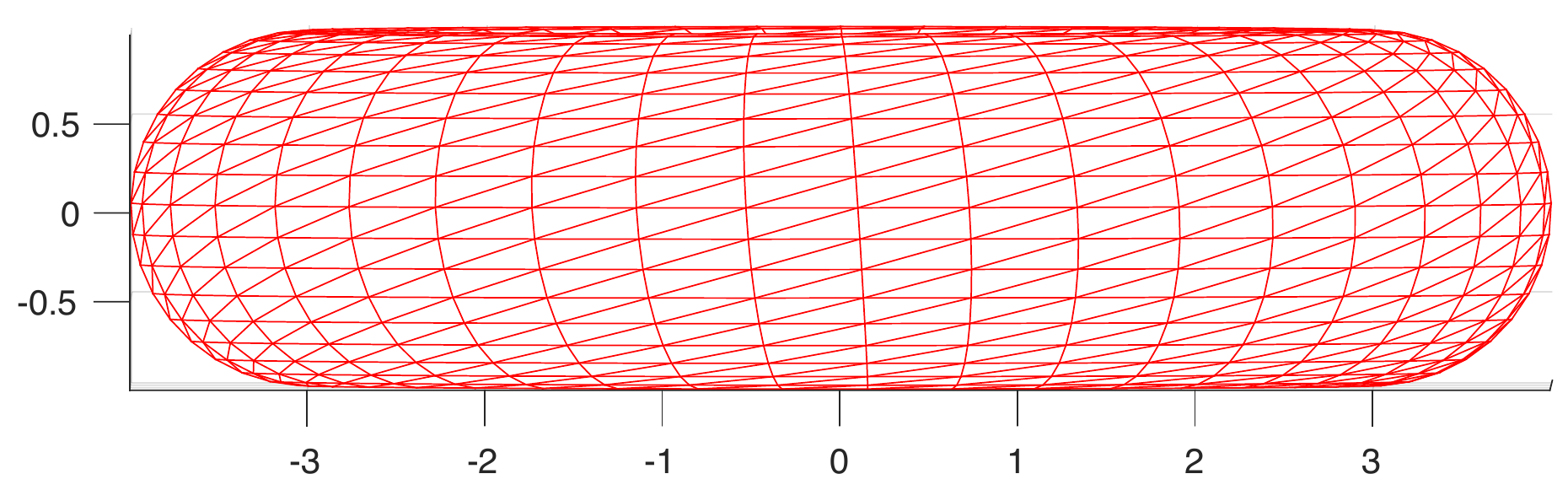}
\end{center}
\centerline{(e) side view of (b)}
\end{minipage}
\hspace{2mm}
\begin{minipage}[t]{2in}
\begin{center}
\includegraphics[width=2in]{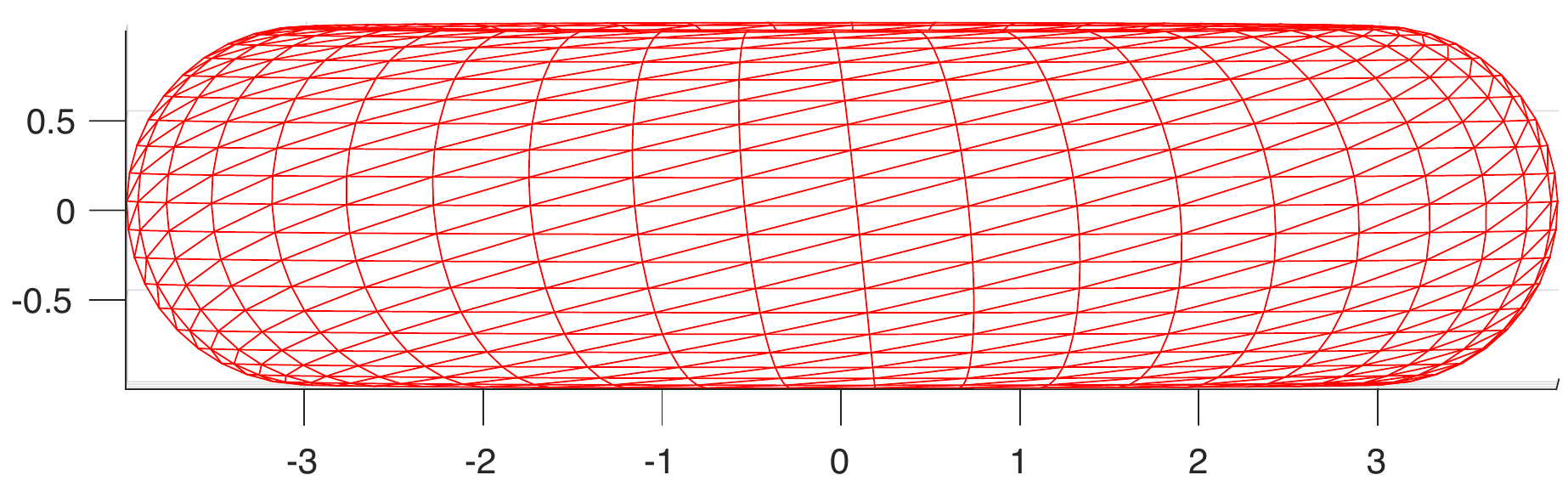}
\end{center}
\centerline{(f) side view of (c)}
\end{minipage}
}
\caption{Example~\ref{torus}. Meshes of $N = 3200$ are obtained for the surface $\Phi(x,y,z) = (2-\sqrt{x^2+y^2})^2+z^2-1$.
}
\label{figtorus}
\end{center}
\end{figure}
\end{exam}

\vspace{10pt}

\begin{exam}
\label{cylinder}
The second three-dimensional example is the cylinder defined by
\[
\Phi(x,y,z) = x^2+y^2-1,
\]
where $z \in [-2,2]$. For this example we take $N = 3200$. Two boundary nodes were fixed, $\V x_1 = ( 0,1,-2 )$ and $\V x_1 = (0,1,2 )$, but the remaining boundary nodes were allowed to slide along the boundary. Although the cylinder has constant curvature like Example \ref{torus}, this example shows the adaptation on a surface with a boundary. 

Fig. \ref{figcylinder} shows the adaptive meshes for the cylinder in two different views. For both $\mathbb{M}_K = I$ and $\mathbb{M}_K = (k_K + \epsilon)I$, the mesh becomes much more uniform and identical. This is consistent with the constant curvature of the cylinder hence the nodes do not concentrate in any specific region of the surface. The equidistribution quality measure improves from 19.07656 to 1.054857 and the alignment quality measure from 23.35403 to 1.192268. The fact that the final quality measures for both conditions are close to 1 indicates that the final meshes are close to satisfying conditions (\ref{equ}) and (\ref{ali}).

\begin{figure}[h]
\begin{center}
\hbox{
\begin{minipage}[t]{2in}
\begin{center}
\includegraphics[width=2in]{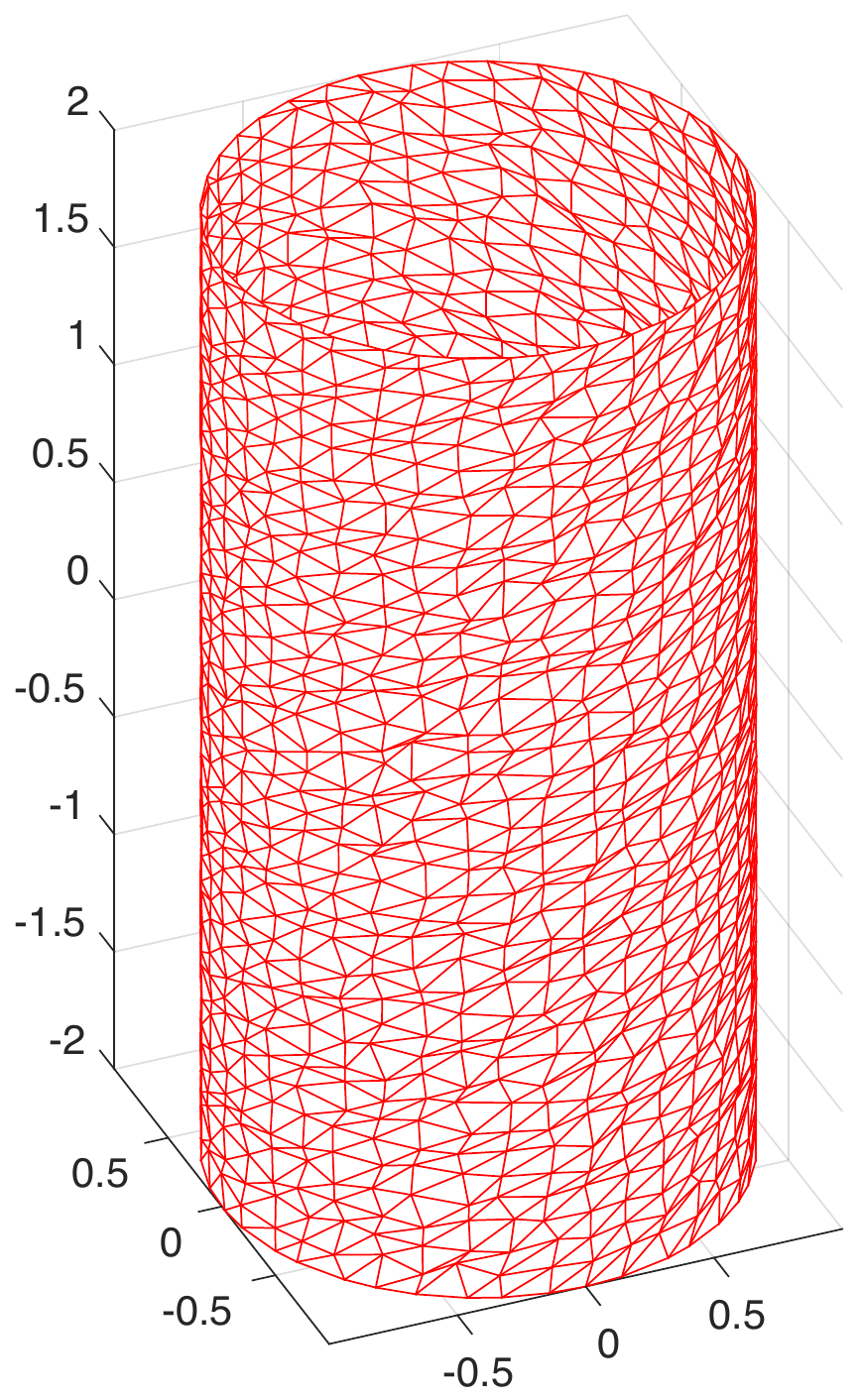}
\end{center}
\vspace{0.42cm}
\centerline{(a) Initial Mesh}\
\end{minipage}
\hspace{2mm}
\begin{minipage}[t]{2in}
\begin{center}
\includegraphics[width=2in]{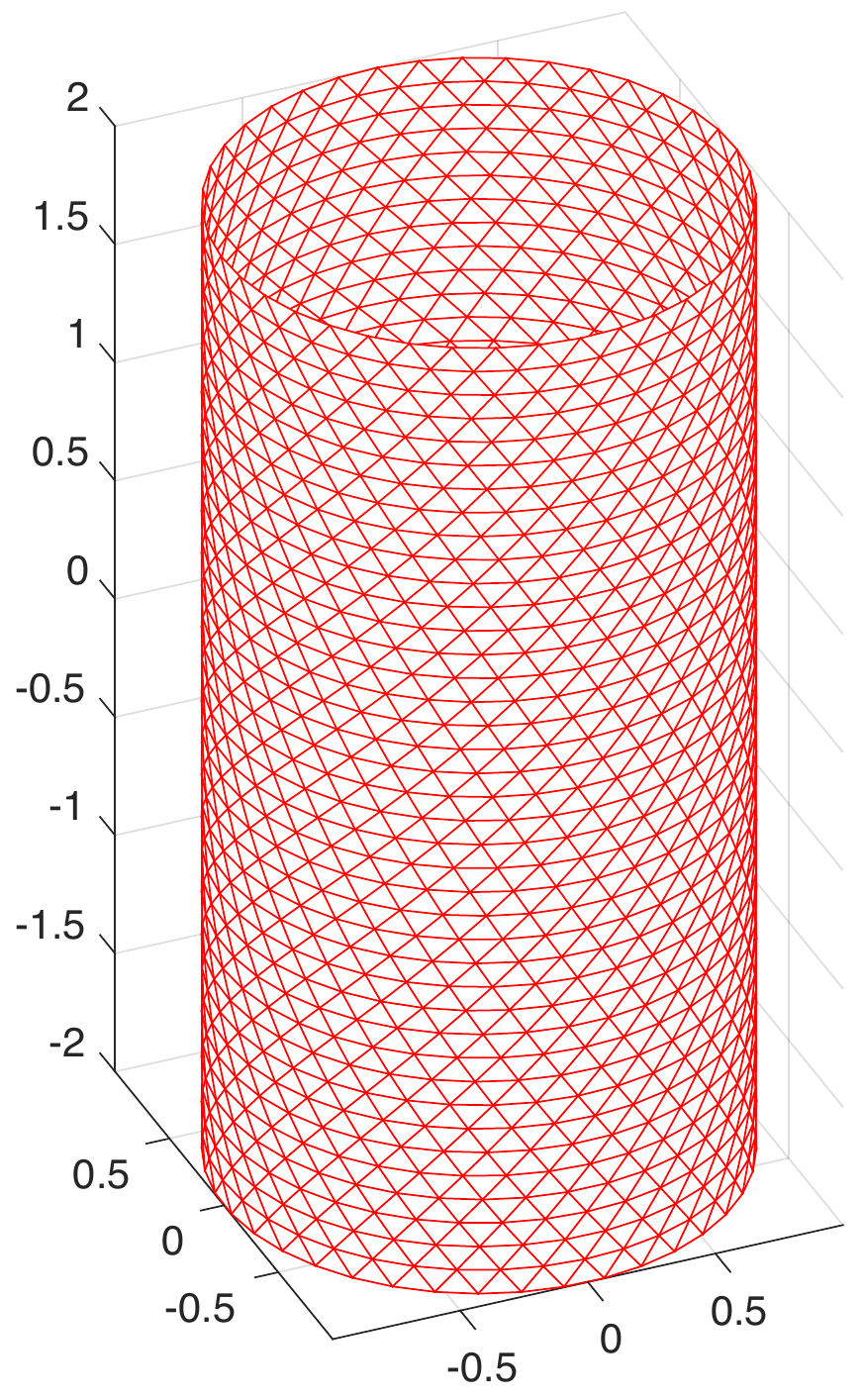}
\end{center}
\centerline{(b) Final Mesh $\mathbb{M}_K = I$}\
\end{minipage}
\hspace{2mm}
\begin{minipage}[t]{2in}
\begin{center}
\includegraphics[width=2in]{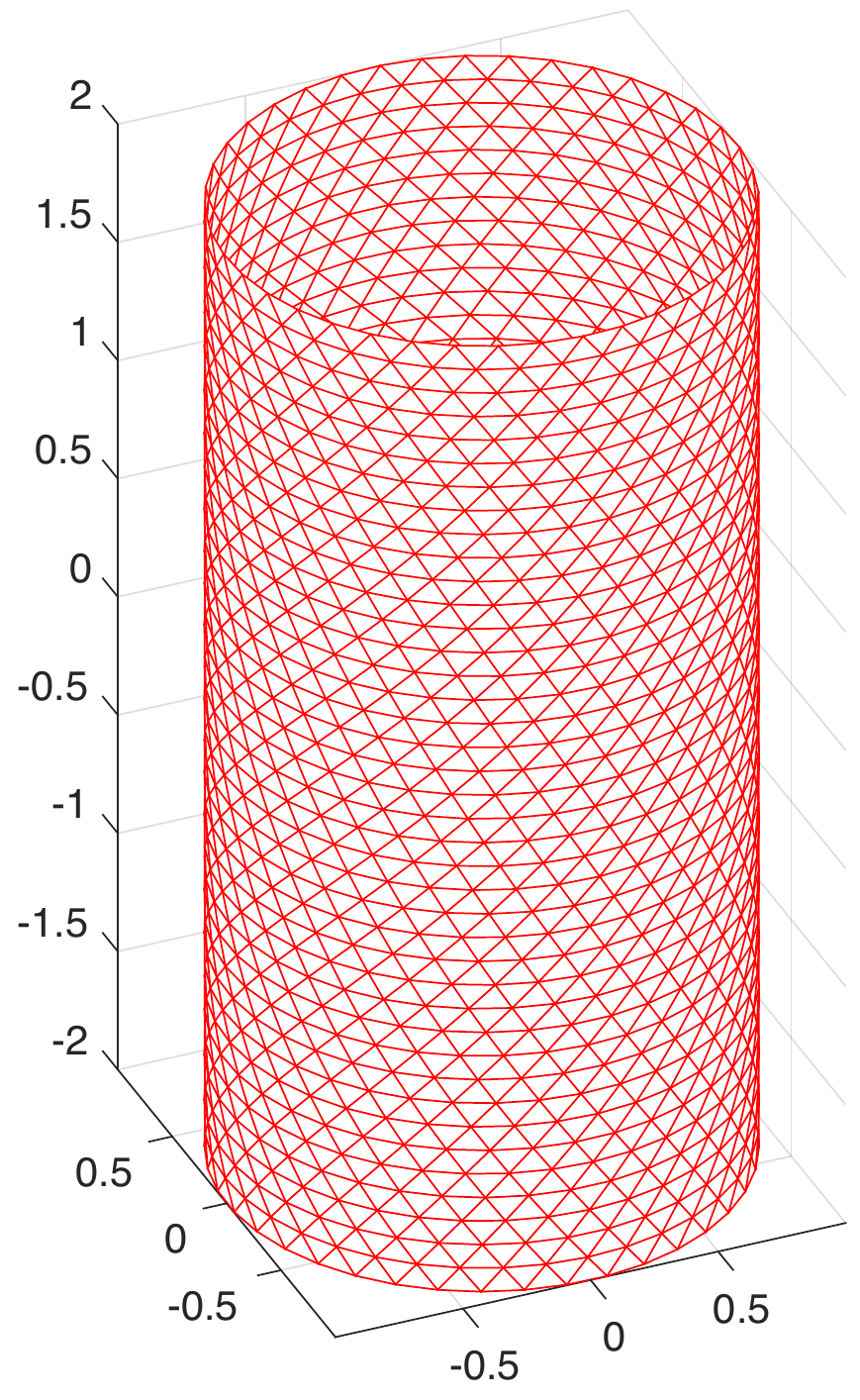}
\end{center}
\centerline{(c) Final Mesh $\mathbb{M}_K = (k_K+\epsilon)I$}\
\end{minipage}
}
\hbox{
\begin{minipage}[t]{2in}
\begin{center}
\includegraphics[width=2in]{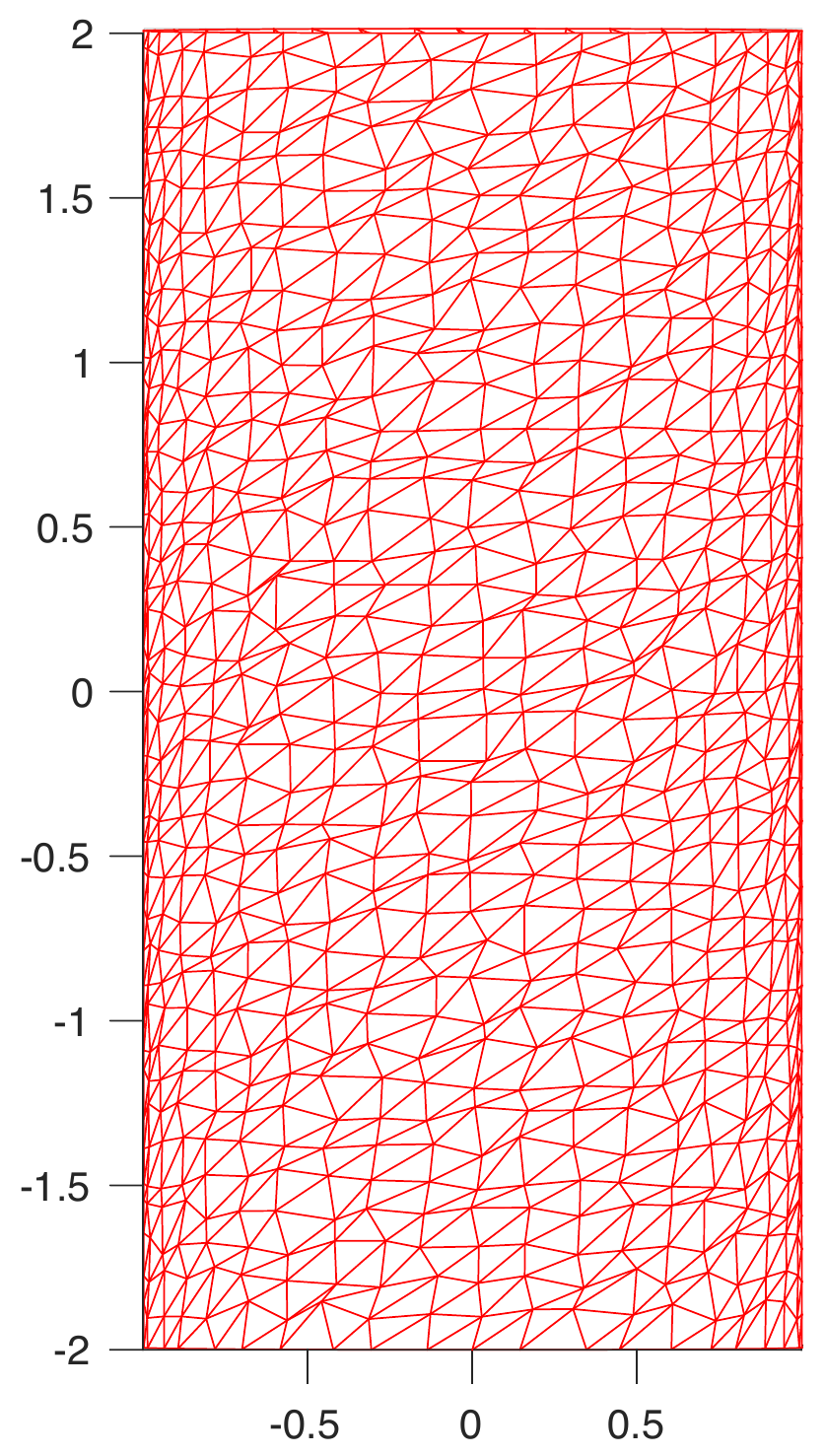}
\end{center}
\centerline{(d) side view of (a)}
\end{minipage}
\hspace{2mm}
\begin{minipage}[t]{2in}
\begin{center}
\includegraphics[width=2in]{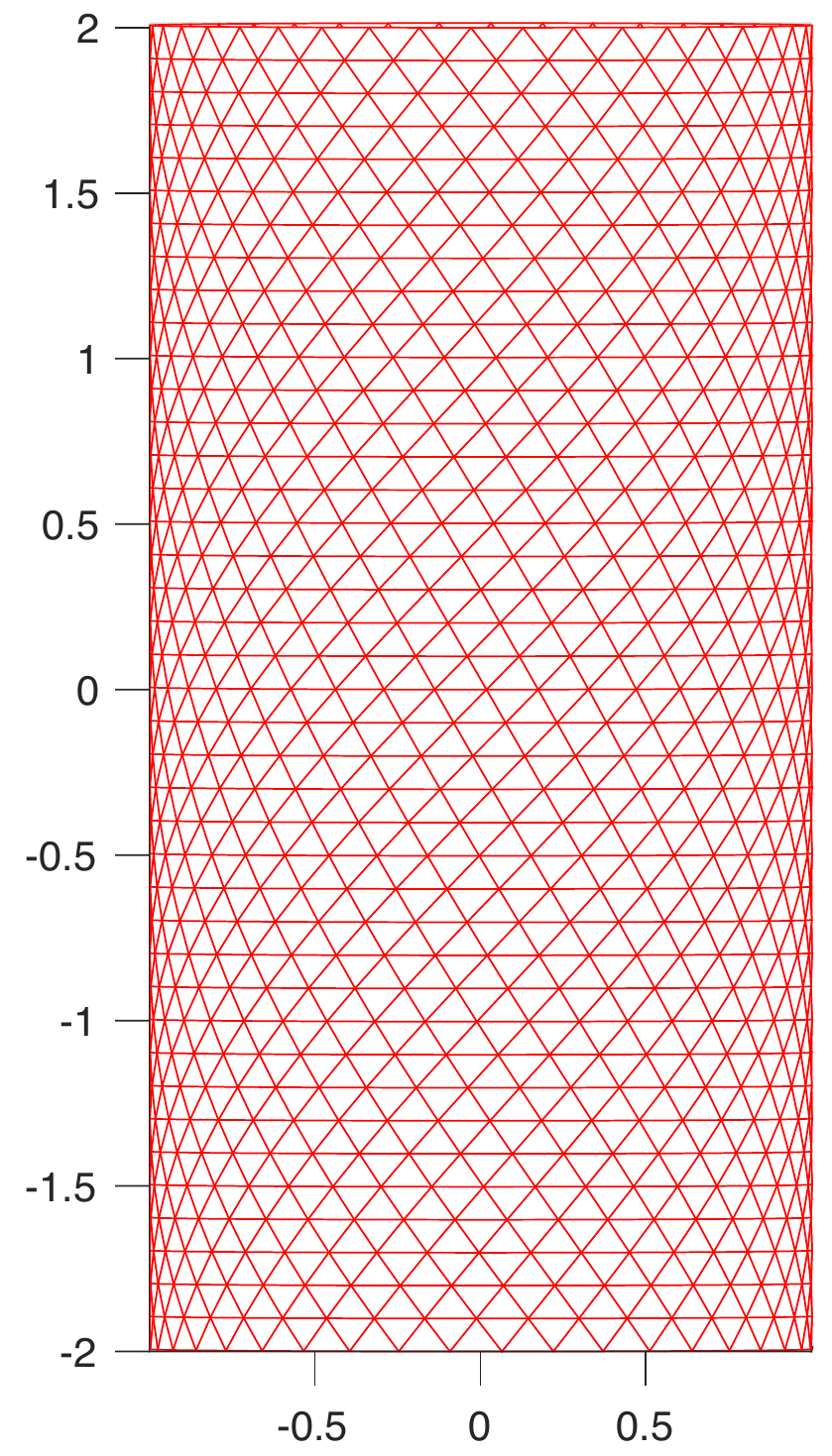}
\end{center}
\centerline{(e) side view of (b)}
\end{minipage}
\hspace{2mm}
\begin{minipage}[t]{2in}
\begin{center}
\includegraphics[width=2in]{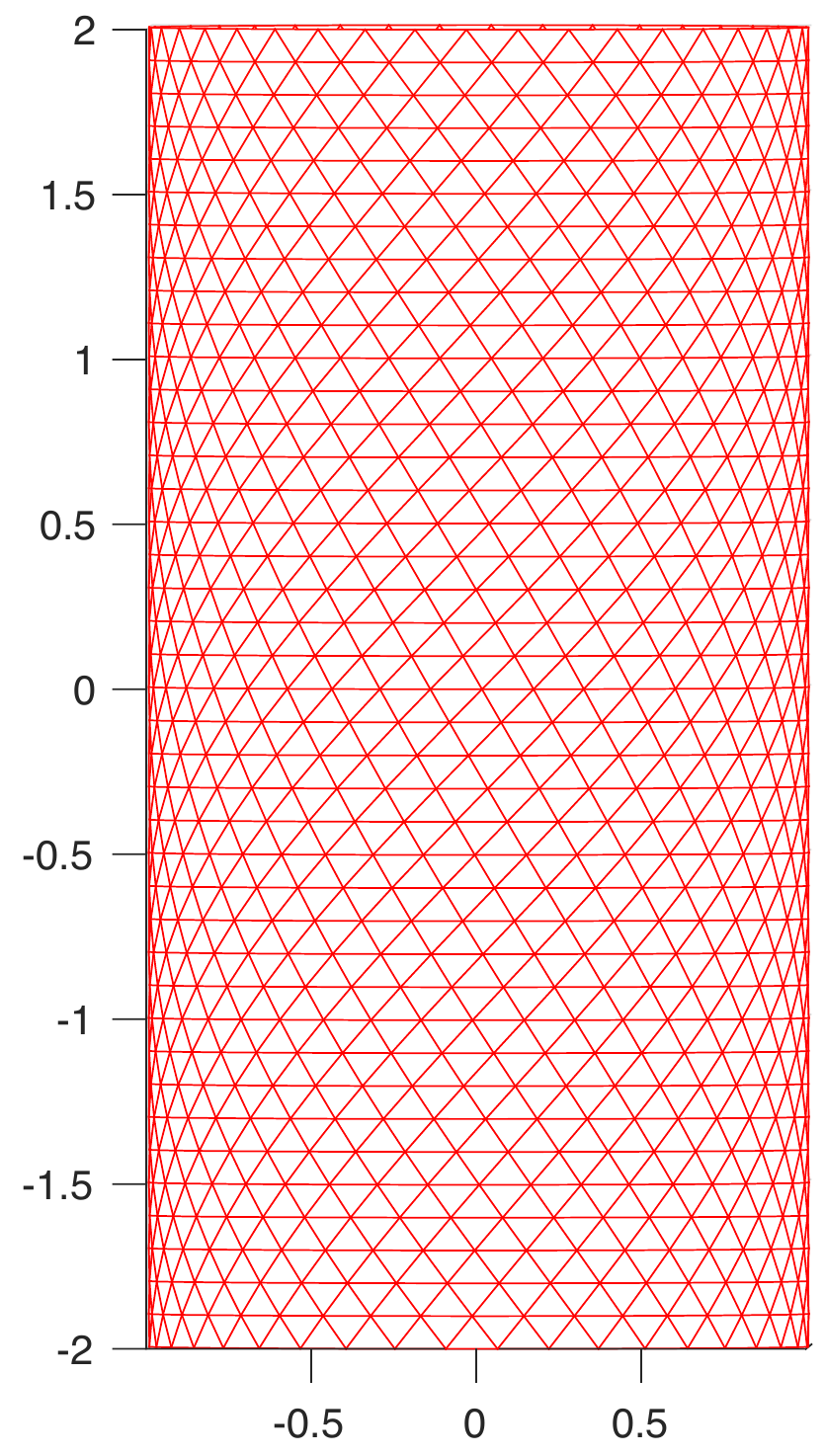}
\end{center}
\centerline{(f) side view of (c)}
\end{minipage}
}
\caption{Example~\ref{cylinder}. Meshes of $N = 3200$ are plotted for $\Phi(x,y,z) = x^2+y^2-1$.
}
\label{figcylinder}
\end{center}
\end{figure}

\end{exam}

\vspace{10pt}

\begin{exam}
\label{sin3d}

Our next example is the sine surface in three dimensions defined by
\[
\Phi(x,y,z) = \sin(x+y) - z.
\]
For this example we take $N= 3200$ and fix the boundary nodes.

Fig. \ref{figsin3d} shows the adaptive meshes for this examples in two different views.  It is clear in Fig. \ref{figsin3d}, when $\mathbb{M}_K = I$, the mesh becomes much more uniform with respect to the Euclidean metric from the initial mesh Fig. \ref{figsin3d}(a) to the final mesh Fig. \ref{figsin3d}(b). The top view of the surface, Fig. \ref{figsin3d}(d) and (e), further confirms this observation.  It is also supported by the improvement of the quality measures from $Q_{eq} = 4.234781$ and $Q_{ali}=6.643755$ to $Q_{eq} = 1.669880$ and $Q_{ali}=1.702617$. 

\begin{figure}[h]
\begin{center}
\hbox{
\begin{minipage}[t]{2in}
\begin{center}
\includegraphics[width=2in]{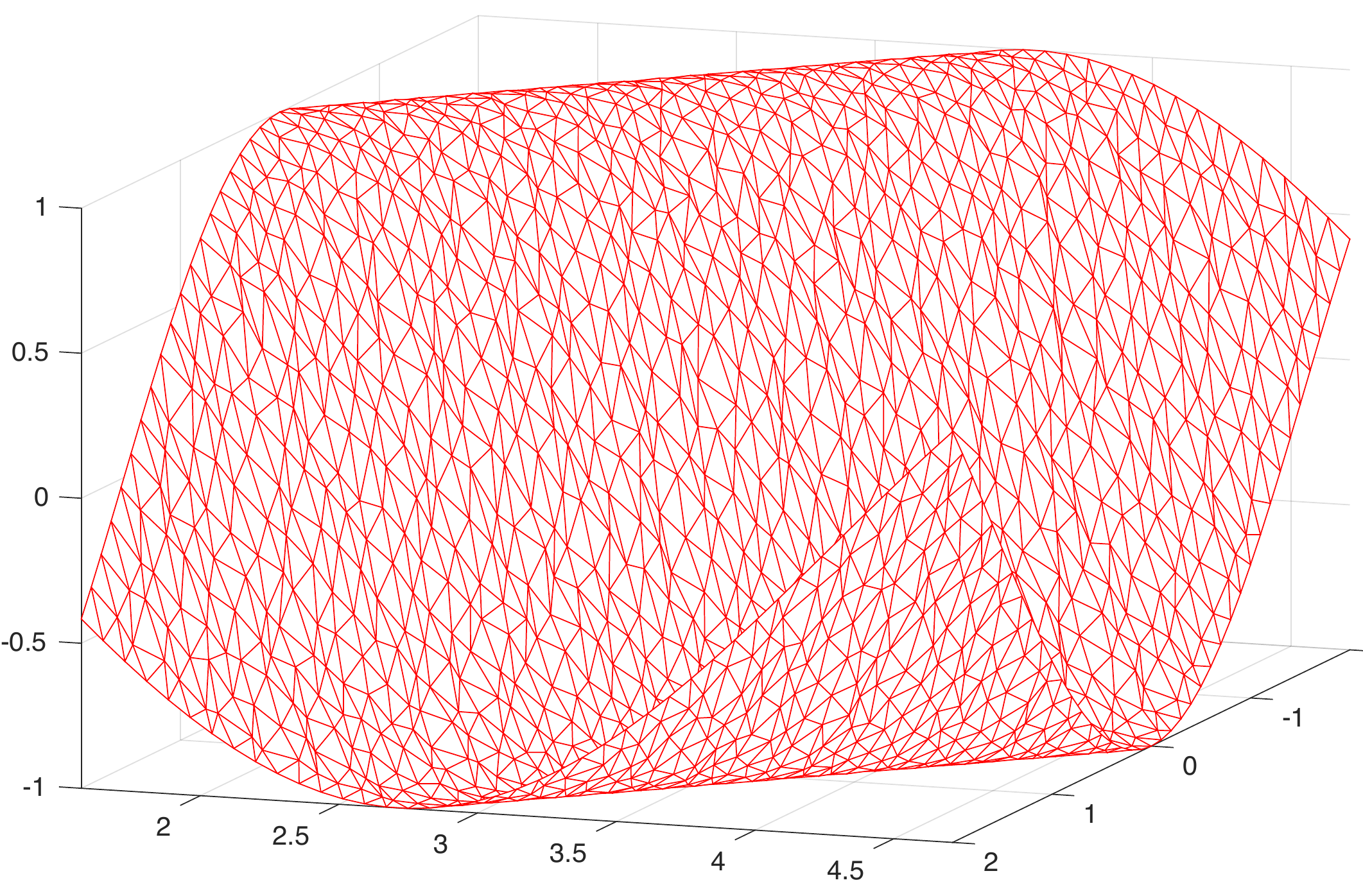}
\end{center}
\vspace{0.42cm}
\centerline{(a) Initial Mesh}
\end{minipage}
\hspace{2mm}
\begin{minipage}[t]{2in}
\begin{center}
\includegraphics[width=2in]{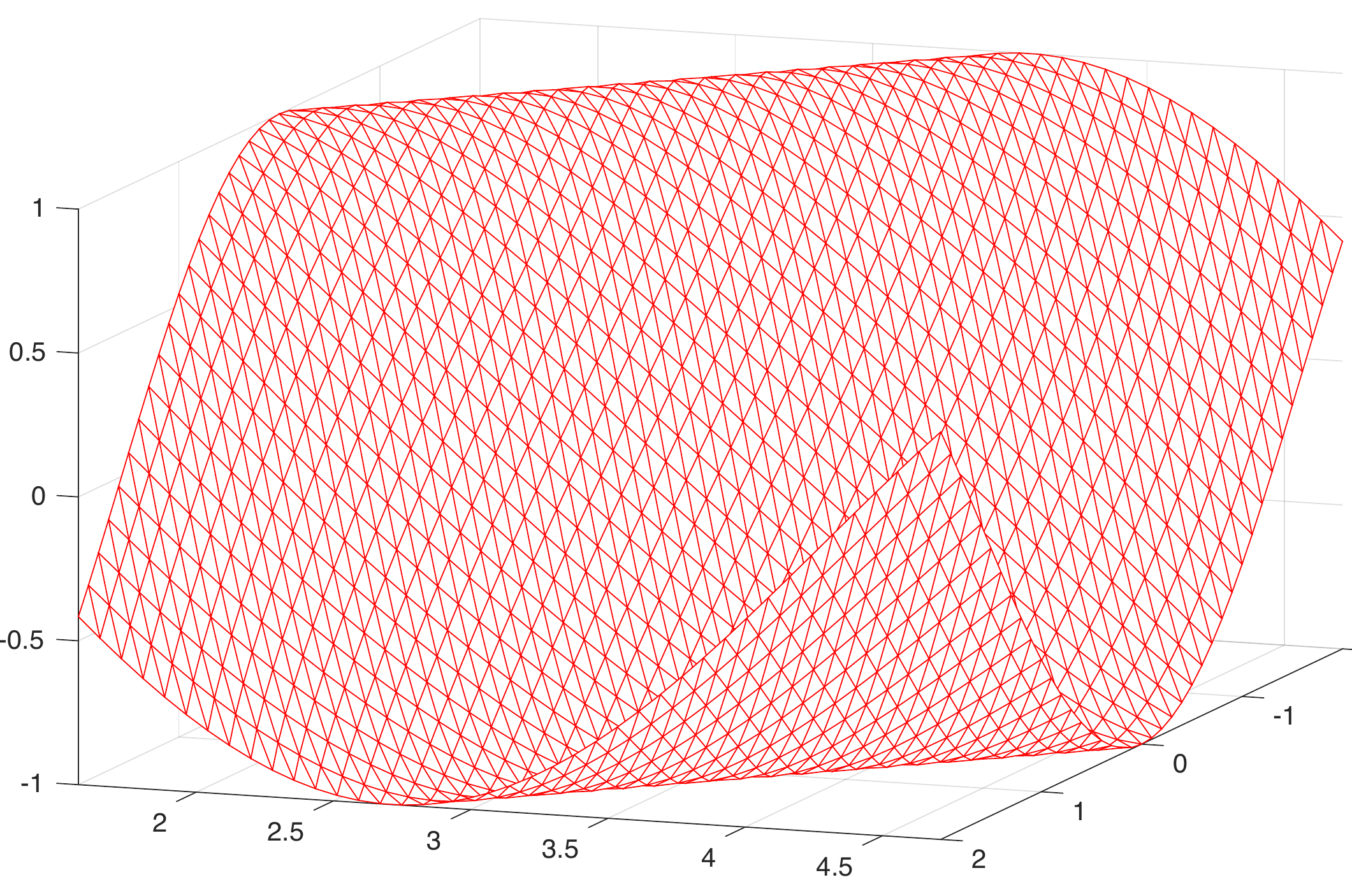}
\end{center}
\centerline{(b) Final Mesh, $\mathbb{M}_K = I$}
\end{minipage}
\hspace{2mm}
\begin{minipage}[t]{2in}
\begin{center}
\includegraphics[width=2in]{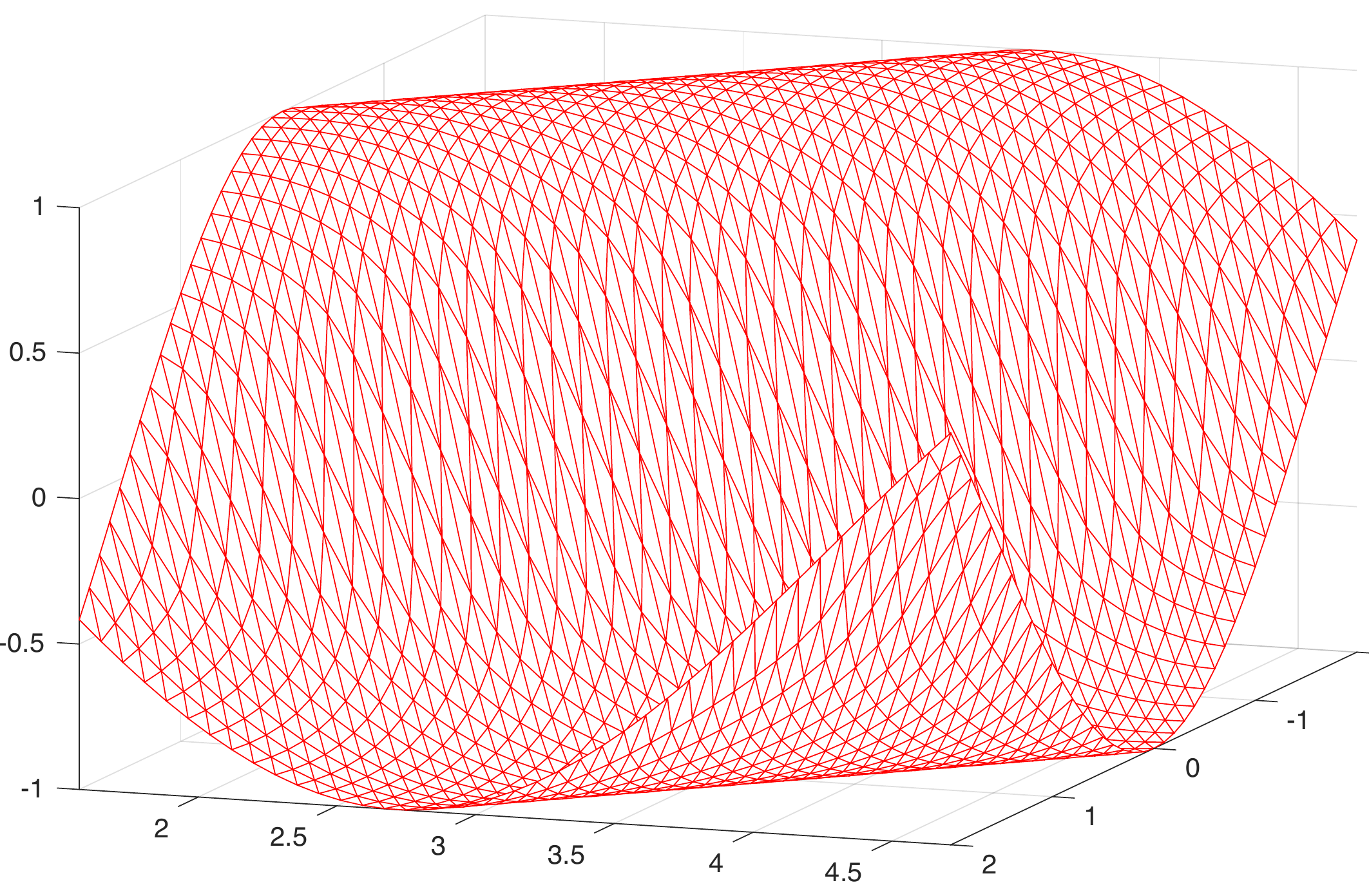}
\end{center}
\centerline{(c) Final Mesh, $\mathbb{M}_K = (k_K+\epsilon)I$}
\end{minipage}
}
\hbox{
\begin{minipage}[t]{2in}
\begin{center}
\includegraphics[width=2in]{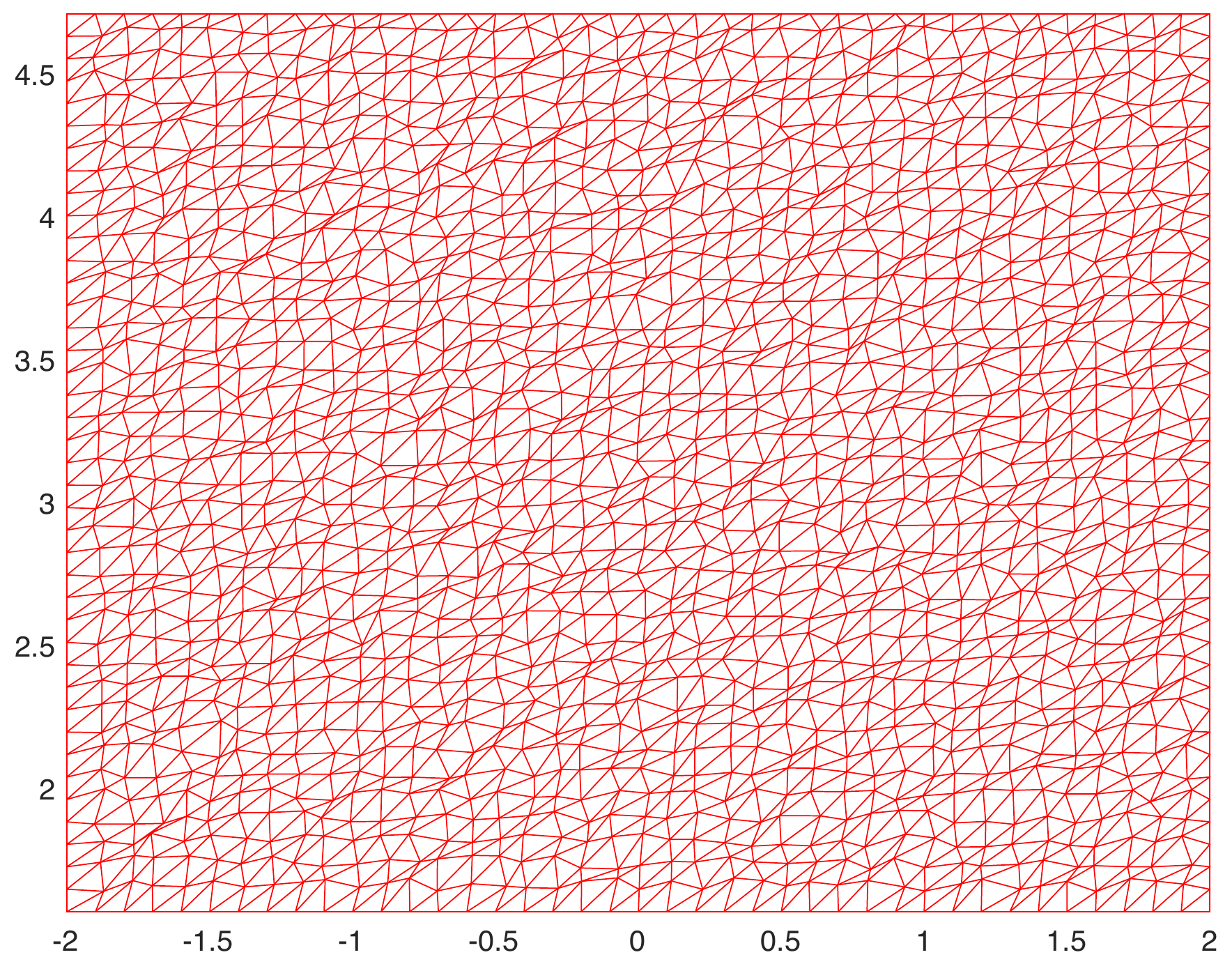}
\end{center}
\centerline{(d) top view of (a)}
\end{minipage}
\hspace{2mm}
\begin{minipage}[t]{2in}
\begin{center}
\includegraphics[width=2in]{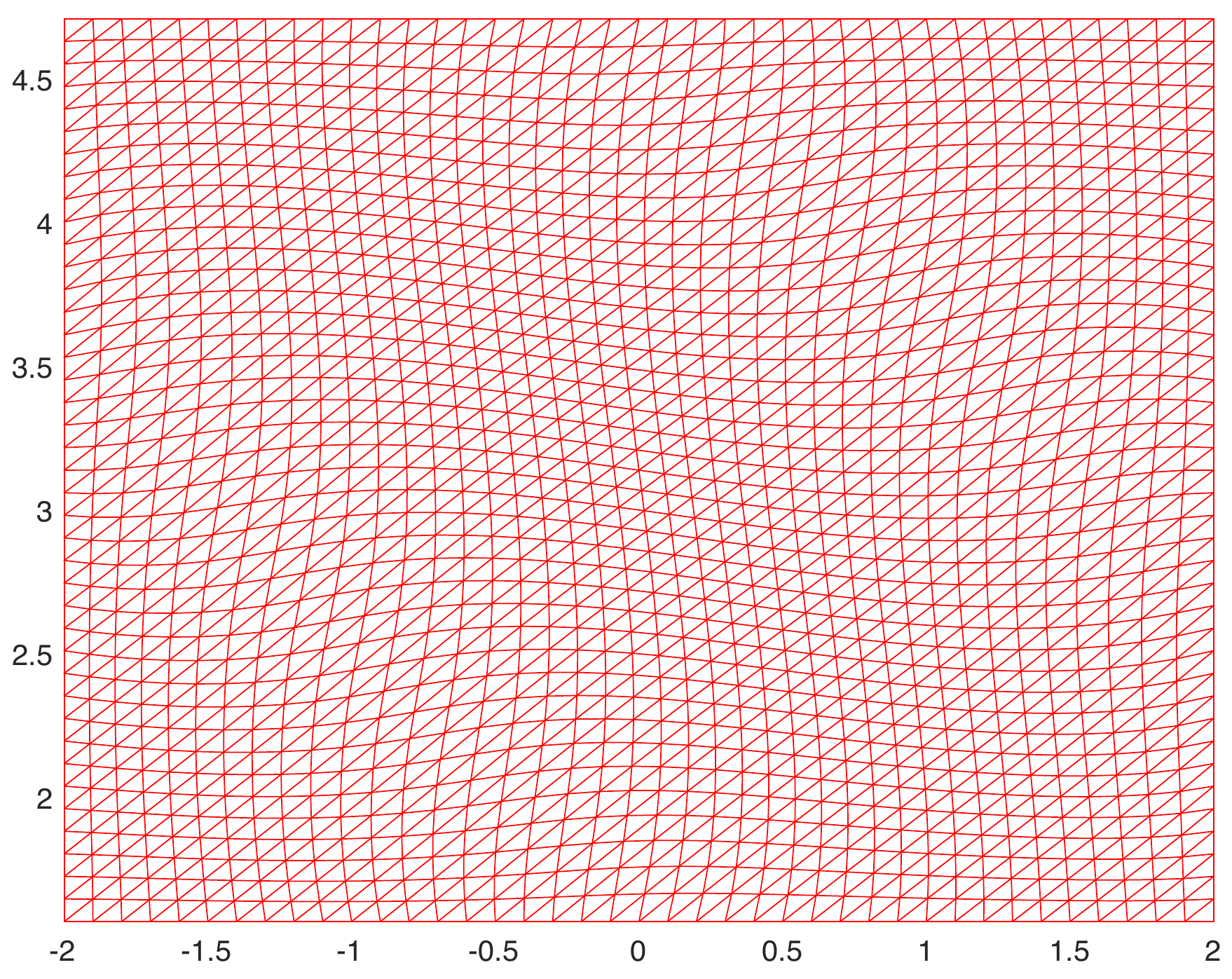}
\end{center}
\centerline{(e) top view of (b)}
\end{minipage}
\hspace{2mm}
\begin{minipage}[t]{2in}
\begin{center}
\includegraphics[width=2in]{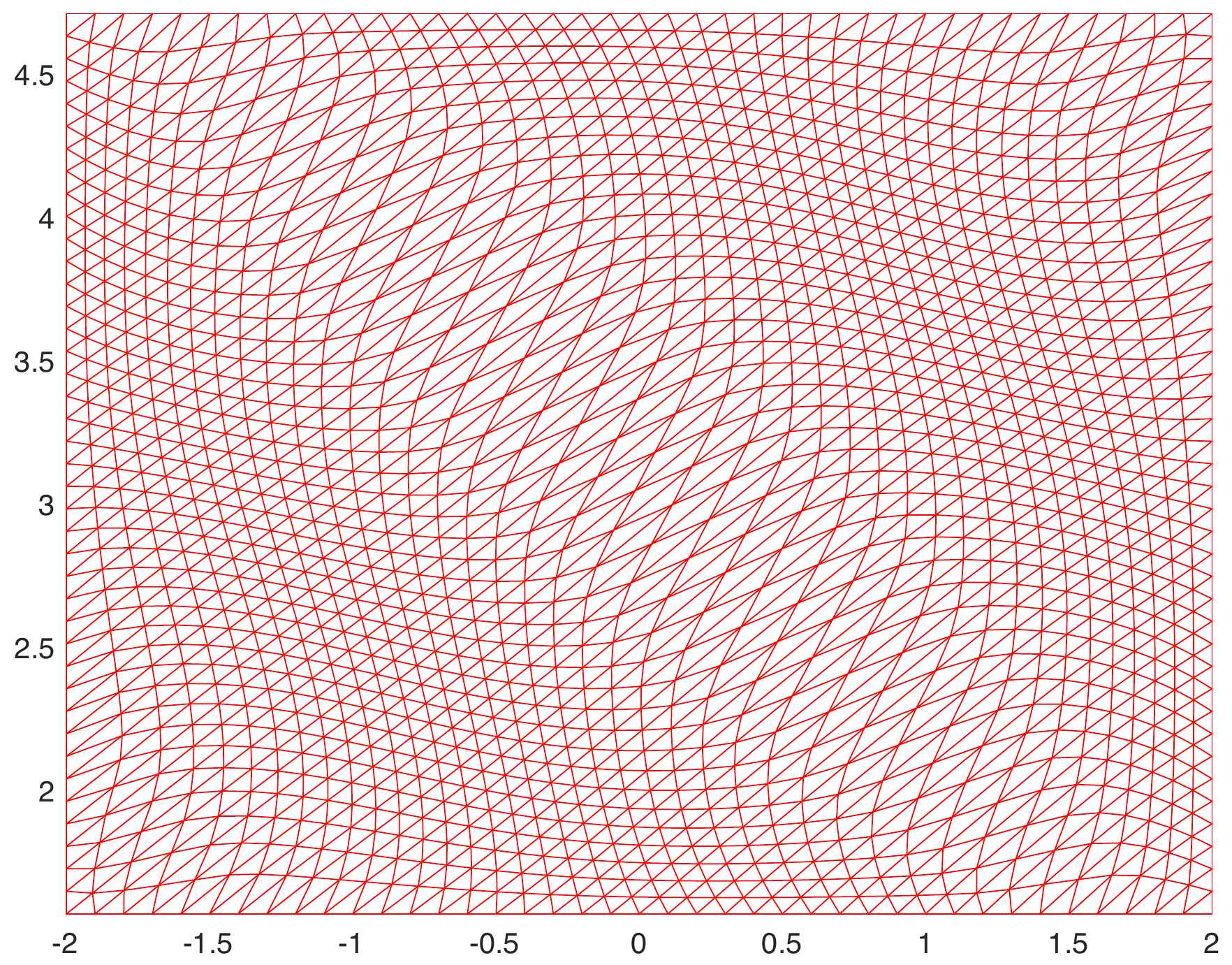}
\end{center}
\centerline{(f) top view of (c)}
\end{minipage}
}
\caption{Example~\ref{sin3d}. Meshes of $N = 3200$ for the surface $\Phi(x,y,z) = \sin(x+y) - z$.
}
\label{figsin3d}
\end{center}
\end{figure}

When $\mathbb{M}_K$ is curvature-based, we see a similar result to Example~\ref{sin2d}. That is, Fig. \ref{figsin3d}(c) and (f) show that the elements are more concentrated in those regions of the surface with larger curvature, i.e., the dip when $z = -1$ and the hill when $z = 1$.  The quality measures with respect to the metric tensor improve from $Q_{eq} = 21.696868$ to $Q_{eq} =1.634091$ and $Q_{ali}=6.527829$ to  $Q_{ali}=2.586702$. The final quality measure for the equidistribution condition close to 1 hence indicating that the final mesh is close to satisfying (\ref{equ}). The final quality measure for the alignment condition is not as close to 1 as the equidistribution condition.  Recall that $\theta$ in the meshing function (\ref{Ih-1}) balances equidistribution and alignment and the choice $\theta = 1/3$ has been used in the computation.
Further computations show that increasing $\theta$ will improve the alignment quality but worsen the equidistribution quality,
and vice versa. This suggests that a perfectly uniform mesh cannot be obtained by minimizing (\ref{Ih-1}) for the curvature-based metric tensor for this example.

Finally, we would like to take a look at the changes of $I_h$ and $|K|_{\min}$ along the mesh trajectory.
As we recall from Section \ref{SEC:theor}, $|K|$ is bounded from below and $I_h$ is decreasing. These can be seen numerically for $\mathbb{M}_K = I$ in Fig. \ref{figsin3d3}(a) and Fig. \ref{figsin3d3}(b). Similar to what we saw in Example \ref{sin2d}, Fig. \ref{figsin3d3}(a) shows that $I_h$ is always decreasing and at around $t = 0.10$ begins to converge. In Fig \ref{figsin3d3}(b) we see an initial increase in the
$|K|_{\min}$ value and then it begins to converge to 4.64$\times 10^{-3}\approx\frac{|S|}{N}$ at $t = 0.10$.  This initial increase, as discussed above, is due to the nonuniformity of the initial mesh. That is, the initial mesh is very nonuniform and therefore $|K|_{\min}$ can be very small whereas when the mesh is adapted, the mesh becomes more uniform and hence the values of $|K|\approx\frac{|S|}{N}$ become almost identical.  This implies that the value of $|K|_{\min}$ is likely to increase as the mesh adapts.

For the case with $\mathbb{M}_K = (k_K + \epsilon) I$, Fig. \ref{figsin3d3}(c) and (d) show similar findings. In \ref{figsin3d3}(c) we see that $I_h$ is decreasing for all time and converging beginning at around $t = 0.15$. Fig. \ref{figsin3d3}(d) shows $|K|_{\min}$ initially increases then begins to converge to about 2.0$\times 10^{-4}$. Furthermore, $|K|_{\min}$ is bounded below by the initial $|K|_{\min}$ value of $0.90\times 10^{-4}$.  These numerical results for the curvature based metric tensor further support the theoretical predictions. 

\begin{figure}[h]
\begin{center}
\hbox{
\begin{minipage}[t]{3in}
\begin{center}
\includegraphics[width=2.75in]{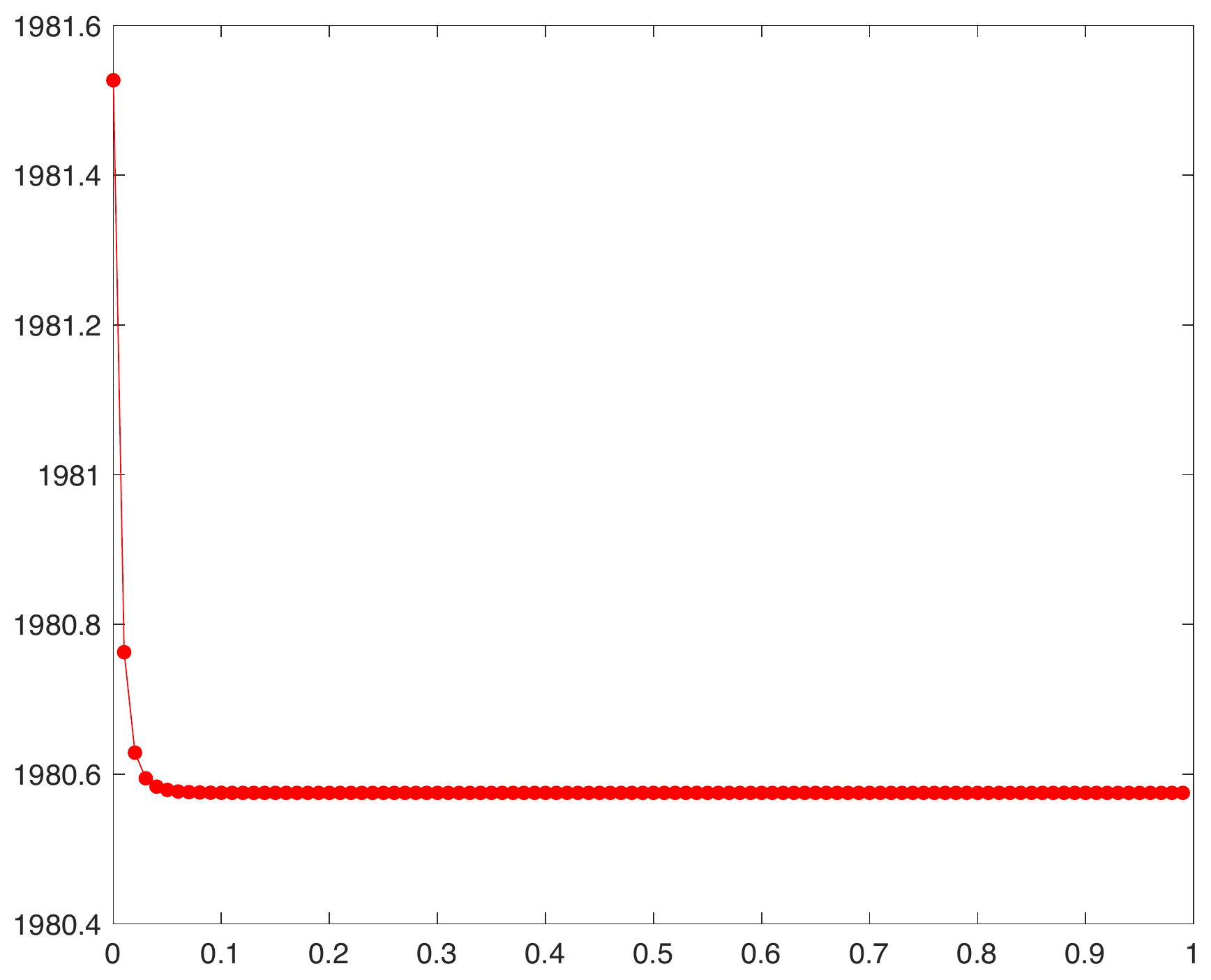}
\end{center}
\vspace{0.42cm}
\centerline{(a) $I_h$, $\mathbb{M}_K = I$}
\end{minipage}
\hspace{2mm}
\begin{minipage}[t]{3in}
\begin{center}
\includegraphics[width=2.75in]{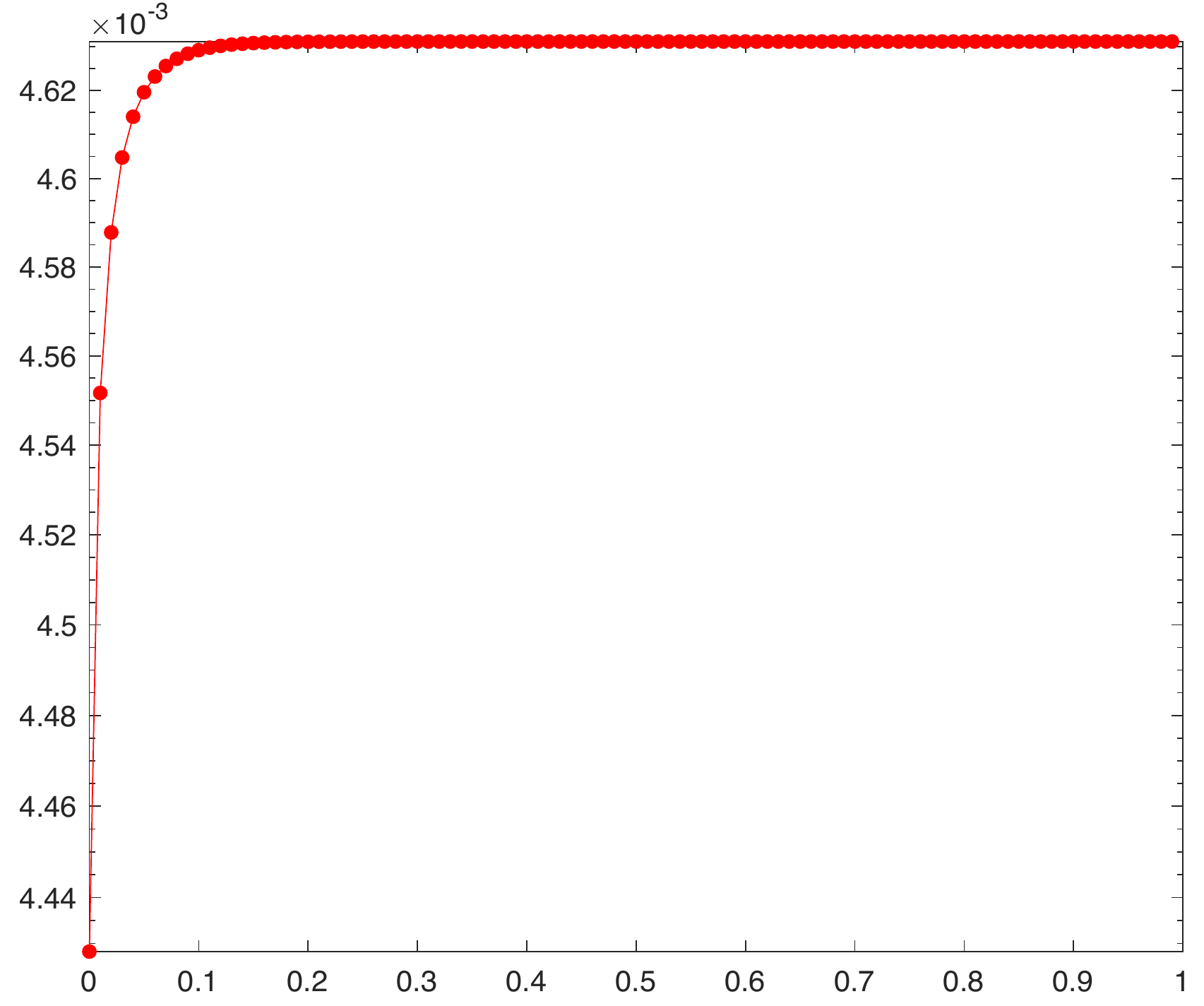}
\end{center}
\centerline{(b) $|K|_{\min}$, $\mathbb{M}_K = I$}
\end{minipage}
}
\hbox{
\begin{minipage}[t]{3in}
\begin{center}
\includegraphics[width=2.75in]{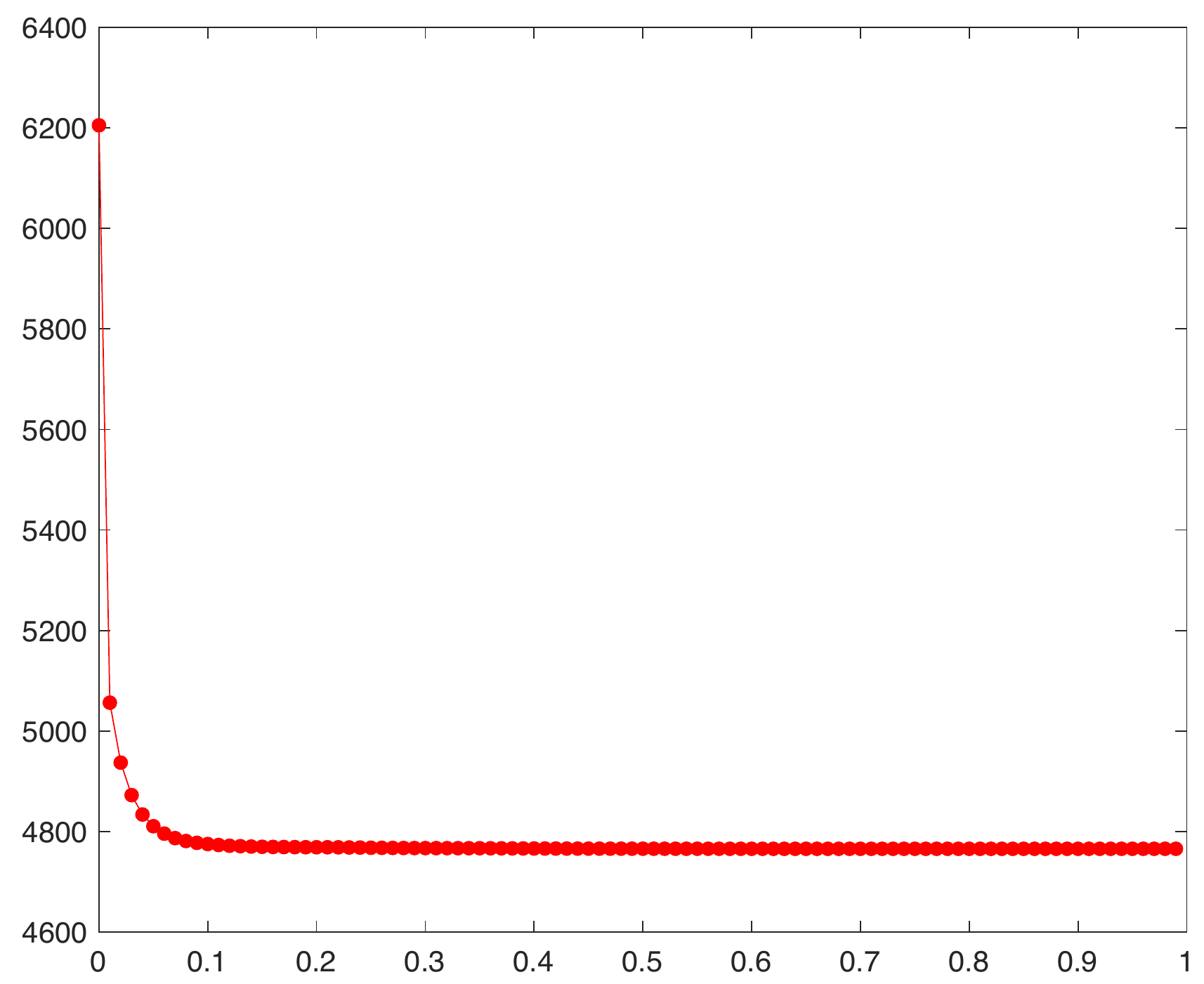}
\end{center}
\vspace{0.42cm}
\centerline{(c) $I_h$, $\mathbb{M}_K =  (k_K + \epsilon) I$}
\end{minipage}
\hspace{2mm}
\begin{minipage}[t]{3in}
\begin{center}
\includegraphics[width=2.75in]{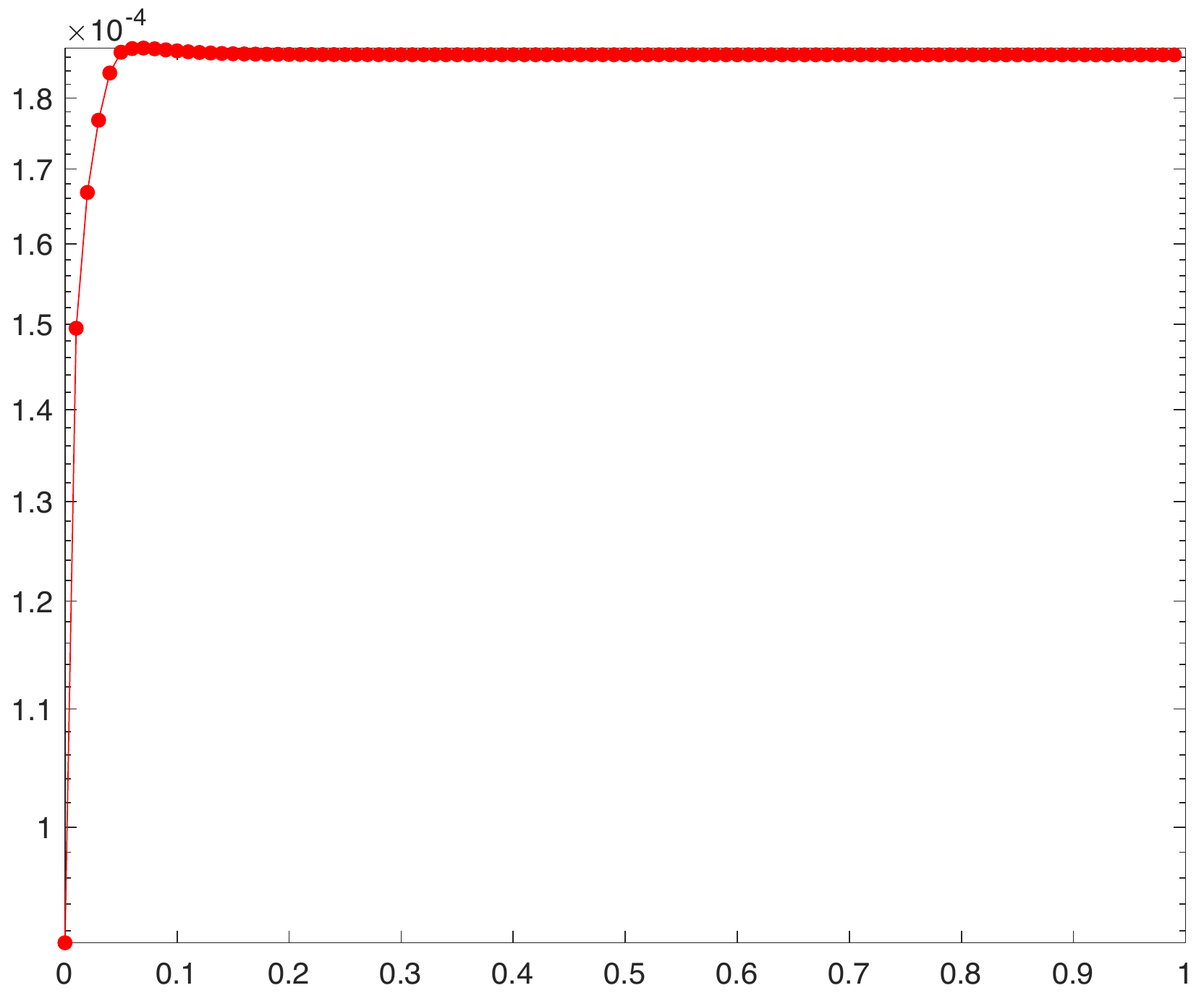}
\end{center}
\centerline{(d) $|K|_{\min}$, $\mathbb{M}_K = (k_K + \epsilon) I$}
\end{minipage}
}
\caption{Example~\ref{sin3d}. $I_h$ and $K_{\min}$ are plotted as functions of $t$ for $\Phi(x,y,z) = \sin(x+y) - z$.
}
\label{figsin3d3}
\end{center}
\end{figure}
\end{exam}

\vspace{10pt}

\begin{exam}
\label{ellipsoid}
Our final example explores the sphere and ellipsoid defined by an icosahedral initial mesh (see \cite{WL} for more details). We begin with the sphere
\[
\Phi(x,y,z) = x^2+y^2+z^2-1.
\]
For this example we take $N= 1280$. 

As we see from Fig. \ref{figsphere}(a), the initial mesh is close to being uniform however, there is a very slight difference in the final mesh Fig. \ref{figsphere}(b) when $\mathbb{M}_K = I$. Indeed, this slight adaptation can be seen in the quality measures which change from
$Q_{eq} = 1.068461$ and $Q_{ali} = 1.025691$ for the initial mesh to $Q_{eq} = 1.289843$
and $Q_{ali} = 1.025972$ for the final mesh. The difference in the quality measures indicates that the initial icosahedral mesh is almost uniform and so the moving mesh method does not affect the mesh significantly.

We further this example to consider adaptive meshes for the ellipsoid defined by 
\[
\Phi(x,y,z) = x^2+y^2+\dfrac{z^2}{4}-1.
\]
We move the mesh on the surface for both $N = 1280$ and $N = 5120$.

First considering $N = 1280$, Fig. \ref{figellipsoid} shows the meshes for this example in two different views.  Studying Fig. \ref{figellipsoid}(a) and Fig. \ref{figellipsoid}(d), the initial mesh, and Fig. \ref{figellipsoid}(b) and Fig. \ref{figellipsoid} (e), the final mesh with $\mathbb{M}_K = I$, we can see that the final mesh adapts to provide a higher concentration of elements in the middle region of the ellipsoid and fewer elements near the tips. The quality measures improve from $Q_{eq} = 1.724289$ and $Q_{ali} = 1.453207$ for the initial mesh to $Q_{eq} = 1.571401$ and $Q_{ali} = 1.102655$ for the final mesh. Although the initial mesh is close to uniform, the final mesh adapts in such a way to satisfy the equidistribution and alignment condition on the surface. However, this is not an accurate representation of the shape thus we consider a curvature-based metric tensor. 

In our numerical experiments, when $\mathbb{M}_K = (k_K + \epsilon)I$ is used,  we saw the mesh adapt in a similar way as with the Euclidean metric. This is because the curvature of the ellipsoid does not change significantly at the tips thus not many nodes move there. With this in mind, we altered the curvature-based metric tensor to concentrate more mesh elements at the tips of the ellipsoid by redefining $\mathbb{M}_K$ as 
\begin{equation}
\tilde{\mathbb{M}}_K = \mathbb{M}_K + \left (\dfrac{1}{\sqrt{(z_K-2)^2 + \epsilon}}+ \dfrac{1}{\sqrt{(z_K+2)^2 + \epsilon}}\right ) I .
\label{M-ellipsoid}
\end{equation}
Fig.~\ref{figellipsoid}(c) and Fig.~\ref{figellipsoid}(f) show the final mesh using this altered metric tensor. As we can see, the mesh elements have concentrated at the tips of the ellipsoid better representing the shape of the surface.  The equidistribution quality measure changes from $1.374300$ initially to $1.967482$ whereas the alignment quality measure from $1.453207$ to $1.262156$. Similar results are seen with a finer mesh in Fig.~\ref{figellipsoid2} for both the Euclidean metric and altered curvature-based metric.

\begin{figure}[tbh]
\begin{center}
\hbox{\hspace{4mm}
\begin{minipage}[t]{3in}
\begin{center}
\includegraphics[width=3in]{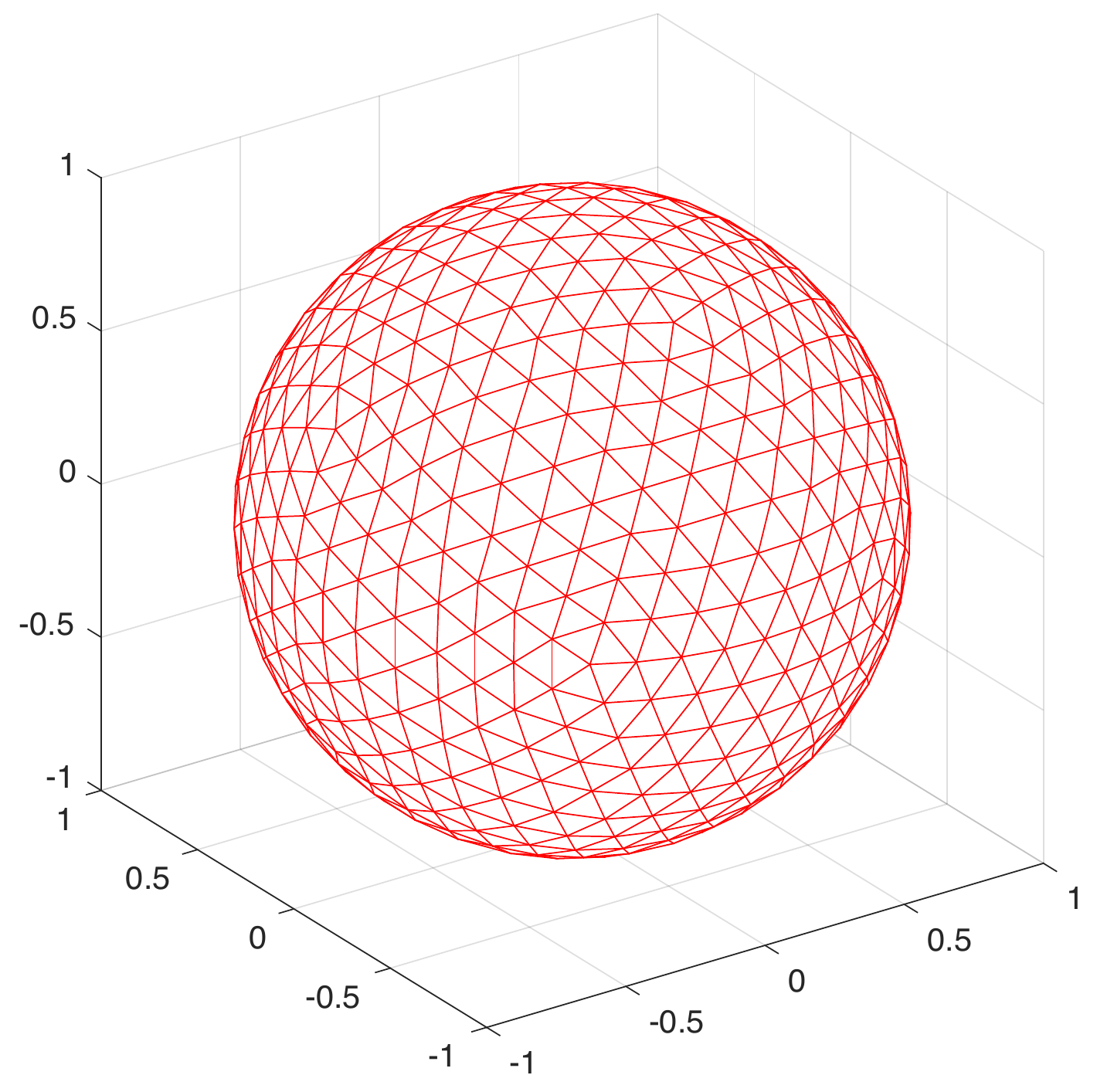}
\end{center}
\vspace{0.42cm}
\centerline{(a) Initial Mesh}\
\end{minipage}
\hspace{4mm}
\begin{minipage}[t]{3in}
\begin{center}
\includegraphics[width=3in]{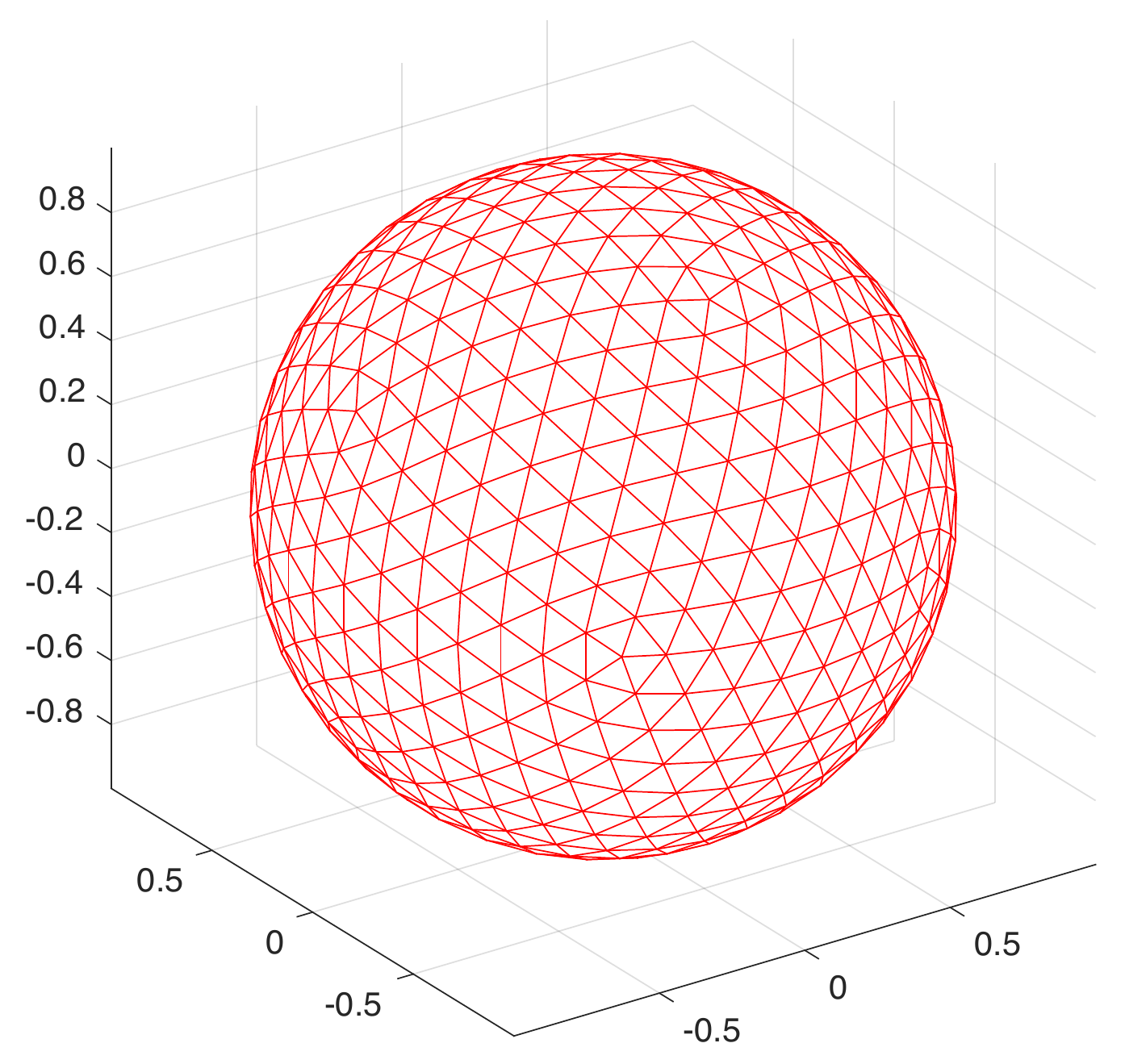}
\end{center}
\centerline{(b) Final Mesh $\mathbb{M}_K = I$}\
\end{minipage}
\hspace{4mm}}
\hbox{\hspace{4mm}
\begin{minipage}[t]{3in}
\begin{center}
\includegraphics[width=2.1in]{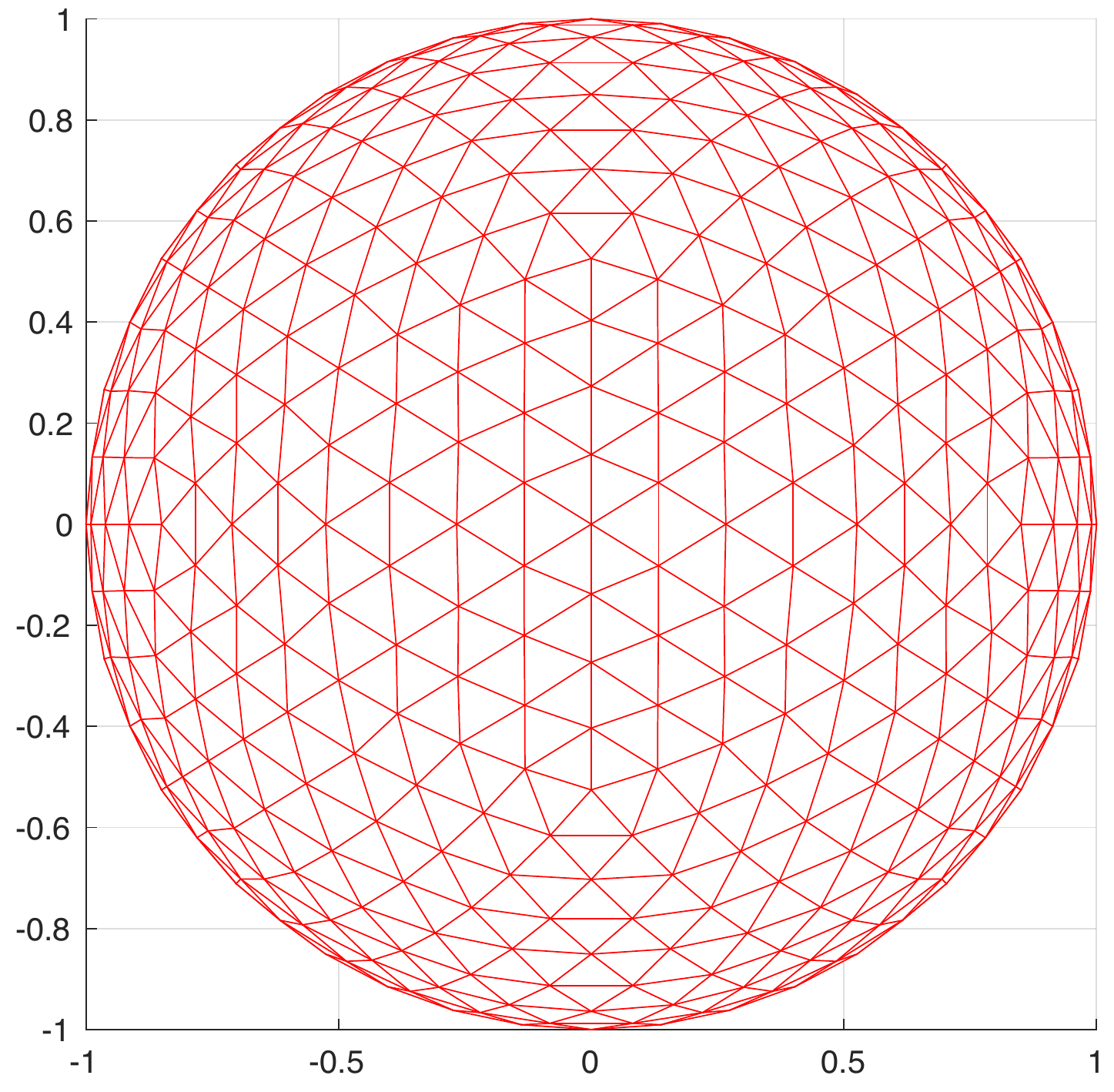}
\end{center}
\centerline{(c) top view of (a)}
\end{minipage}
\hspace{4mm}
\begin{minipage}[t]{3in}
\begin{center}
\includegraphics[width=2.1in]{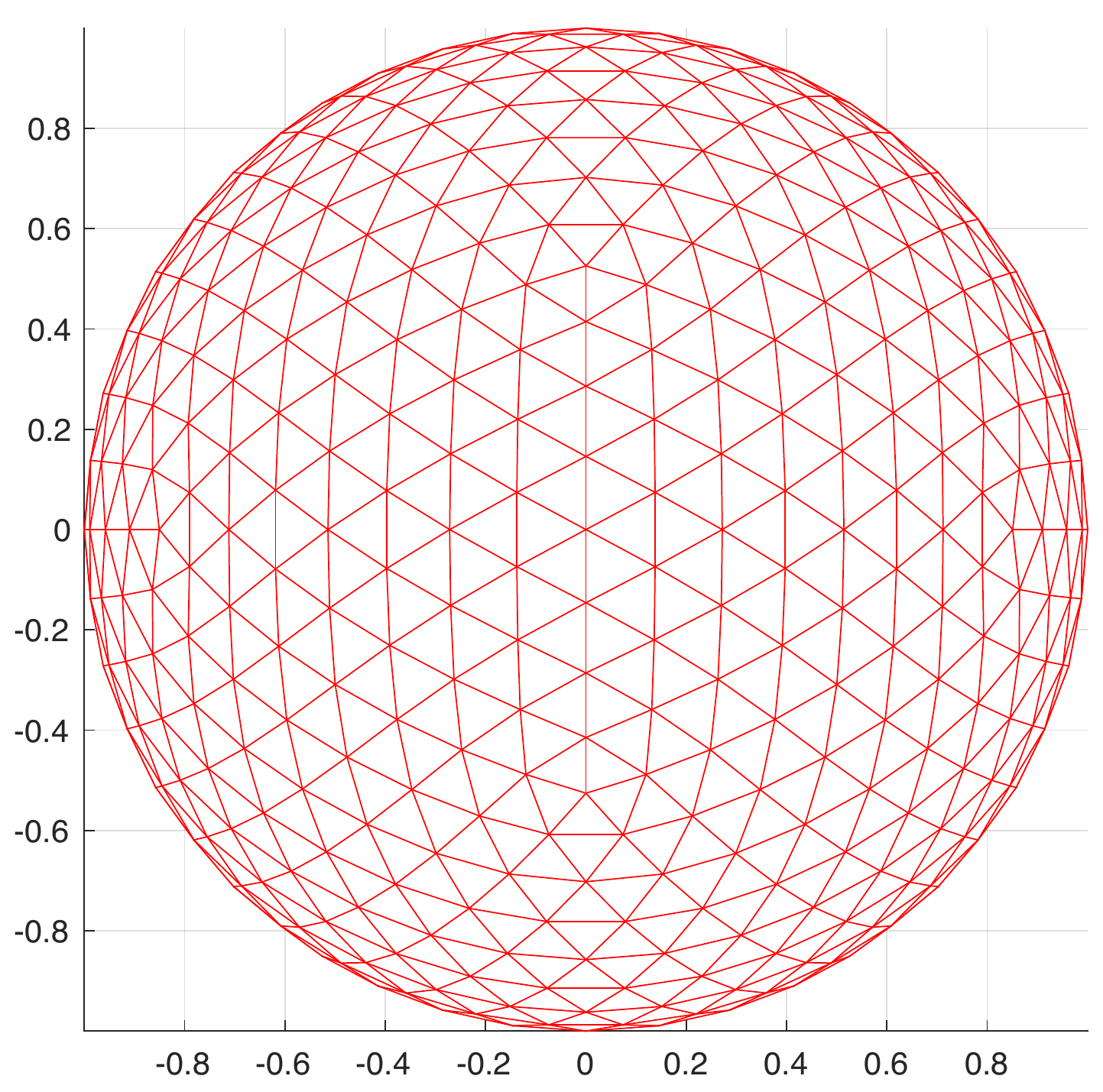}
\end{center}
\centerline{(d) top view of (b)}
\end{minipage}
\hspace{4mm}}
\caption{Example~\ref{ellipsoid}. Meshes of $N = 1280$ are plotted for $\Phi(x,y,z) = x^2+y^2+z^2-1$.
}
\label{figsphere}
\end{center}
\end{figure}

\begin{figure}[tbh]
\begin{center}
\hbox{
\begin{minipage}[t]{2in}
\begin{center}
\includegraphics[width=2in]{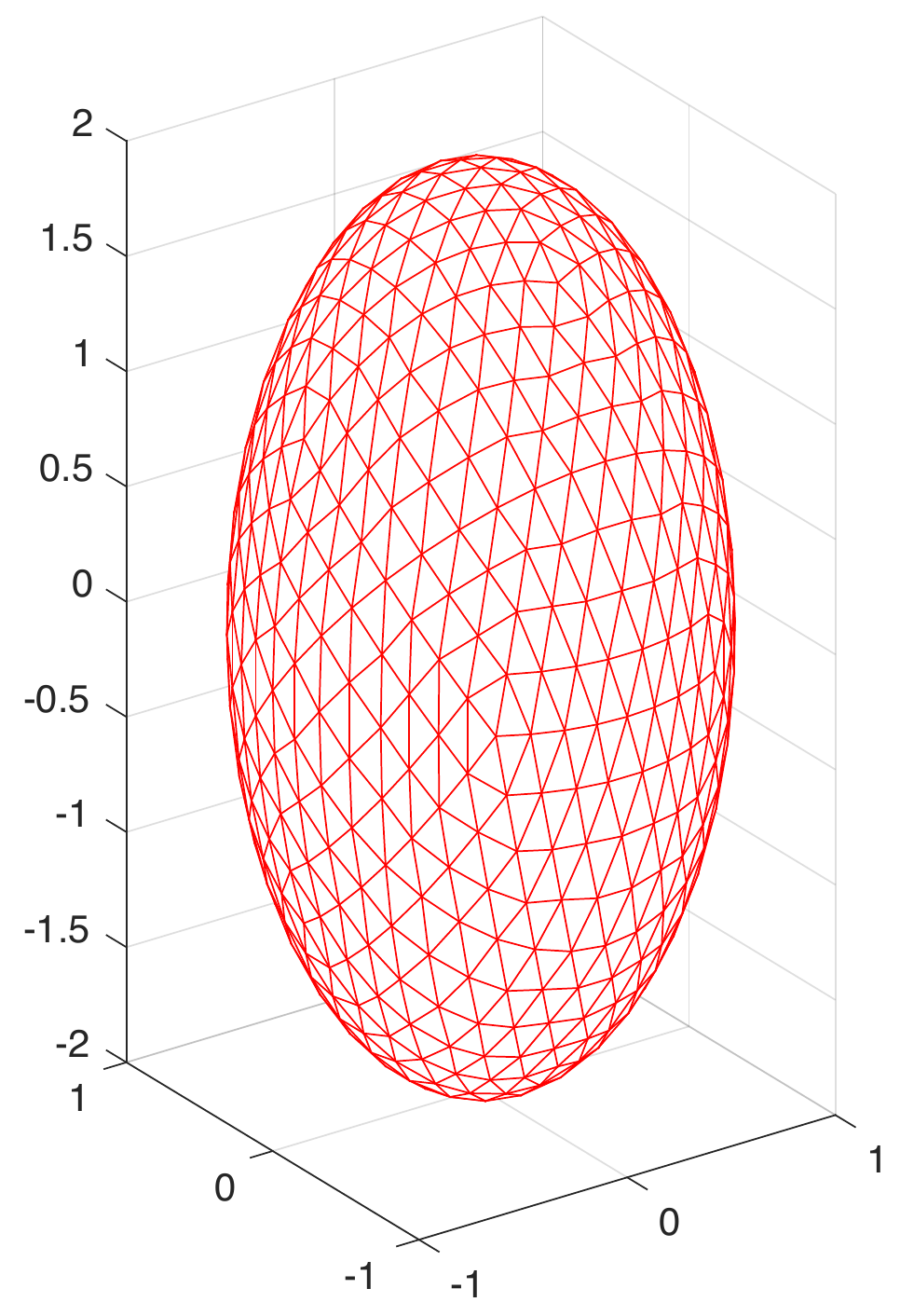}
\end{center}
\vspace{0.42cm}
\centerline{(a) Initial Mesh}\
\end{minipage}
\hspace{2mm}
\begin{minipage}[t]{2in}
\begin{center}
\includegraphics[width=2in]{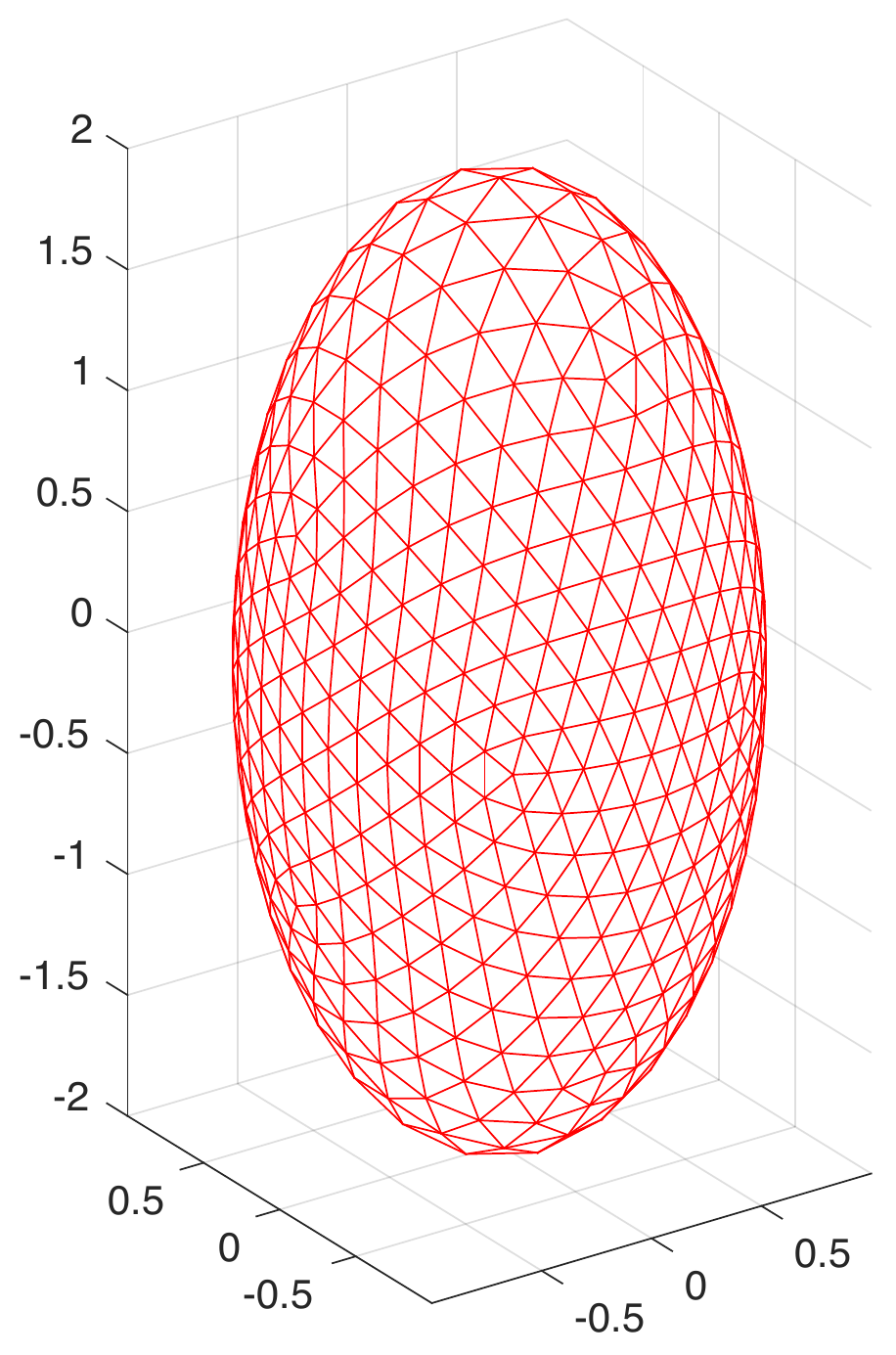}
\end{center}
\centerline{(b) Final Mesh $\mathbb{M}_K = I$}\
\end{minipage}
\hspace{2mm}
\begin{minipage}[t]{2in}
\begin{center}
\includegraphics[width=2in]{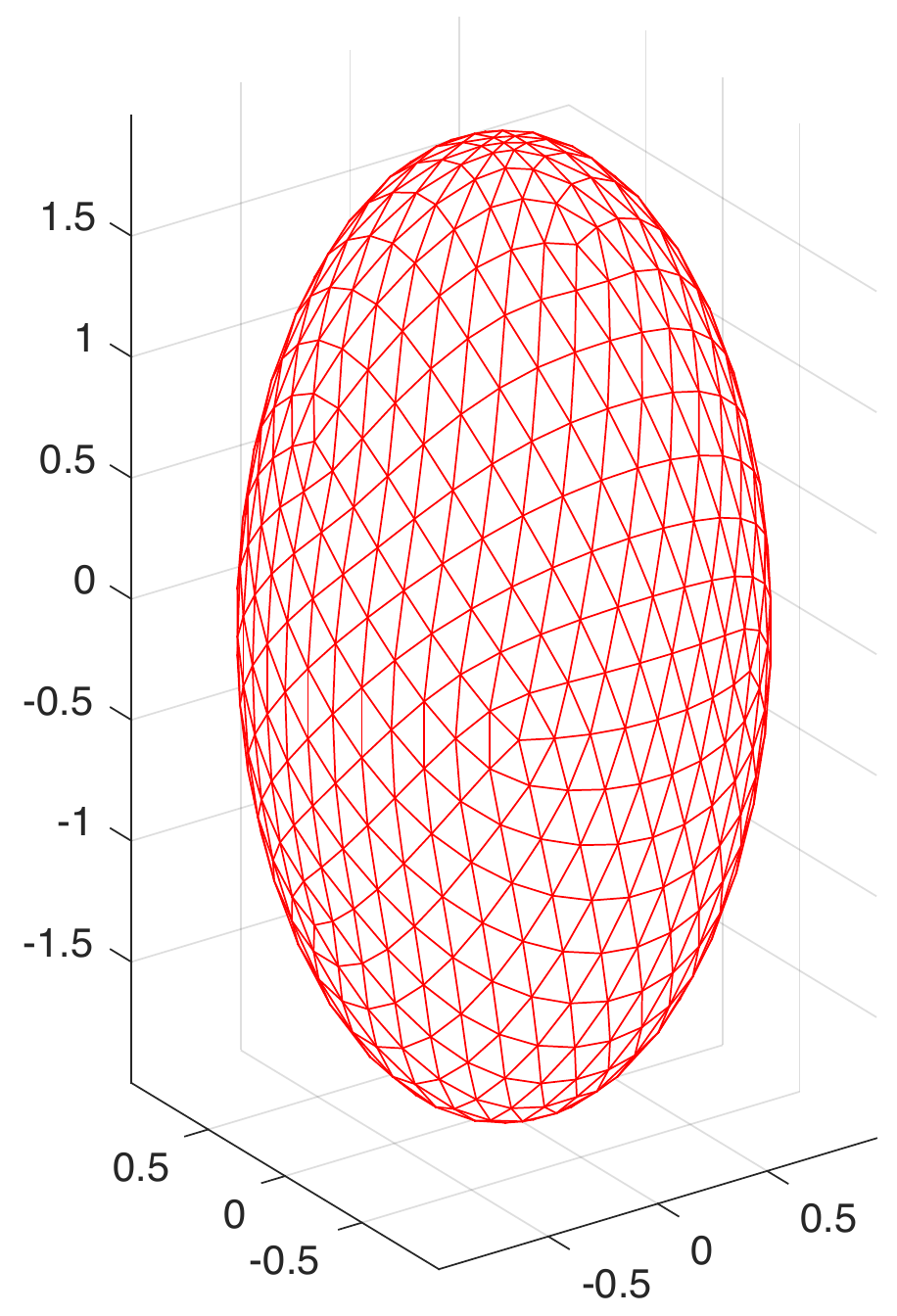}
\end{center}
\centerline{(c) Final Mesh $\tilde{\mathbb{M}}_K$ in (\ref{M-ellipsoid})}\
\end{minipage}
}
\hbox{
\begin{minipage}[t]{2in}
\begin{center}
\includegraphics[width=2in]{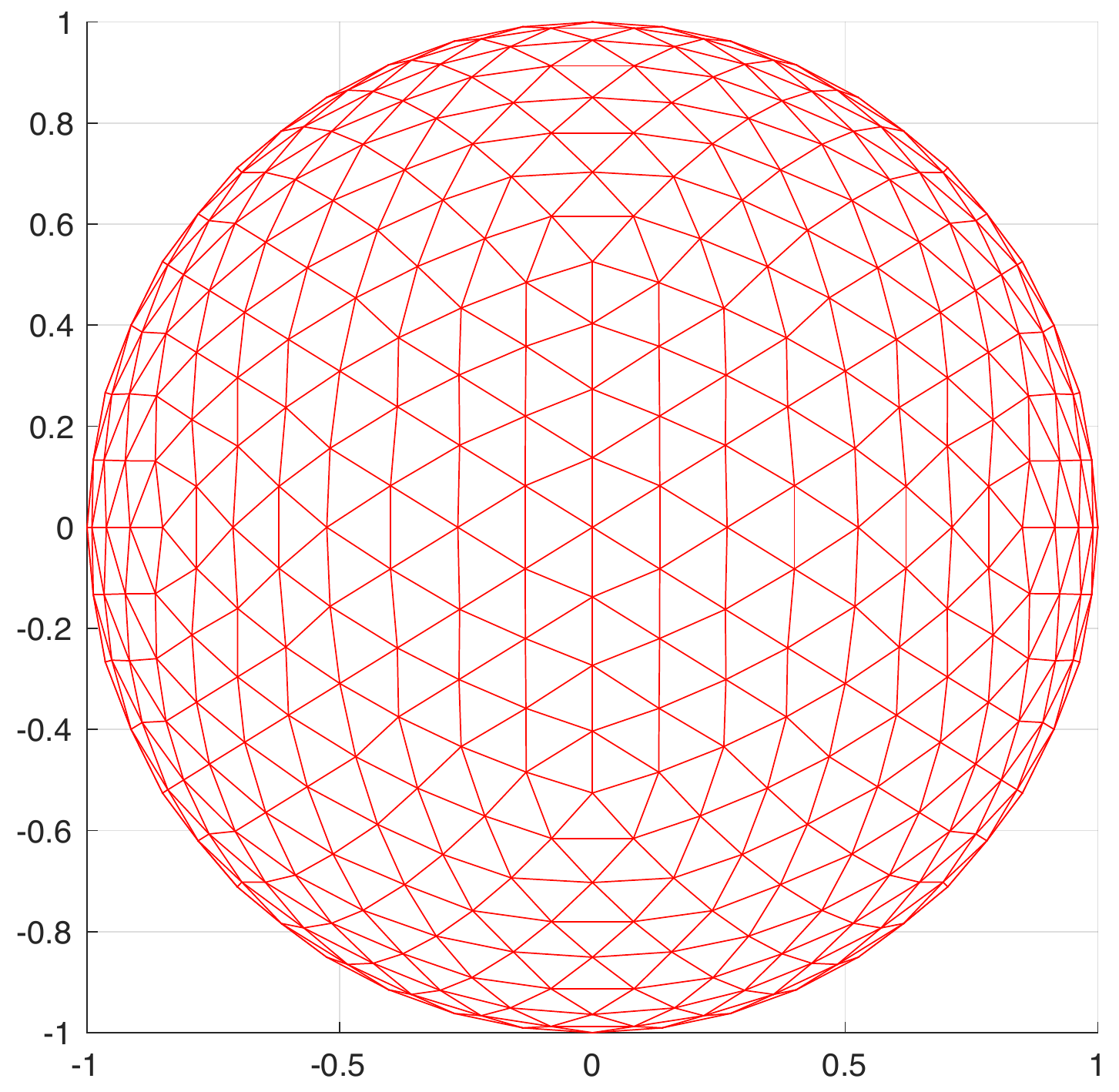}
\end{center}
\centerline{(d) top view of (a)}
\end{minipage}
\hspace{2mm}
\begin{minipage}[t]{2in}
\begin{center}
\includegraphics[width=2in]{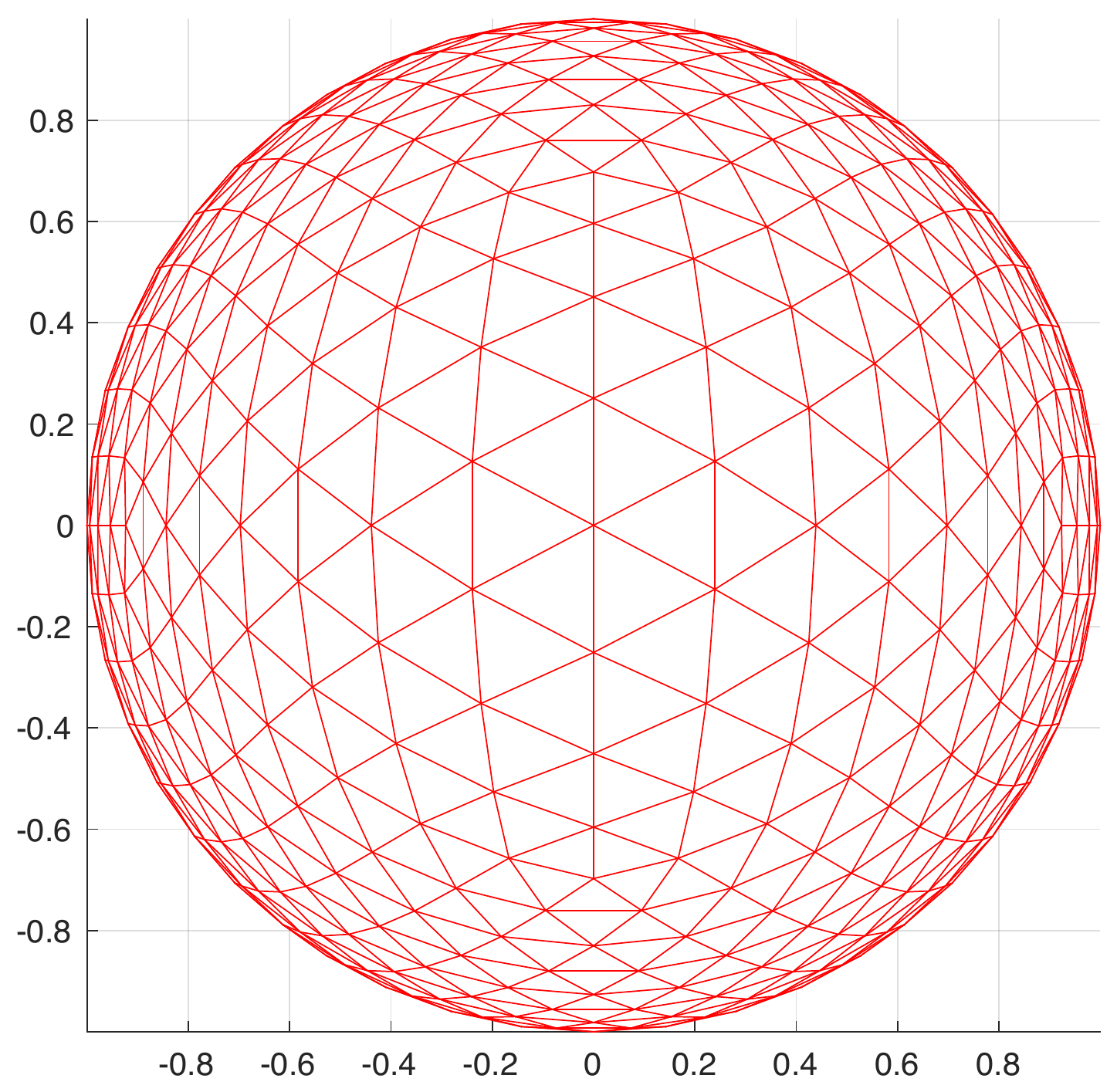}
\end{center}
\centerline{(e) top view of (b)}
\end{minipage}
\hspace{2mm}
\begin{minipage}[t]{2in}
\begin{center}
\includegraphics[width=2in]{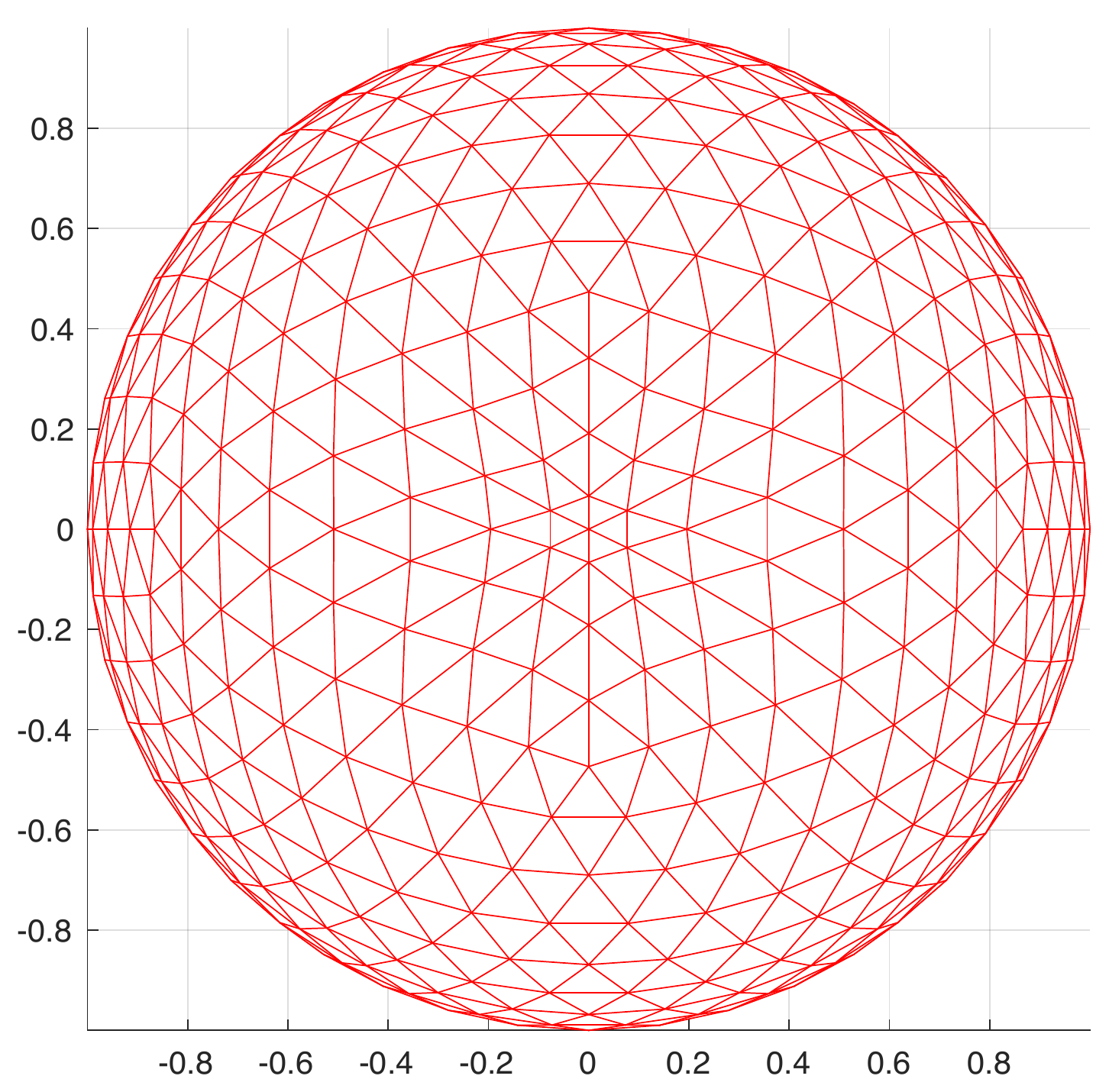}
\end{center}
\centerline{(f) top view of (c)}
\end{minipage}
}
\caption{Example~\ref{ellipsoid}. Meshes of $N = 1280$ are plotted for $\Phi(x,y,z) = x^2+y^2+\dfrac{z^2}{4}-1$.
}
\label{figellipsoid}
\end{center}
\end{figure}

\begin{figure}[tbh]
\begin{center}
\hbox{
\begin{minipage}[t]{2in}
\begin{center}
\includegraphics[width=2in]{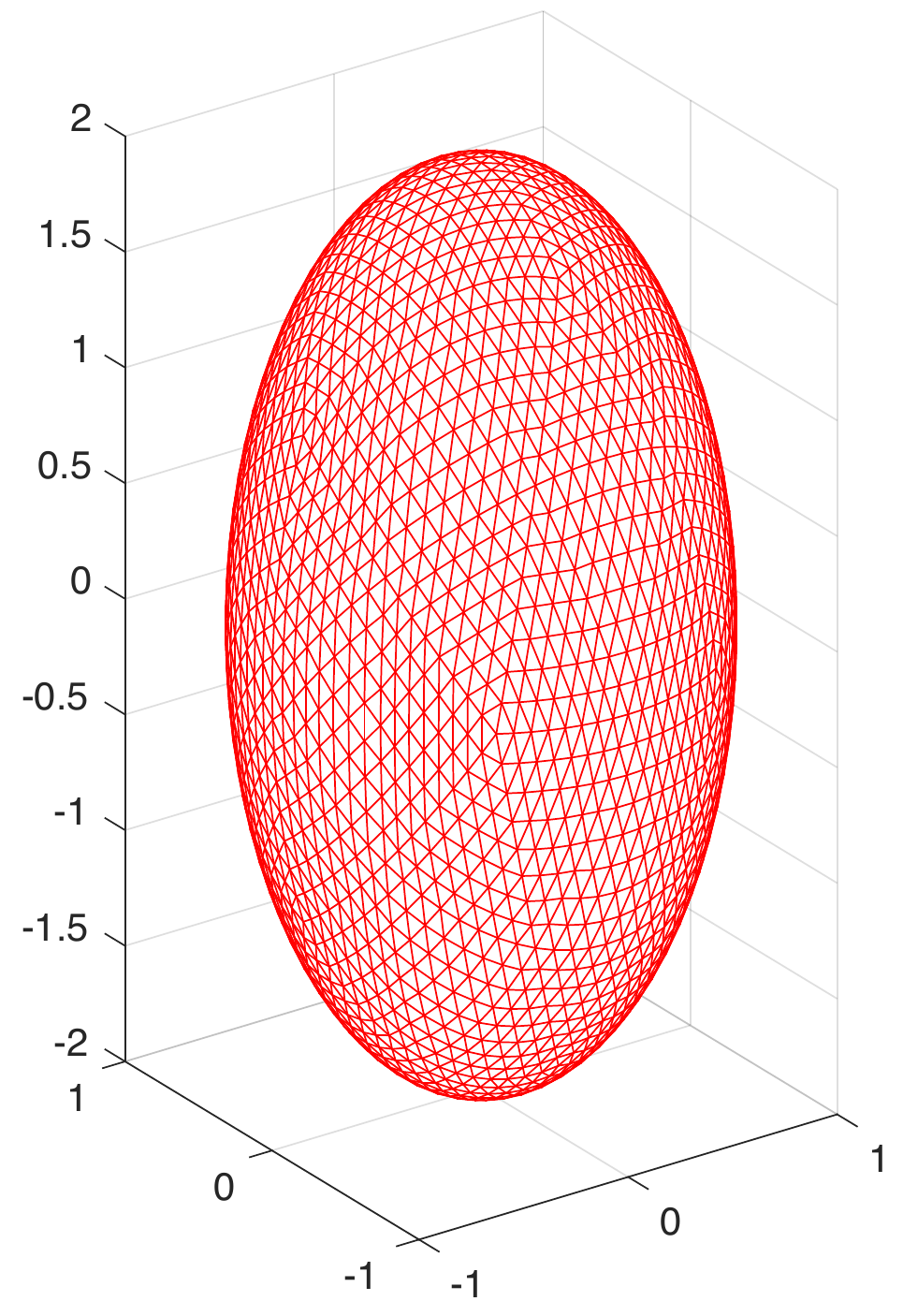}
\end{center}
\vspace{0.42cm}
\centerline{(a) Initial Mesh}\
\end{minipage}
\hspace{2mm}
\begin{minipage}[t]{2in}
\begin{center}
\includegraphics[width=2in]{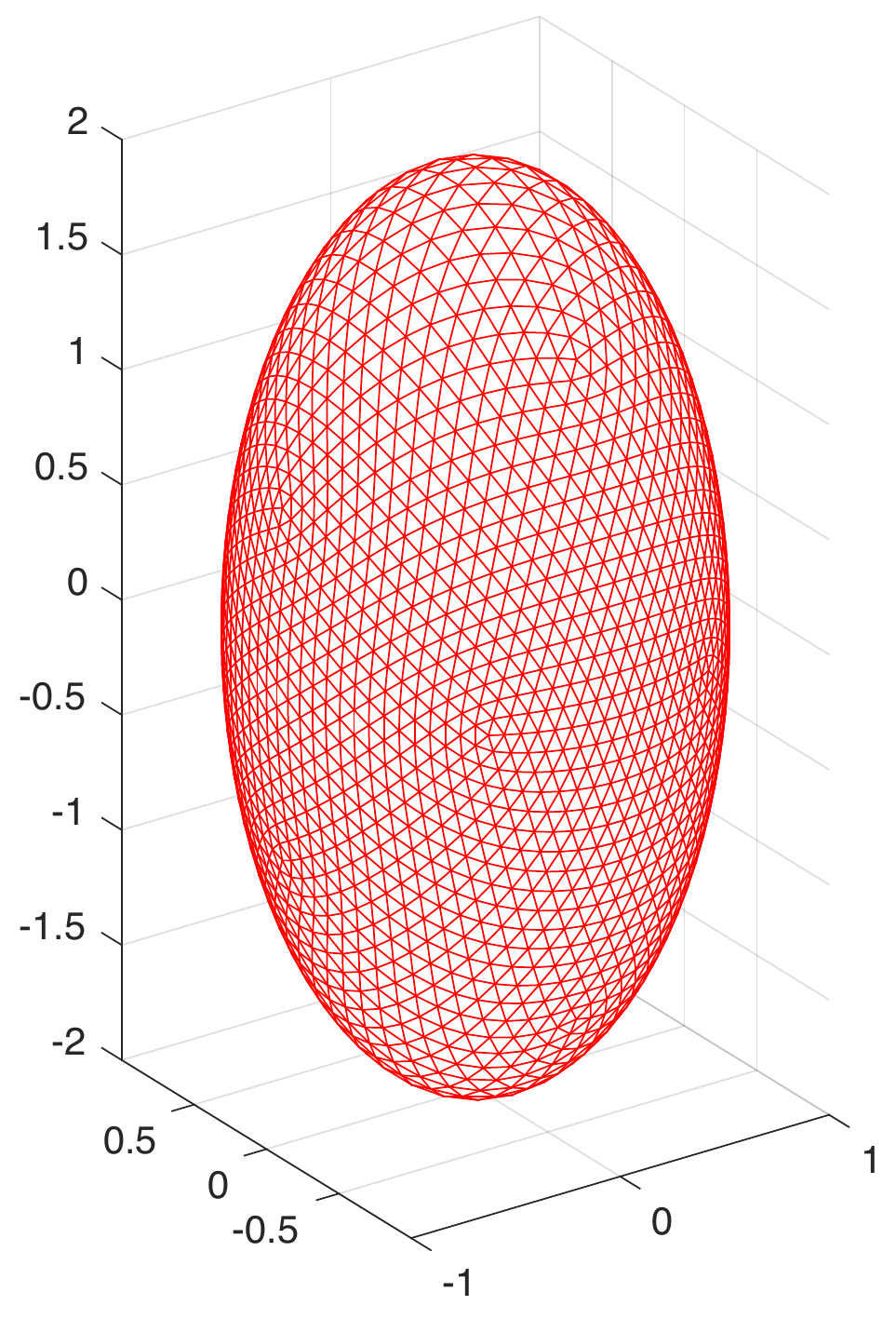}
\end{center}
\centerline{(b) Final Mesh $\mathbb{M}_K = I$}\
\end{minipage}
\hspace{2mm}
\begin{minipage}[t]{2in}
\begin{center}
\includegraphics[width=2in]{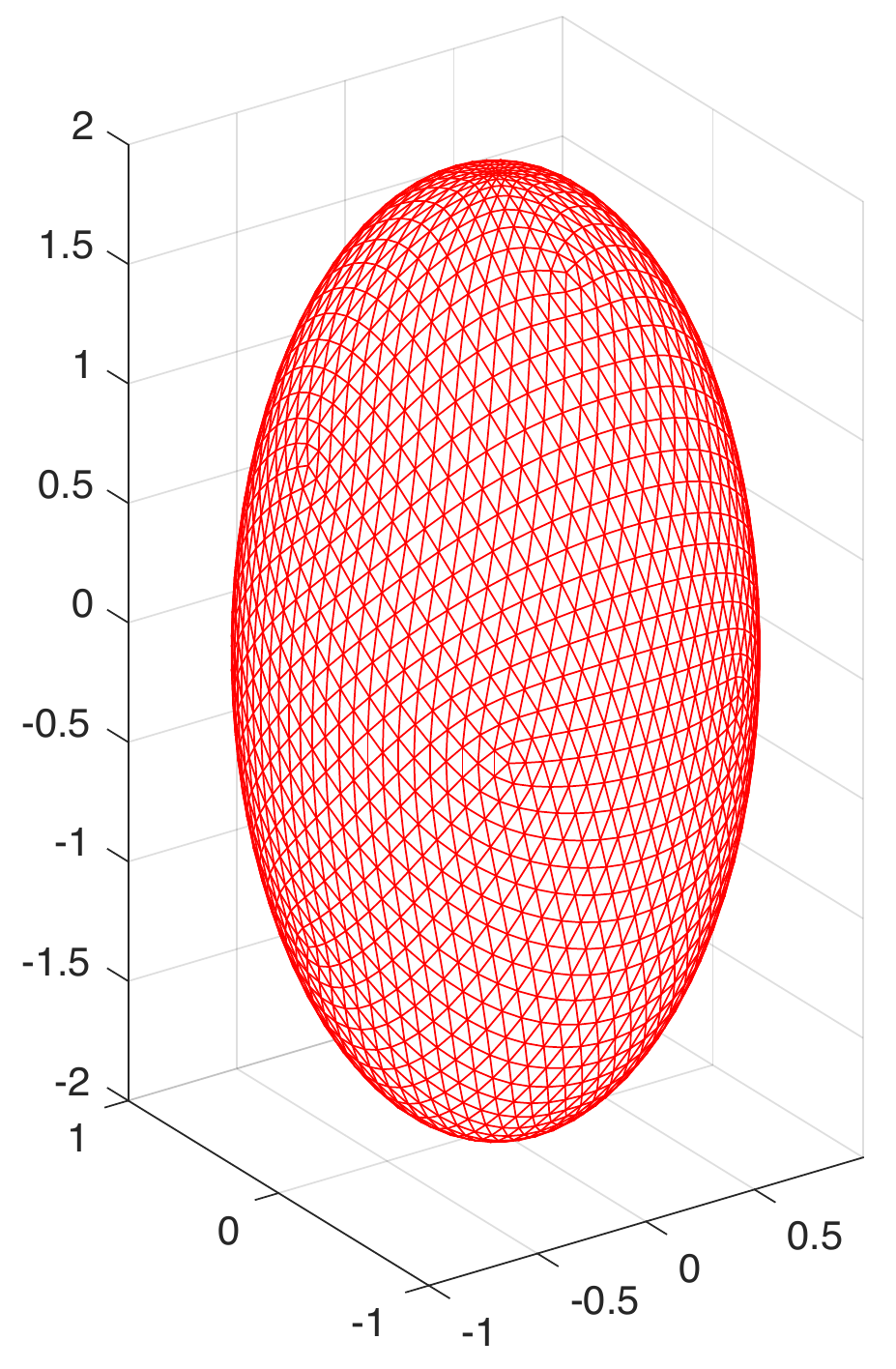}
\end{center}
\centerline{(c) Final Mesh $\tilde{\mathbb{M}}_K$ in (\ref{M-ellipsoid})}\
\end{minipage}
}
\hbox{
\begin{minipage}[t]{2in}
\begin{center}
\includegraphics[width=2in]{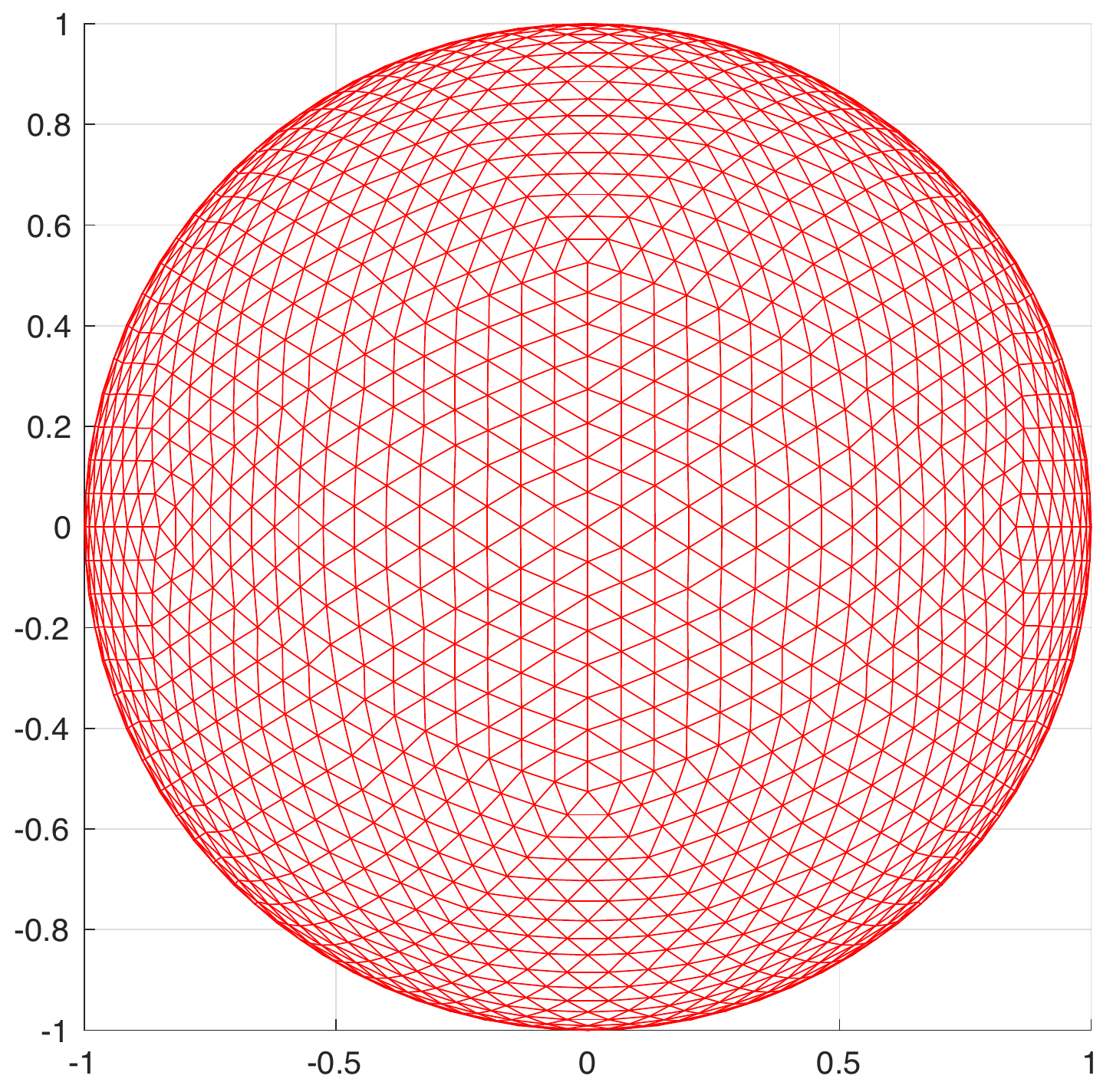}
\end{center}
\centerline{(d) top view of (a)}
\end{minipage}
\hspace{2mm}
\begin{minipage}[t]{2in}
\begin{center}
\includegraphics[width=2in]{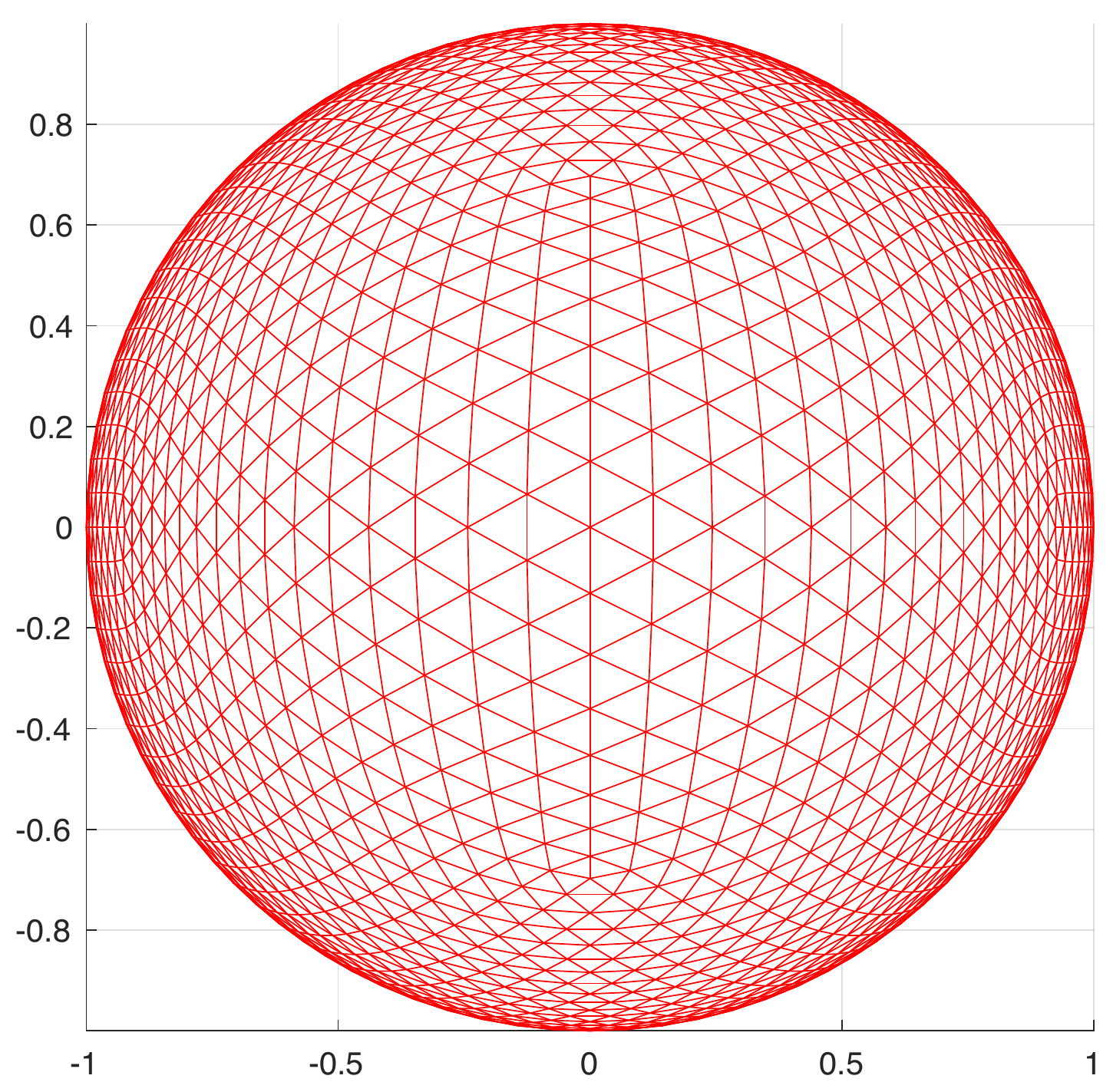}
\end{center}
\centerline{(e) top view of (b)}
\end{minipage}
\hspace{2mm}
\begin{minipage}[t]{2in}
\begin{center}
\includegraphics[width=2in]{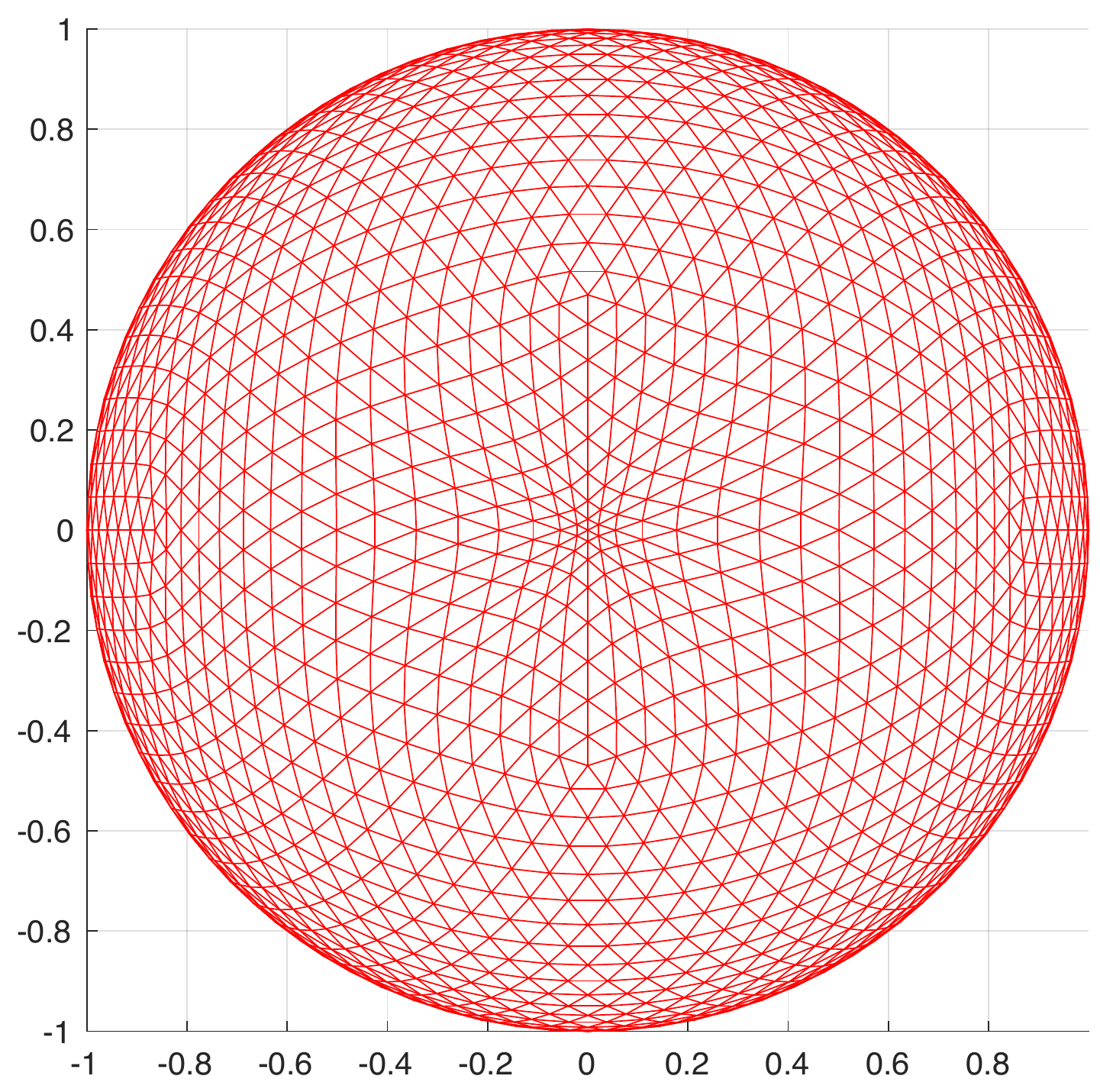}
\end{center}
\centerline{(f) top view of (c)}
\end{minipage}
}
\caption{Example~\ref{ellipsoid}. Meshes of $N = 5120$ are plotted for $\Phi(x,y,z) = x^2+y^2+\dfrac{z^2}{4}-1$.
}
\label{figellipsoid2}
\end{center}
\end{figure}

\end{exam}

\section{Conclusions and further comments}
\label{SEC:conclusion}
We have proposed a direct approach for surface mesh movement and adaptation that can be applied to a general surface with or without analytical expressions. 
We did so by first proving the relation (\ref{thm-|K|-2}) between the area of a surface element in a Riemannian metric and the Jacobian matrix of the affine mapping between the reference element and any simplicial surface element.
From this we formulated the equidistribution and alignment conditions as given in (\ref{equ}) and (\ref{ali}), respectively.  These two conditions enabled us to formulate a surface meshing function that is similar to a discrete version of Huang's functional for bulk meshes \cite{H}.  
The surface function satisifies the coercivity condition (\ref{coer}) for $\theta\in(0,1/2]$ and $p>1$. 

We defined the surface MMPDE (\ref{MMPDE-1}) as the gradient system of the meshing function, which utilizes surface normal vectors to inherently ensure that the mesh vertices remain on the surface during movement. 
Equations (\ref{vel2}) and (\ref{vel1}) give explicit, compact formulas for the mesh velocities making the time integration of the surface MMPDE (\ref{MMPDEx}) relatively easy to implement. 
Moreover, we showed that this surface MMPDE satisfies the energy decreasing property, which is one of the keys to proving Theorem \ref{nonsin}. This theorem is an important theoretical result as it states that the surface mesh remains nonsingular for all time if it is so initially.  Finally, we proved Theorem \ref{limiting} that states the mesh has limiting meshes, all of which are nonsingular.

A point of emphasis is that the new method is developed directly on surface meshes thus, making no use of any information on surface parameterization.  As mentioned, the MMPDE (\ref{MMPDE-1}) only depends on surface normal vectors which can be computed even when the surface has a numerical representation.  This allows the new method to be applied to general surfaces with or without explicit parameterization.

The numerical results presented in this work demonstrated that this new approach to surface mesh movement is successful. In all of the examples, the final mesh was seen to be much more uniform with respect to both cases of the metric tensor $\mathbb{M}_K= I$ and $\mathbb{M}_K = (k_K + \epsilon)I$ which was supported by the mesh quality measures.  Moreover, the theoretical properties were numerically verified in Ex. \ref{sin2d} and Ex \ref{sin3d} as we showed that $I_h$ is decreasing and $|K|$ is bounded below. 

The future goal is to develop this algorithm for any surface with or without analytical expression. Even though we only presented examples which have analytical expressions, we should emphasize that the MMPDE (\ref{MMPDEx}) uses only  the normal direction of the surface. Since these derivatives can be obtained numerically from the initial mesh or a background mesh, the method developed in this paper should work in principle for surfaces without explicit expressions. A practical difficulty is that the initial mesh or a background mesh typically does not represent the underlying surface accurately and approximate partial derivatives of $\Phi$ obtained directly from the mesh may not lead to acceptable results. Our next step in the research is to investigate the use of spline approximations of surfaces for this purpose. Moreover, the monitor functions we used in the examples are limited to simple scalar metric tensors.  It will be interesting to see how an anisotropic metric tensor such as one based on the shape map affects mesh movement and quality.

\vspace{10pt}

{\textbf{Acknowledgement.}} We would like to thank Dr. Lei Wang at the University of Wisconsin-Madison for providing us her code to generate initial icosahedral meshes for Example \ref{ellipsoid}.

\begin{appendices}
\section{Derivation of derivatives of the meshing function with respect to the physical coordinates}
\label{SEC:discrederiv}
Recall from Section~\ref{discretization} that $G_K = G\left(\J_K, r_K\right)$ and our objective is to compute
the derivatives
\[
\frac{\partial G_K}{\partial [\V x_1^K,\V  x_2^K, \dots ,\V  x_d^K]} .
\]
Let $t$ be an entry of $[\V x_1^K, \V x_2^K, \dots ,\V x_d^K]$. Using the chain rule we have
\[
\frac{\partial G_K}{\partial t} 
= \text{tr} \left (\frac{\partial G_K}{\partial E_K}  \frac{\partial E_K}{\partial t} \right )
+ \text{tr} \left (\frac{\partial G_K}{\partial \mathbb{M}_K}  \frac{\partial \mathbb{M}_K}{\partial t} \right ) .
\]
Denote
\beq
\label{term1}
\frac{\partial G_K}{\partial t} (I) = \text{tr} \left (\frac{\partial G_K}{\partial E_K}  \frac{\partial E_K}{\partial t} \right ) ,
\eeq
\beq 
\label{term2}
\frac{\partial G_K}{\partial t} (II)
= \text{tr} \left (\frac{\partial G_K}{\partial \mathbb{M}_K}  \frac{\partial \mathbb{M}_K}{\partial t} \right )  .
\eeq
To begin, consider (\ref{term1}).  When $t$ is an entry of $[\V x_2^K, \dots , \V x_d^K]$, recalling that
$E_K = [\V x_2^K-\V x_1^K, \dots , \V x_d^K-\V x_1^K]$, we have
\[
\frac{\partial G_K}{\partial t} (I) = \text{tr} \left (\frac{\partial G_K}{\partial E_K} 
\frac{\partial [\V x_2^K, \dots,\V x_d^K]}{\partial t} \right ) ,
\]
which implies
\[
\frac{\partial G_K}{\partial [\V x_2^K, \dots , \V x_d^K]} (I) = \frac{\partial G_K}{\partial E_K} .
\]
Moreover, for $t = \left(\V x_1^K\right)^{(1)}$ (the first component of $\V x_1^K$), we have
\[
\frac{\partial G_K}{\partial \left(\V x_1^K\right)^{(1)}} (I)
 = \text{tr}\left ( \frac{\partial G_K}{\partial E_K} \begin{bmatrix} -1 & -1 & \cdots & -1 \\ 0 & 0 & \cdots & 0 \\
\vdots & \vdots & & \vdots \\ 0 & 0 & \cdots & 0 \end{bmatrix} \right )
= - \displaystyle\sum_{i} \left (\frac{\partial G_K}{\partial E_K} \right )_{i, 1}.
\]
We can have similar expressions for $ \left(\V x_1^K\right)^{(j)}$ for $j = 2, \dots, d$. This gives
\[
\frac{\partial G_K}{\partial \V x_1^K}(I) = - \V e^T \frac{\partial G_K}{\partial E_K} \]
where $\V e^T = [1, \dots, 1]\in \mathbb{R}^{1\times (d-1)}.$
For (\ref{term2}), we assume that $\mathbb{M} = \mathbb{M}(\V x)$ is a piecewise linear function
defined on the current mesh, i.e., $\mathbb{M} = \displaystyle\sum_{j=1}^d \mathbb{M}_{j,K} \phi_j^K$, where
$\phi_j^K$ is the linear basis function associated with the vertex $\V x_j^K$ for all $j = 1, \dots, d$.
Denote the $i$th components of $\V{x}$ and $\V{x}_K$ by $\V x^{(i)}$ and $\V x_K^{(i)}$, respectively.
Then, 
for any entry $t$ of $[\V x_1^K,\V x_2^K, \dots , \V x_d^K]$, we have
\begin{align*}
\frac{\partial G_K}{\partial t} (II)
& = \text{tr} \left (\frac{\partial G_K}{\partial \mathbb{M}_K} \displaystyle\sum_{i=1}^d \frac{\partial \mathbb{M}_K}{\partial \V x^{(i)}} \right )
\frac{\partial \V x_K^{(i)}}{\partial t} \quad
\\
& = \text{tr} \left (\frac{\partial G_K}{\partial \mathbb{M}_K}  \displaystyle\sum_{i=1}^d 
\displaystyle \sum_{j=1}^d \mathbb{M}_{j,K}\frac{\partial \phi_{j,K}}{\partial \V x^{(i)}} \right )
\frac{\partial \V x_K^{(i)}}{\partial t}
\\
& =  \displaystyle\sum_{j=1}^d \text{tr} \left (\frac{\partial G_K}{\partial \mathbb{M}_K}  \mathbb{M}_{j,K} \right )
\frac{\partial \phi_{j,K}}{\partial \V x} \frac{\partial \V x_K}{\partial t} ,
\end{align*}
where we notice that $\frac{\partial \phi_{j,K}}{\partial \V x} $ and $\frac{\partial \V x_K}{\partial t}$
are a row and a column vector, respectively, and thus $\frac{\partial \phi_{j,K}}{\partial \V x} \frac{\partial \V x_K}{\partial t}$
is a dot product.
From this and the identity $\V x_K = (\V x_1^K+\dots+\V x_d^K)/d$, we get
\[
\frac{\partial G_K}{\partial [\V x_2^K, \dots , \V x_d^K]}(II)
= \frac{1}{d} \displaystyle\sum_{j=1}^d \text{tr} \left (\frac{\partial G_K}{\partial \mathbb{M}_K}  \mathbb{M}_{j,K} \right )
\begin{bmatrix} \frac{\partial \phi_{j,K}}{\partial \V x} \\ \vdots \\ \frac{\partial \phi_{j,K}}{\partial \V x} \end{bmatrix}\]
and
\[
\frac{\partial G_K}{\partial \V x_1^K}(II)
= \frac{1}{d} \displaystyle\sum_{j=1}^d \text{tr} \left (\frac{\partial G_K}{\partial \mathbb{M}_K}  \mathbb{M}_{j,K} \right )
\frac{\partial \phi_{j,K}}{\partial \V x} .
\]
Summarizing the above results, we have
\begin{align}
& \frac{\partial G_K}{\partial [\V x_2^K, ..., \V x_d^K]} 
= \frac{\partial G_K}{\partial E_K} 
+\frac{1}{d} \displaystyle\sum_{j=1}^d \text{tr} \left (\frac{\partial G_K}{\partial \mathbb{M}_K}  \mathbb{M}_{j,K} \right )
\begin{bmatrix} \frac{\partial \phi_{j,K}}{\partial \V x} \\ \vdots \\ \frac{\partial \phi_{j,K}}{\partial \V x} \end{bmatrix},
\label{der-1}
\\
& \frac{\partial G_K}{\partial \V x_1^K} 
= - \V e^T \frac{\partial G_K}{\partial E_K} 
+ \frac{1}{d} \displaystyle\sum_{j=1}^d \text{tr} \left (\frac{\partial G_K}{\partial \mathbb{M}_K}  \mathbb{M}_{j,K} \right )
\frac{\partial \phi_{j,K}}{\partial \V x} .
\label{der-2}
\end{align}
Notice that (\ref{der-2}) can be rewritten as
\begin{align}
\frac{\partial G_K}{\partial \V x_1^K} 
= - \sum_{j=2}^d \frac{\partial G_K}{\partial \V x_j^K} 
+ \displaystyle\sum_{j=1}^d \text{tr} \left (\frac{\partial G_K}{\partial \mathbb{M}_K}  \mathbb{M}_{j,K} \right )
\frac{\partial \phi_{j,K}}{\partial \V x} ,
\label{der-2-1}
\end{align}
which gives (\ref{vel1}).

Next, we establish the relations between
\[
\frac{\partial G_K}{\partial E_K}, \quad 
\frac{\partial G_K}{\partial \mathbb{M}_K} \quad \quad \text{ and }\quad \quad
\frac{\partial G_K}{\partial \J},\quad \frac{\partial G_K}{\partial r} .
\]
First recall that $F_K' = E_K \hat{E}^{-1}$, thus
\beq
\label{J}
\J = \left(\left(F_K'\right)^T\mathbb{M}_KF_K'\right)^{-1} = \hat{E} \left(E_K^T\mathbb{M}_K E_K\right)^{-1} \hat{E}^T .
\eeq
Let $E_K = E_K(t)$. Then we have
\begin{align}
\frac{\partial G_K}{\partial t} & = \text{tr} \left ( \frac{\partial G_K}{\partial \J} \frac{\partial \left(\left(F_K'\right)^T\mathbb{M}_KF_K'\right)^{-1}}{\partial t} \right )
+ \frac{\partial G_K}{\partial r} \frac{\partial \det\left(\left(F_K'\right)^T\mathbb{M}_KF_K'\right)^{-1}}{\partial t}\notag
\\
& \label{dGdt}= \text{tr} \left ( \frac{\partial G_K}{\partial \J} \hat{E} \frac{\partial \left(E_K^T\mathbb{M}_K E_K\right)^{-1} }
{\partial t} \hat{E}^T \right )
+\det(\hat{E})^{2} \frac{\partial G_K}{\partial r} \frac{\partial \det\left(E_K^T \mathbb{M}_K E_K\right)^{-1}}{\partial t}.
\end{align}
Consider the first term of (\ref{dGdt}). Using the properties of matrix derivatives (\ref{deriv}) we get
\begin{align*}
& \text{tr} \left ( \frac{\partial G_K}{\partial \J} \hat{E} \frac{\partial \left(E_K^T\mathbb{M}_K E_K\right)^{-1} }
{\partial t} \hat{E}^T \right )\\
 & =  -\text{tr} \left ( \frac{\partial G_K}{\partial \J} \hat{E}\left(E_K^T\mathbb{M}_K E_K\right)^{-1} \frac{\partial \left(E_K^T\mathbb{M}_K E_K\right) }
{\partial t}  \left(E_K^T\mathbb{M}_K E_K\right)^{-1}\hat{E}^T\right )\\
& = -\text{tr} \left ( \frac{\partial G_K}{\partial \J} \hat{E}\left(E_K^T\mathbb{M}_K E_K\right)^{-1} \left( \frac{\partial E_K^T}{\partial t}\mathbb{M}_K E_K+ E_K^T\mathbb{M}_K\frac{\partial E_K}{\partial t}\right) \left(E_K^T\mathbb{M}_K E_K\right)^{-1}\hat{E}^T\right ).\\
\end{align*}
Since $ \frac{\partial G_K}{\partial \J}$, $\mathbb{M}_K$, and $\left(E_K^T\mathbb{M}_K E_K\right)^{-1} $ are all symmetric, it follows that
\begin{align*}
 & \text{tr} \left ( \frac{\partial G_K}{\partial \J} \hat{E} \frac{\partial \left(E_K^T\mathbb{M}_K E_K\right)^{-1} }{\partial t} \hat{E}^T \right ) \\
 & \quad = -2\text{tr} \left (  \left(E_K^T\mathbb{M}_K E_K\right)^{-1} \hat{E}^T\frac{\partial G_K}{\partial \J} \hat{E} \left(E_K^T\mathbb{M}_K E_K\right)^{-1}  E_K^T  \mathbb{M}_K \frac{\partial E_K}{\partial t}\right ).
\end{align*}

The second term of (\ref{dGdt}) is then
\begin{align*}
&\det(\hat{E})^{2}\frac{\partial G_K}{\partial r}  \frac{\partial \det\left(E_K^T \mathbb{M}_K E_K\right)^{-1}}{\partial t}\\
 & =\frac{\det(\hat{E})^{2}}{\det\left(E_K^T \mathbb{M}_K E_K\right)}~\frac{\partial G_K}{\partial r}~\tr\left(\left(E_K^T \mathbb{M}_K E_K\right)\frac{\partial \left(E_K^T \mathbb{M}_K E_K\right)^{-1}}{\partial t}\right)\\
& = -\frac{\det(\hat{E})^{2}}{\det\left(E_K^T \mathbb{M}_K E_K\right)}~\frac{\partial G_K}{\partial r}~\tr\left(\frac{\partial \left(E_K^T \mathbb{M}_K E_K\right)}{\partial t} \left(E_K^T \mathbb{M}_K E_K\right)^{-1}\right)\\
& = -\frac{\det(\hat{E})^{2}}{\det\left(E_K^T \mathbb{M}_K E_K\right)}~\frac{\partial G_K}{\partial r}~\tr\left(\left( \frac{\partial E_K^T}{\partial t}\mathbb{M}_K E_K+ E_K^T\mathbb{M}_K\frac{\partial E_K}{\partial t}\right)\left(E_K^T \mathbb{M}_K E_K\right)^{-1}\right)\\
& = -2\frac{\det(\hat{E})^{2}}{\det\left(E_K^T \mathbb{M}_K E_K\right)}~\frac{\partial G_K}{\partial r}~\tr\left(\left(E_K^T \mathbb{M}_K E_K\right)^{-1} E_K^T\mathbb{M}_K\frac{\partial E_K}{\partial t}\right).
\end{align*}
Therefore
\begin{align}
\frac{\partial G_K}{\partial E_K} 
& = -2 \left(E_K^T\mathbb{M}_K E_K\right)^{-1} \hat{E}^T
\frac{\partial G_K}{\partial \J} \hat{E} (E_K^T\mathbb{M}_K E_K)^{-1} E_K^T\mathbb{M}_K
\notag \\
& \quad  - 2 \frac{\det(\hat{E})^{2}}{\det\left(E_K^T \mathbb{M}_K E_K\right)}\frac{\partial G_K}{\partial r} \left(E_K^T \mathbb{M}_K E_K\right)^{-1} E_K^T\mathbb{M}_K .
\label{der-3}
\end{align}
Combining this with (\ref{der-2}) we obtain (\ref{vel2}).

The identity (\ref{der-4}) can be obtained similarly.

Finally, we derive the relations between
\[
\frac{\partial \phi_{j,K}}{\partial \V x}, \quad j = 1, ..., d \quad \quad \text{ and } \quad E_K .
\]
First, note that the basis functions satisfy
$$
\displaystyle\sum_{i=1}^d\phi_{i,K}=1\hspace{1cm}\text{and}\hspace{1cm}\displaystyle\sum_{i=1}^d\V x_i^K\phi_{i,K}=\V x.
$$
Eliminating the $\V x_{1}^K$ yields
$$
\V x-\V x_{1}^K=\displaystyle\sum_{i=2}^d (\V x_{i}^K-\V x_{1}^K)\phi_{i,K}. 
$$
Then differentiating with respect to $\V x^{(k)}$ gives
$$
\V e_k=\dfrac{\partial (\V x-\V x_{1}^K)}{\partial \V x^{(k)}}
=\dfrac{\partial}{\partial \V x^{(k)}}\left(\displaystyle\sum_{i=2}^d (\V x_{i}^K-\V x_{1}^K)\phi_{i,K}\right)
=\displaystyle\sum_{i=2}^d(\V x_{i}^K-\V x_{1}^K)\dfrac{\partial \phi_{i,K}}{\partial \V x^{(k)}} ,
$$
where $\V e_k$ is the $k^{th}$ unit vector in $\mathbb{R}^d$. Hence we have
\[
I = E_K\begin{bmatrix}\frac{\partial \phi_{2,K}}{\partial \V x}\\ \vdots \\ \frac{\partial \phi_{d,K}}{\partial \V x}\end{bmatrix} ,
\]
which gives
\[
E_K^TE_K
\begin{bmatrix}\frac{\partial \phi_{2,K}}{\partial \V x}\\ \vdots \\ \frac{\partial \phi_{d,K}}{\partial \V x}\end{bmatrix}
=E_K^T
\]
and thus (\ref{der-5}).

\end{appendices}

\end{document}